\documentclass[12pt,Letterpaper,oneside,leqno]{amsart}
\usepackage{amssymb, mathdots, url,graphicx,amscd, amsthm,fullpage}
\numberwithin{equation}{section}
\usepackage[margin=1in]{geometry}
\usepackage{graphicx}
\usepackage{mathdots}
\usepackage{paralist}
\usepackage{enumerate}
\usepackage{tikz-cd}
\usepackage{xcolor}
\usepackage[hyperfootnotes=false, colorlinks, linkcolor={blue}, citecolor={magenta}, filecolor={blue}, urlcolor={blue}, plainpages=false, pdfpagelabels]{hyperref}
\usepackage[style=alphabetic,backend=biber, maxbibnames=99]{biblatex}
\newtheorem{theorem}{Theorem}[section]
\newtheorem{conjecture}[theorem]{Conjecture}
\newtheorem{lemma}[theorem]{Lemma}

\newtheorem{proposition}[theorem]{Proposition}
\newtheorem{prop}[theorem]{Proposition}
\newtheorem{corollary}[theorem]{Corollary}

\theoremstyle{definition}
\newtheorem{definition}[theorem]{Definition}

\newtheorem{remark}[theorem]{Remark}

\def\Lie{{\mathrm {Lie}}}

\DeclareMathOperator{\tr}{tr}

\DeclareMathOperator{\SO}{SO}
\DeclareMathOperator{\Spin}{Spin}

\DeclareMathOperator{\Sp}{Sp}

\DeclareMathOperator{\SU}{SU}

\DeclareMathOperator{\SL}{SL}
\DeclareMathOperator{\GL}{GL}

\DeclareMathOperator{\Span}{Span}

\def\g{{\mathfrak g}}
\def\h{{\mathfrak h}}
\def\k{{\mathfrak k}}
\def\p{{\mathfrak p}}

\def\so{{\mathfrak {so}}}
\def\sl{{\mathfrak {sl}}}

\def\G{{\mathbf G}}
\def\C{{\mathbf C}}
\def\R{{\mathbf R}}
\def\Z{{\mathbf Z}}
\def\Q{{\mathbf Q}}
\def\A{{\mathbf A}}

\def\Q{{\mathbf Q}}

\def\V{{\mathbf V}}

\def\linspan{\textnormal{-span}}
\begin{document}

\title{On the Vanishing and Cuspidality of $D_4$ Modular Forms}

\author{Finn McGlade}
\address{Department of Mathematics, The University of Oklahoma, Norman, OK 73072, USA}
\email{finn.mcglade@ou.edu}

\date{\today}


\begin{abstract} We develop vanishing and cuspidality criteria for quaternionic modular forms on $G=\Spin(4,4)$ using a theory of scalar Fourier coefficients.  By analyzing a Fourier-Jacobi expansion for these forms,  we prove that a level one quaternionic modular form on $G$ vanishes if and only if its primitive Fourier coefficients are zero.  Using this criterion, we characterize Pollack's quaternionic Saito-Kurokawa subspace by imposing a system of linear relation among certain primitive Fourier coefficients. This characterization strengthens earlier work of the author with Johnson-Leung, Negrini, Pollack, and Roy.  We also study quaternionic modular forms in the more general setting of a group $G_J$ associated to a cubic norm structure $J$. Here we establish a new relationship between the degenerate Fourier coefficients of quaternionic modular forms, and the Fourier coefficients of the holomorphic modular forms associated to their constant terms.  As a consequence,  we prove that in weights $\ell\geq 5$,  a level one quaternionic modular form on $G$ is cuspidal if and only if its non-degenerate Fourier coefficients satisfy a polynomial growth condition.
\end{abstract}

\maketitle

\goodbreak 

\tableofcontents

\goodbreak
\section{Introduction}
\setcounter{page}{1}
Let $G$ be the split spin group of rank $4$ over the rational numbers.   Given an automorphic function $\varphi$ on $G$,   one can ask the following two fundamental questions; (i) does  $\varphi$ vanish identically, and (ii) is $\varphi$ cuspidal? The purpose of this paper is to apply a theory of scalar Fourier coefficients to address these questions in the case when $\varphi$ is a quaternionic modular form on $G$.   Before we can formulate our results, we need to review some preliminary ideas.  
\subsection{Quaternionic Modular Forms on $G$} 
In \cite{weissmanD4},  Weissman studies $D_4$ modular forms, which are certain automorphic functions on $G$.  Weissman's work is influenced by that of Gan-Gross-Savin \cite{ganGrossSavin}, who studied an analogous class of modular forms on $G_2$.  Pollack \cite{pollackQDS} later studied a generalization of this class of automorphic forms in the setting of a reductive group $G_J$ associated to a cubic norm structure $J$.   In Pollack's conventions,  a weight $\ell\in \Z_{>0}$ quaternionic modular form on $G_J$ is a vector-valued automorphic function that is annihilated by the Schmid differential \cite{Schmid89} associated to a particular $G_J(\R)$-representation.  One notable feature of a $D_4$ modular form $\varphi$ is the set of  scalar Fourier coefficients,
$$
\{\Lambda_{\varphi}[B]\in \C\colon B=[T_1,T_2]\in \mathrm{M}_2(\Z)^{\oplus 2}\backslash\{[0,0]\}\},
$$
which are indexed by non-zero pairs $B$ consisting of $2\times 2$ integral matrices; $T_1$ and $T_2$.
The coefficients $\Lambda_{\varphi}[B]$ arise from taking Fourier coefficients of $\varphi$ along a Heisenberg group $N_P$, which is the unipotent radical of a maximal parabolic subgroup $P$ in $G$. The existence of $\Lambda_{\varphi}[B]$ as a scalar is non-trivial, and relies on results regarding the generalized Whittaker vectors of quaternionic $G(\R)$-representations (\cite[Theorem 16]{wallach} and \cite[Theorem 1.2]{pollackQDS}).  
Another important feature of $\varphi$ is the structure of the constant term $\varphi_{N_P}$.  In \cite[Theorem 11.1.1]{pollackQDS}, it is shown that $\varphi_{N_P}$ gives rise to a holomorphic modular form $\Phi$ on a Levi factor $M_P$ of $P$.  More specifically,  the derived subgroup $M_P^{\mathrm{der}}$ is isomorphic to $\SL_2^3$, and $\Phi$ is the automorphic form associated to a holomorphic modular form on three copies of the upper half plane.
\subsection{Primitivity in the Fourier Expansion of $D_4$ Modular Forms}
Given an automorphic form $\varphi$, it is generally important to develop criteria to assess whether $\varphi$ is identically zero.  When $\varphi$ has scalar Fourier coefficients,  these may be applied to study the vanishing of $\varphi$ in unique ways. For example,  let $\varphi$ be the automorphic function on $\Sp_4(\A_\Q)$ associated to a Siegel modular form $F_{\varphi}$ of level one, and genus $2$.   Then $F_{\varphi}$ admits a Fourier expansion
$$
F_{\varphi}(Z)=\sum_{T\geq 0}A_{\varphi}[T]\exp(2\pi i \mathrm{tr}(TZ)).
$$Here $Z$ is an element in the Siegel upper half space of genus $2$,  and $T=[a,b,c]$ runs over the set of positive semi-definite, integral, binary quadratic forms. The complex scalars $A_{\varphi}[T]$ are the  Fourier coefficients of $\varphi$.  \\
\indent Say that $A_{\varphi}[T]$ is primitive if $\gcd(a,b,c)=1$.  
Zagier \cite[pg.  387]{ZaFrench} proves that $\varphi$ vanishes identically if and only if each of the primitive coefficients $A_{\varphi}[T]$ is $0$.  The interplay between the vanishing of the coefficients $A_{\varphi}[T]$, and the vanishing properties of $\varphi$ has since developed into an area of significant interest.  For example,  Saha \cite{MR3004586} proves a refinement of Zagier's Theorem,  which in the case of cuspidal $\varphi$, characterizes the vanishing of $\varphi$ using a sparser subset of \textit{fundamental} Fourier coefficients.  The result of (loc. cit.) is  generalized to vector valued, possibly non-cuspidal,  Siegel modular forms of arbitrary genus in \cite{MR4515287}. We refer the reader to the introductions in \cite{MR3908796,MR3004586} for a survey of similar vanishing results.
Our first main theorem is an analogue of Zagier's Theorem for $D_4$ modular forms.  Below,  we say that $B\in \mathrm{M}_2(\Z)^{\oplus 2}$ is primitive if 
$
\Q\linspan\{B\}\cap \mathrm{M}_2(\Z)^{\oplus 2}=\Z\linspan\{B\}.
$
\begin{theorem}
\label{thm-primitivity-intro-1}
Let $\varphi$ be quaternionic modular forms on $G$ of weight $\ell>0$ and level one such that $\Lambda_{\varphi}[B]=0$ for all primitive $B\in \mathrm{M}_2(\Z)^{\oplus 2}$.  Then $\varphi=0$.  
\end{theorem}
In the setting of quaternionic modular forms on $G_2$,  \cite[Theorem 16.12]{ganGrossSavin} gives an analogue of Theorem~\ref{thm-primitivity-intro-1},  showing that a level one cuspidal quaternionic Hecke eigenform on $G_2$ is zero if and only if its primitive Fourier coefficients are zero. 
Our proof of Theorem \ref{thm-primitivity-intro-1}, which does not generalize to modular forms on $G_2$,  utilizes certain \textit{Fourier-Jacobi coefficients} for $D_4$ modular forms.   To define these coefficients, let $R=M_RN_R$ be a maximal parabolic subgroup of $G$ corresponding to an outer vertex in the Dynkin diagram of $G$, and write $N_R$ for the abelian unipotent radical of $R$.  Given a character $\chi\colon N_R(\Q)\backslash N_R(\A) \to \C$,  the $\chi$ Fourier-Jacobi coefficient of $\varphi$ is 
$$
\mathcal{F}(\varphi;\chi)(g)= \int_{N_R(\Q)\backslash N_R(\A)}\varphi(ng)\chi^{-1}(n)\, dn. 
$$ 
The coefficient $\mathcal{F}(\varphi;\chi)$ is studied for a specific character $\chi$ in \cite[Corollary 7.6]{JMNPR24}.  When $\varphi$ is a cuspidal quaternionic modular on certain exceptional groups,  an analogue of $\mathcal{F}(\varphi;\chi)$ is analyzed in \cite[Theorem 7.12]{pollack2024}.  The results of (loc. cite. ) apply only to the case when the character is non-degenerate. Here we say that a character of a unipotent radical is non-degenerate if it has an open orbit under the action of a Levi subgroup. In this paper, we study $\mathcal{F}(\varphi;\chi)$ for a general non-trivial character $\chi$, without assuming $\varphi$ is cuspidal. We apply the coefficients $\mathcal{F}(\varphi;\chi)$ to adapt Zagier's original proof of his theorem into the setting of $D_4$ modular form.  The implementation of this strategy on $G$ presents several novelties.   
\\
\indent Firstly, the analytic properties of the non-degenerate coefficients $\mathcal{F}(\varphi;\chi)$ are more difficult to analyze in the case of $D_4$-modular forms.  Indeed, the classical Fourier-Jacobi coefficients of a Siegel modular form $F$ easily inherent their holomorphy properties from those of $F$.  Similarly, the non-degenerate Fourier-Jacobi coefficients $\mathcal{F}(\varphi;\chi)$  obey certain holomorphy properties, however,  these properties are harder to prove since $\varphi$ itself is not holomorphic (see Proposition \ref{prop-existence-holomorphic-FJ-coefficient}).  In particular,  our analysis of the non-degenerate coefficients $\mathcal{F}(\varphi;\chi)$ depends on the explicit formula for the generalized Whittaker function of quaternionic $G(\R)$-representation proven in \cite[Theorem 1.2]{pollackQDS}.  \\
\indent Secondly, the unipotent radical $N_R$ supports non-trivial degenerate characters, and the degenerate coefficients $\mathcal{F}(\varphi;\chi)$ must be analyzed separately from the non-degenerate coefficients.  This makes the proof of Theorem \ref{thm-primitivity-intro-1} more challenging in the case when $\varphi$ is non-cuspidal.  More specifically, to prove Theorem \ref{thm-primitivity-intro-1} for non-cuspidal $\varphi$, we establish a new relationship between the degenerate Fourier coefficients $\Lambda_{\varphi}[B]$ of $\varphi$, and the Fourier coefficients of the holomorphic modular form $\Phi$ (see Proposition~\ref{cor-FCs-constant-degenerate-Heisenberg-FCS-D4}).  In Subsection~\ref{subsec-The Degenerate Coefficients of a Quaternionic Modular Form}, we show that this relationship  is true in the generality of  quaternionic modular forms on the group $G_J$ associated to a cubic norm structure $J$.   
\subsection{An Application to The Quaternionic Maass Spezialschar} Let $\SO_8$ denote the split special orthogonal group of rank $4$ over $\Q$. Then $\SO_8$ supports a theory of quaternionic modular forms which is closely related to the theory of modular forms on $G$.  In fact, if $\varphi$ is a quaternionic modular form on $\SO_8$ of level $1$, then $\varphi$ is uniquely determined by its pull back to $G$ (see Lemma~\ref{lemma-unique-determination}),  which is a quaternionic modular form on $G$.  As such,  a level one quaternionic modular form $\varphi$ on $\SO_8$ has an associated set of Fourier coefficients $\{\Lambda_{\varphi}[B]\colon B\in \mathrm{M}_2(\Z)^{\oplus 2}\backslash \{[0,0]\}\}$, for which the statement of Theorem \ref{thm-primitivity-intro-1} holds true. \\ 
\indent In \cite{pollackCuspidal},  Pollack studies the theta correspondence arising from the dual pair $\Sp_4\times \mathrm{O}_8$, and describes the lifting of a level one holomorphic modular form $F$ on $\Sp_4$, to a level one quaternionic modular form $\theta^{\ast}(F)$ on $\SO_8$.  The Fourier coefficients of the \textit{quaternionic Saito-Kurokawa lift} $\theta^{\ast}(F)$ are given as linear combinations of the Fourier coefficients of $F$ \cite[Theorem 4.1.1]{pollackCuspidal}. In \cite{JMNPR24},  the authors characterize the quaternionic modular forms $\theta^{\ast}(F)$ using a system of linear equations.  This result is analogous to the characterization of the classical Saito-Kurokawa subspace via the \textit{Maass Relations}; see for example \cite[\S 6, (9)]{EichZag}.\\
\indent In Section \ref{sec-The Quaternionic Maass Spezialschar}, we refine \cite[Theorem 1.3]{JMNPR24}, to characterize the modular forms $\theta^{\ast}(F)$ in the style of  \cite[Theorem 1 (iii)]{ZaFrench}, which states that a level one,  genus $2$ Siegel modular form $F_{\varphi}$ is a Saito-Kurokawa lift if and only if each primitive Fourier coefficient $A_{\varphi}[T]$ only depends on the discriminant $\mathrm{disc}(T)$. Our characterization uses a construction of \cite{MR2051392}, which associates $B=[T_1,T_2]$ with the binary quadratic form in variables $x$ and $y$ given by,
$$
T(B)=\det(xT_1-yT_2).
$$  
We also require a more refined notion of primitivity (see \cite[Definition 5.4]{JMNPR24}).
\begin{definition}
Say that $[T_1,T_2]\in \mathrm{M}_2(\Z)^{\oplus 2}$ is \textit{strongly primitive} or  \textit{slice primitive} if 
$$
\Q\linspan\{T_1,T_2\}\cap \mathrm{M}_2(\Z)=\Z\linspan\{T_1,T_2\}. 
$$
\end{definition}
In Corollary~\ref{slice-primitivity-corollary},  we show a cuspidal $D_4$ modular form $\varphi$ vanishes if and only if the slice primitive Fourier coefficients of $\varphi$ are zero.  Combining \cite[Theorem 1.3]{JMNPR24} and Corollary~\ref{slice-primitivity-corollary} yields the following application.
\begin{theorem} 
\label{Slice-Primitive-Lifts}
Suppose $\ell\geq 16$ is even and let $\varphi$ be a level one, cuspidal, quaternionic modular form on $\SO_8$ of weight $\ell$. The following are equivalent:
\begin{compactenum}[(i)]
    \item There exists a level one, weight $\ell$, Siegel modular form $F$ on $\Sp_4$ such that $\varphi=\theta^*(F)$. 
    \item If $B,B'\in \mathrm{M}_2(\Z)^{\oplus 2}$ are slice primitive and satisfy $T(B)=T(B')$,  then $\Lambda_{\varphi}[B]=\Lambda_{\varphi}[B']$. 
\end{compactenum}
\end{theorem}
\indent The problem of characterizing theta lifts is well studied within the general theory of automorphic forms. For example, in the case when $\varphi$ is a certain type of quaternionic Hecke eigenform on $G$, condition (a) above is equivalent to the condition that $\varphi$ admits a non-zero period along an embedded copy of $\SO_6$ in $\SO_8$ \cite[Corollary 9.8]{JMNPR24}. As such, Theorem \ref{Slice-Primitive-Lifts} posits an indirect connection between the slice primitive Fourier coefficients $\Lambda_{\varphi}[B]$, and the study of automorphic period integrals. 
\subsection{Characterizing Quaternionic Cusp Forms via a Hecke Bound} 
A central question in the theory of automorphic forms is the problem of determining whether a given automorphic function is cuspidal.  This question interacts in fascinating ways with theories of scalar Fourier coefficients. For example, assume $F_{\varphi}$ is a level one, genus $2$ Siegel modular form of weight $\ell\geq 4$.  Then,  Kohnen-Martin \cite[Theorem 2.1]{KoMa14} prove ${\varphi}$ is cuspidal if and only if $A_\varphi[T]\ll_F|\mathrm{disc}(T)|^{\ell/2}$ for all positive definite forms $T$.  The aforementioned theorem has an antecedent in \cite{MR2683571}, where it shown that certain elliptic cusp forms are similarly characterized by the sizes of their Fourier coefficients.  These results relate two fundamental properties of modular forms; the growth of their Fourier coefficients, and cuspidality. As such, they are of broad interest. For example, in \cite{BoSo14}, the authors reprove  \cite[Theorem 2.1]{KoMa14} by a different method, and extend the result to Siegel modular forms of higher genus.  We recommend the introduction to \cite{das2025fouriercoefficientscuspidalitymodular} for an overview of related work. In Section \ref{sec-The Hecke Bound Characterization of Cusp Form on $G$}, we establish an analogue of \cite[Theorem 2.1]{KoMa14}.
\begin{theorem}
\label{thm-intro-Hecke-IFF-Cuspidal-Quaternionic}
Suppose $\varphi$ is a level one, quaternionic modular form on $G$ of weight $\ell \geq 5$. Then $\varphi$ is cuspidal if and only if, for all $B\in \mathrm{M}_2(\Z)^{\oplus 2}$ satisfying $\mathrm{disc}(T(B))<0$, 
\begin{equation}
\label{eqn-Hecke-Bound-intro-quaternionic}
\Lambda_{\varphi}[B]\ll_{\varphi} |\mathrm{disc}(T(B))|^{\frac{\ell+1}{2}}. 
\end{equation}
\end{theorem}
To the author's knowledge,  Theorem \ref{thm-intro-Hecke-IFF-Cuspidal-Quaternionic} is the first instance of a result characterizing the cuspidality of quaternionic modular forms, on any group, via the growth of their Fourier coefficients.  To prove Theorem~\ref{thm-intro-Hecke-IFF-Cuspidal-Quaternionic},  the Fourier-Jacobi coefficients $\mathcal{F}(\varphi;y)$ are of indispensable use,  as are the ideas underlying the proof of \cite[Theorem 2.1]{KoMa14}.  The proof of Theorem \ref{thm-intro-Hecke-IFF-Cuspidal-Quaternionic} also requires a more general cuspidality criterion, Theorem \ref{thm-into-cuspidality-criteria}.  \\
\indent To set up Theorem \ref{thm-into-cuspidality-criteria},  let $C$ denote a composition algebra over $\Q$, and write $H_3(C)$ for the cubic norm structure consisting of $3\times 3$ Hermitian matrices with entries in $C$.  Let $J=\mathbb{G}_a^3$ or $J=H_3(C)$, and write $G_J$ for the $\Q$-rational algebraic group associated to $J$ in \cite{pollackQDS}. So, $G_J$ is an adjoint group of real rank $4$ and type $D_4$, $F_4$, $E_6$, $E_7$, or $E_8$.  Let $P_J$ be the Heisenberg parabolic subgroup of $G_J$, and write $N_{J}$ for the unipotent radical of $P_J$. 
\begin{theorem}
\label{thm-into-cuspidality-criteria}
Suppose $\varphi$ is a quaternionic modular form on $G_J$. Write $\varphi_{N_{J}}$ for the constant term of $\varphi$ along the unipotent radical $N_{J}$.  If $\varphi_{N_{J}}\equiv 0$, then $\varphi$ is cuspidal. 
\end{theorem}
For the proof of Theorem~\ref{thm-into-cuspidality-criteria}, we assume $J$ is any cubic norm structure with a positive definite trace pairing. As a consequence of \cite[Proposition 11.1.1.]{pollackQDS},  if $\varphi$ is a quaternionic modular form on $G_J$,  then $\varphi_{N_{J}}$ gives rise to a holomorphic modular form $\Phi$ on a Levi subgroup of $P_J$.  In Proposition \ref{prop-FCs-of-Heisenberg-constant-term}, we show that the Fourier coefficients of $\Phi$ are given as finite sums of the degenerate Fourier coefficients of $\varphi$.  Conversely,  if $J$ is a cubic norm structure of the type appearing in Theorem \ref{thm-into-cuspidality-criteria},  then Lemma \ref{lemma-auto-convergence-13.?} and Lemma \ref{orbit-lemma} imply that every degenerate Fourier coefficient of $\varphi$ is a finite sum of Fourier coefficients of $\Phi$.  Hence, in the setting of Theorem \ref{thm-into-cuspidality-criteria}, the vanishing of $\varphi_{N_J}$ implies that all degenerate Fourier coefficients of $\varphi$ are zero, which can used to deduce the cuspidality of $\varphi$ (see Subsection~\ref{subsec-The General Cuspidality Criterion}). 
\begin{remark} Our proof of Theorem \ref{thm-intro-Hecke-IFF-Cuspidal-Quaternionic} does not directly generalize to quaternionic modular forms in other Dynkin types.  This is because we use a unique features of $D_4$ modular form. Namely,  in Proposition \ref{prop-non-vanishing-rank-3}, we show that if $\varphi$ is a level one, non-cuspidal,  $D_4$ modular form,  then $\varphi$ has a non-zero, primitive, rank $3$ Fourier coefficient.   We refer the reader to  \cite[Definition 4.3.2]{pollackLL} for the definition of rank in this setting.  One benefit of Proposition \ref{prop-non-vanishing-rank-3} is that it implies a stronger version of Theorem \ref{thm-intro-Hecke-IFF-Cuspidal-Quaternionic}. More specifically, in Theorem \ref{main-theorem-cuspidal-IF-Hecke} we show that a level one $D_4$ modular form $\varphi$ is cuspidal if,   for all primitive $B\in \mathrm{M}_2(\Z)^{\oplus 2}$ satisfying $\mathrm{disc}(T(B))<0$,  $\Lambda_{\varphi}[B]\ll_{\varphi} |\mathrm{disc}(T(B))|^{\frac{\ell+1}{2}}$.  In other words,  the cuspidality of $\varphi$ is characterized by the growth of the primitive, non-degenerate Fourier coefficients $\Lambda_{\varphi}[B]$.
\end{remark}
\subsection{The structure of paper} In Section~\ref{sec:Notation} we fix notation regarding the Lie algebra of $G$, certain parabolic subgroups,  and a maximal compact subgroup $K_{\infty}\leq G(\R)$.  In Section~\ref{Sec-Modular-Forms-on-Spin(8)-and-Sp4}, we define holomorphic modular forms on various groups,  review the Fourier expansions of these modular forms, and prove several vanishing and cuspidality criteria for holomorphic modular forms.  In Section~\ref{Sec-Prelimns-on-Quat-Modular-Forms}, we define quaternionic modular forms on \(G\),  their Fourier coefficients, and prove that the Fourier coefficients of cuspidal quaternionic modular forms on $G$ satisfy the bound \eqref{eqn-Hecke-Bound-intro-quaternionic}. The orthogonal Fourier-Jacobi expansion for quaternionic modular forms on \(G\) is developed in Section~\ref{chap:THE FOURIER-JACOBI EXPANSION OF MODULAR FORMS ON $G$}, where the vanishing criterion of Theorem~\ref{thm-primitivity-intro-1} is proven.  In Section~\ref{sec-The Quaternionic Maass Spezialschar}, we define the quaternionic Saito-Kurokawa lifts $\theta^{\ast}(F)$, and prove Theorem~\ref{Slice-Primitive-Lifts}.  In Section~\ref{sec-The Hecke Bound Characterization of Cusp Form on $G$}, we analyze the degenerate Fourier coefficients of quaternionic modular forms on $G_J$,  proving Theorem \ref{thm-into-cuspidality-criteria} and Theorem~\ref{thm-intro-Hecke-IFF-Cuspidal-Quaternionic}  .

\section{Acknowledgments}
The author thanks Aaron Pollack for valuable discussions and guidance related to this work. 
Part of this work was initiated during the 2022 Research Innovations and Diverse Collaborations workshop at the University of Oregon, and further developed during a SQUaREs workshop at the American Institute of Mathematics. We thank the organizers of these programs for providing a stimulating research environment.
The author is also grateful to Benedict Gross,  Bryan Hu, Jennifer Johnson-Leung,  Isabella Negrini,  Ameya Pitale, and Manami Roy for helpful conversations\\
\section{Algebraic Preliminaries}
\label{sec:Notation}
\subsection{Standard Notation}
The symbols $\Z$, $\Q$, $\C$, and $\R$ denote the rings of integers, rational numbers, complex numbers, and real numbers respectively.  Given a prime number $p$,  $\Z_p$ and $\Q_p$ are the rings of $p$-adic integers and $p$-adic numbers respectively. 
Let $\A$ denote the adele ring of $\Q$, and $\psi\colon \A/\Q\to \C^{\times}$ be a fixed non-trivial additive character of $\A/\Q$.
\subsection{The Quadratic Space $V$}
\label{subsec:The underlying quadratic space $V$}
Throughout, $(V,q)$ denotes a non-degenerate split quadratic space over $\Q$ of dimension $8$.  Write $(x,y) = q(x+y)-q(x)-q(y)$ for the bilinear form associated to $q$ and let
$$
\{b_1, b_2, b_3, b_4, b_{-4}, b_{-3}, b_{-2}, b_{-1}\}
$$
be a basis of $V$ consisting of isotropic vector $b_{\pm i}$, $i=1,\ldots 4$.  We choose this basis so that $(b_{\pm i}, b_{\pm j})=0$ and $(b_{\pm i}, b_{\mp j})=\delta_{i,j}$ for all $i,j=1,\ldots, 4$.  Here $\delta_{i,j}$ denotes the Kronecker delta symbol. Next, we let $U\subseteq V$ be the isotropic two-plane spanned by $\{b_1, b_2\}$, and $U^{\vee}$ denote the isotropic two plane spanned by $\{b_{-1}, b_{-2}\}$.  Write $V_{2,2}$ to denote the orthogonal complement of $U+ U^{\vee}$ in $V$. So $(V_{2,2},q)$ is a split quadratic space of signature $(2,2)$ and 
\begin{equation}
\label{decomposition of V stabilized by MP}
V=U\oplus V_{2,2}\oplus U^{\vee}.
\end{equation} 
Later, it will be convenient to identify $V_{2,2}$ with the space of $2\times 2$ matrices $\mathrm{M}_2$ over $\Q$ with $q\rvert_{V_{2,2}}$ equal to the determinant of $\mathrm{M}_2$.  We form this identification $V_{2,2}\xrightarrow{\sim} \mathrm{M}_2$ via the map
\begin{equation}
\label{identification-of-V_22}
b_3\mapsto \begin{pmatrix} 1 &0 \\ 0 &0 \end{pmatrix}, \quad b_{-3}\mapsto \begin{pmatrix} 0 &0 \\ 0&1 \end{pmatrix}, \quad b_4\mapsto \begin{pmatrix} 0 &0 \\ -1 &0 \end{pmatrix}, \quad b_{-4}\mapsto \begin{pmatrix} 0 &1 \\ 0 &0 \end{pmatrix}. 
\end{equation}
Similarly,  we let $V_{3,3}$ denote the orthogonal complement of $\Q b_1+\Q b_{-1}$ inside $V$. So $V$ admits an orthogonal decomposition
\begin{equation}
\label{decomposition of V stabilized by MQ}
V=\Q b_1\oplus V_{3,3}\oplus \Q b_{-1},
\end{equation} 
where $V_{3,3}=(\Q b_1+\Q b_{-1})^{\perp}$ is a split quadratic space of signature $(3,3)$. 
\subsection{The Group $G$}
\label{subsec-the-Lie-Algebra-of-G}
Throughout the paper,  $G=\Spin(V)$
denotes the spin group associated to the split $8$-dimensional quadratic space $V$.  We fix a covering homomorphism $\pi \colon G\to \SO(V)$. 
Consider embedding of $\SO(V)$-modules,
\begin{equation}
\label{eqn-embedding-exterior-square-map}
    \bigwedge^2V\to \mathrm{End}(V)
\end{equation}
sending $v_1\wedge v_2\in \bigwedge^2V$ to the endomorphism $v_1\wedge v_2\cdot x=(x,v_2)v_1-(x,v_1)v_2$ where $x\in V$. Then \eqref{eqn-embedding-exterior-square-map} induces an identification between $\wedge^2V$ and $\Lie(\SO(V))$.  Furthermore,   \eqref{eqn-embedding-exterior-square-map} promotes to an identification of Lie algebras where the bracket on $\wedge^2V$ is defined by
\begin{equation}
\label{defn-Lie-bracket}
[v\wedge w, v'\wedge w']=(v\wedge w\cdot v')\wedge w'+v'\wedge(v\wedge w\cdot w').
\end{equation}
As such,  $\wedge^2V$ gives a model for the Lie algebra of $G$,  in which the adjoint action is described by $g\cdot (v\wedge w)=gv\wedge gw$ for $g\in G$ and $v,w\in V$.  \\
\indent The Lie algebra $\Lie(G)$ admits another model $\g_E$,  defined via an algebraic structure $E$ known as a cubic norm structure.  We refer the reader to \cite[\S 4.2]{pollackLL} for the general definition of cubic norm structures.  For our purposes, an important example will be the cubic norm structure, $E=\G_a^3$, which consists of $3$ by $3$ diagonal matrices $(z_1,z_2,z_2)$.  \\
\indent As a subset of the space of $3$ by $3$ matrices, $E$ naturally carries a multiplicative structure.  The cubic norm structure on $E$ is specified by the choice of base point $1_E=(1,1,1)$, the norm $N_E\colon E\to \G_a$,  given by $N_E(Z) =\det(Z)$, and the adjoint on $E$, which is the quadratic map $(\cdot)^{\#}\colon E\to E$ given by $(z_1,z_2,z_3)^\# = (z_2 z_3, z_3 z_1,z_1 z_2)$. For $Z,Z' \in E$ one sets $Z \times Z' = (Z+Z')^\#-Z^\#-(Z')^\#$.  We define a trace pairing $(Z,Z')_E=z_1z_1'+z_2z_2'+z_3z_3'$, which induces an identification $E^{\vee}=E$. \\
\indent Let $V_3=\Span\{e_1,e_2,e_3\}$ be the standard representation of $\SL_3$, and $V_3^\vee=\Span\{\delta_1, \delta_2,\delta_3\}$ its dual.  Then $V_3$ and $V_3^{\vee}$ are endowed with an action of the Lie algebra $\sl_3$.  Let $E^0$ be the trace $0$ subspace of $E$. So $E^0$ consists of elements $Z\in E$ such that $(Z , 1_E)_E=0$.  Equivalently, we view elements of $E^0$ as endomorphism of $E$, where an element $u\in E^0$ corresponds to the endomorphism $\Psi_u\colon E\to E$, which is given by $x\mapsto ux$. \\
\indent  As a vector space, $\g_E$ is given as 
\begin{equation}
\label{eqn-defn-of-gE}
\g_E = (\sl_3 \oplus E^0) \oplus (V_3 \otimes E) \oplus (V_3^\vee \otimes E^\vee)
\end{equation}
We refer the reader to \cite[\S 4]{pollackQDS} for the definition of the Lie bracket on $\g_E$.  
\subsection{The Heisenberg Parabolic Subgroup  $P$}
\label{subsec:The Heisenber Type Parabolic P} 
Let $U$ (resp. $U^{\vee}$) denote the isotropic subspace of $V$ with basis $\{b_1,b_2\}$ (resp. $\{b_{-2},b_{-1}\}$). The Heisenberg parabolic subgroup $P\leq G$ is defined as the stabilizer
$$
P=\mathrm{Stab}_G(U). 
$$
The Levi decomposition of $P$ takes the form $P=M_PN_P$ . Here $N_P$ is the unipotent radical of $P$ and $M_P$ is the Levi factor $M_P=\mathrm{Stab}_P(U^{\vee})$.
Equivalently,  $P$ is the parabolic subgroup associated to the grading on $\Lie(G)$ defined by the element $h_P:=b_1\wedge b_{-1}+b_2\wedge b_{-2}$. That is, $\mathrm{ad}(h_P)$ has eigenvalues $-2,-1,0,1,2$ on $\Lie(G)$,  $\Lie(M_P)$ is the $0$ eigenspace of $\mathrm{ad}(h_P)$, and $\Lie(N_P)$ is the direct sum of the $1$ and $2$ eigenspaces of $\mathrm{ad}(h_P)$.  \\
\indent Using the definition \eqref{defn-Lie-bracket}, $\Lie(M_P)=U\wedge U^{\vee}+V_{2,2}\wedge V_{2,2}$, and so
$$
\Lie(M_P)=\Lie(\GL(U))\oplus \Lie(\SO(V_{2,2})).
$$
Since $G$ is semi-simple and simply connected, the same is true of the derived subgroup $M_P^{\mathrm{der}}$ \cite[8.4.6(6)]{MR1642713}.  Therefore, since $\Lie(M_P^{\mathrm{der}})= \Lie(\SL(U))\oplus \Lie(\SO(V_{2,2})$,  
$$M_P^{\mathrm{der}}= \SL_2\times \SL_2\times \SL_2. 
$$
Since $\Lie(N_P)=U\wedge U+U\wedge V_{2,2}$, the center $Z\leq N_P$ equals the $\lambda=2$ eigenspace of $h_P$. Thus $\Lie(Z)=\Q\linspan\{b_1\wedge b_2\}$, $[N_P, N_P]=Z$,  and the Lie algebra of $N_P^{\mathrm{ab}}:=N_P/Z$ is $\Lie(N_P^{\mathrm{ab}})=\{b_1\wedge v+b_2\wedge v' +\Lie(Z) \colon v,v'\in V_{2,2}\}$.
If $T_1,T_2\in V_{2,2}$,  let 
\begin{equation}
    \label{eqn defn epsilon}
    \varepsilon_{[T_1,T_2]}\colon N_P(\Q)\backslash N_P(\A)\to \C^{\times} 
\end{equation}
be the character associated to the functional $b_1\wedge v+b_2\wedge v'\mapsto (T_1,v)+(T_2,v')$. \\
\indent As a module over $M_P(\Q)$, the character lattice of $N_P(\Q)\backslash N_P(\A)$ is 
\begin{equation}
\label{defn-characters-NP} 
U^{\vee}\otimes V_{2,2} \xrightarrow{\sim} \mathrm{Hom}(N_P(\Q) \backslash N_P(\A), \C^1), \qquad b_{-1}\otimes T_1+b_{-2}\otimes T_2\mapsto \varepsilon_{[T_1,T_2]}.
\end{equation}
Via the isomorphism \cite[Theorem A.8]{JMNPR24}, $\Lie(N_P^{\mathrm{ab}})$ is identified with the subalgebra of $\g_E$ given by
$
\Lie(N_P^{\mathrm{ab}})=\Q E_{12}+e_1\otimes E+\delta_3\otimes E^{\vee}+\Q E_{21}, 
$
and $\Lie(Z)= \Q E_{13}$.  \\
\indent Consider, 
$W_E:=\Q\otimes E\oplus E^{\vee}\oplus \Q$
as a symplectic vector via the form 
$$
\langle (a,b,c,d), (a',b',c',d')\rangle_{W_E}=ad'-(b,c')_E+(c,b')_E-a'd, 
$$
where $a,a',d,d'\in \Q$, $b,b'\in E$, and $c,c'\in E^{\vee}$.  Given $w\in W_E$, let $\varepsilon_w$ denote the character of $N_P$ associated to the functional 
$
aE_{12}+e_1\otimes b+\delta_3\otimes c+dE_{21}\mapsto\langle  w, (a,b,c,d)\rangle_{W_E}. 
$
Then $w\mapsto \varepsilon_w$ induces a second identification 
\begin{equation}
\label{defn-characters-NP2} 
W_E \xrightarrow{\sim} \mathrm{Hom}(N_P(\Q) \backslash N_P(\A), \C^1).
\end{equation}
By \cite[\S 7.4]{JMNPR24}, under the isomorphism $W_E\simeq U^{\vee}\otimes V_{2,2}$ induced by \eqref{defn-characters-NP} and \eqref{defn-characters-NP2}, if $w=(a, \mathrm{diag}(\beta_1, \beta_2, \beta_3), \mathrm{diag}(\gamma_1, \gamma_2, \gamma_3), d)\in W_E$ then, $\varepsilon_w$ is identified with $\varepsilon_{[T_1,T_2]}$ where
\begin{equation}
\label{eqn-character-lattice isom}
[T_1, T_2]= [\gamma_1 b_{-4} -\beta_2 b_{-3} + db_{3} + \gamma_3 b_{4},-\beta_3 b_{-4} + a b_{-3} -\gamma_2 b_{3}  - \beta_1 b_{4}].
\end{equation}
\subsection{The Orthogonal Parabolic Subgroup $R$}
\label{subsec:The Klingen Type Parabolic Q}
Let $R$ denote the parabolic subgroup of $G$ that stabilizes the line $\Q b_1$. Then $R$ admits a Levi decomposition $R=M_RN_R$ where $M_R$ is the subgroup of $R$ stabilizing the line $\Q b_{-1}$, and $N_R\trianglelefteq R$ denotes the unipotent radical of $R$. Equivalently,  $R$ is the parabolic subgroup of $G$ associated to the element $h_R=b_1\wedge b_{-1}$, which acts on $\Lie(G)$ with eigenvalues $-1,0,1$. Since $\Lie(M_R)=b_1\wedge b_{-1}+\bigwedge^2V_{3,3}$ equals to the $\lambda=0$ eigenspace of $\mathrm{ad}(h_R)$,
\begin{equation}
\label{Lie(M_R) identified} 
\Lie(M_R)=\Lie(\GL(\Q b_1))\oplus \Lie(\SO(V_{3,3})).
\end{equation}
Hence, $M_R^{\mathrm{der}}\simeq \Spin(V_{3,3})$, and the identification $\Lie(N_R)=b_1\wedge V_{3,3}$ gives an $M_R^{\mathrm{der}}$-module isomorphism between $V_{3,3}$ and the character lattice of $N_R(\Q)\backslash N_R(\A)$.  Precisely, we have
\begin{equation}
\label{eqn-character-lattice-R}
V_{3,3}(\Q) \xrightarrow{\sim} \mathrm{Hom}(N_R(\Q)\backslash N_R(\A), \C^{\times}), \quad y\mapsto \chi_y
\end{equation}
where $\chi_y(\exp(b_1\wedge v))=\psi((v,y))$ for all $v\in V_{3,3}(\A)$ and $y\in V_{3,3}(\Q)$. 
\subsection{The Siegel Parabolic Subgroup $Q'$}
\label{subsec-The-Siegel-Parabolic-$Q'$} 
In this subsection we assume that $y\in V_{3,3}$ is orthogonal to the  plane $\Q\linspan\{b_2,b_{-2}\}$ and satisfies $(y,y)\neq 0$.  Define $V_{2,3}'$ to be the complement of $\Q y$ inside $V_{3,3}$,  so that $V_{2,3}'$ is a rational quadratic space of signature $(2,3)$. The group $M'=M'_y$ is defined as
\begin{equation}
M'=\mathrm{Stab}_{\Spin(V_{3,3})}(y)\simeq \mathrm{Spin}(V_{2,3}'). 
\end{equation} 
 Since $y$ is orthogonal to $\Q b_2$,  $M'$ contains a parabolic subgroup $Q'$ defined as the stabilizer in $M'$ of the line $\Q b_2$.  Write $Q'=M_{Q'}N_{Q'}$ for the Levi decomposition of $Q'$ where
$$
M_{Q'}=\mathrm{Stab}_{Q'}(\Q b_{-2}).
$$ We let $V_{1,2}'$ denote the orthogonal complement of $\Q b_2+\Q y+ \Q b_{-2}$ inside of $V_{3,3}$. Then $M_{Q'}$ acts on $V_{2,3}'$ preserving the decomposition $V_{2,3}' =\Q b_2 +V_{1,2}'+ \Q b_{-2}$.  \\
\indent The unipotent radical $N_{Q'}\trianglelefteq Q'$ is abelian, and may be identified as a module over the derived subgroup $M_{Q'}^{\mathrm{der}}\simeq \Spin(V_{1,2}')$ according to the map $V_{1,2}'\xrightarrow{\sim} N_Q$,  $v \mapsto \exp(b_2\wedge v)$. As such,  $\mathrm{Hom}(N_{Q'}(\Q) \backslash N_{Q'}(\A), \C^1)$ is identified with $V_{1,2}'(\Q)$ via 
\begin{equation}
\label{defn-characters-Q'}
V_{1,2}'(\Q) \xrightarrow{\sim} \mathrm{Hom}(N_{Q'}(\Q) \backslash N_{Q'}(\A), \C^1), \qquad S\mapsto \varepsilon_{[0,S]}.
\end{equation}
Here the character $\varepsilon_{[0,S]}$ is defined in \eqref{eqn defn epsilon}. 
\subsection{Compact Subgroups}
\label{subsec-compact-subgroups}
Let $V^+$ be the definite $4$ plane in $V(\R)$ spanned by the vectors 
$$
 u_1=\frac{1}{\sqrt{2}}(b_1+b_{-1}),\quad u_2=\frac{1}{\sqrt{2}}(b_2+b_{-2}), \quad v_1=\frac{1}{\sqrt{2}}(b_3+b_{-3}), \quad v_2=\frac{1}{\sqrt{2}}(b_4+b_{-4})
$$
Set $V^{-}:=(V^+)^{\perp}$ so that $V^-$ is the negative definite subspace spanned by
$$
u_{-1}=\frac{1}{\sqrt{2}}(b_1-b_{-1}), \quad u_{-2}=\frac{1}{\sqrt{2}}(b_2-b_{-2}), \quad v_{-1}=\frac{1}{\sqrt{2}}(b_3-b_{-3}), \quad v_{-2}=\frac{1}{\sqrt{2}}(b_4-b_{-4}).
$$ The maximal compact subgroup $K_{\infty}\leq G(\R)$ is defined as 
$$K_{\infty}=\mathrm{Stab}_{G(\R)}(V^+).$$Then the Lie algebra $\k_0=\Lie(K_{\infty})$ is the $1$ eigenspace of the Cartan involution $\theta\colon \wedge^2 V\to \wedge^2V$ given by $\theta(b_i\wedge b_{\pm j})=b_{-i}\wedge b_{\mp j}$.  Set $\p_0$ to be the $-1$ eigenspace of $\theta$. Then if $\k=\k_0\otimes \C$ and $\p=\p_0\otimes \C$, we obtain a Cartan decomposition of $\g=\Lie(G)\otimes_{\Q} \C$ as $
\g=\k\oplus \p$. \\ 
\indent At the level of groups, we have a homomorphism  
\begin{equation}
\label{eqn-2-copies of Spin}
\Spin(V^+)\times \Spin(V^-)\to K_{\infty}, \quad (g,h)\mapsto gh,
\end{equation}
which is induced by the inclusions of Clifford algebras $\mathrm{Cl}(V^{\pm})\hookrightarrow \mathrm{Cl}(V(\R))$.  The kernel of \eqref{eqn-2-copies of Spin} is the diagonally embedded $\mu_2=\{(1,1), (-1,-1)\}$.  By comparing dimensions, it follows that the derivative of \eqref{eqn-2-copies of Spin} is an isomorphism.  Hence, $K_{\infty}^0=(\Spin(V^+)\times \Spin(V^-))/\mu_2$. Since $G$ is simply connected, and simple,  $G(\R)$ is connected.  Hence,  $K_{\infty}=K_{\infty}^0$, and $K_{\infty}=(\Spin(V^+)\times \Spin(V^-))/\mu_2$.\\
\indent When we define quaternionic modular forms on $G$ in Subsection \ref{subsec: Quaternionic Modular Forms on G}, it will be necessary to speak about a distinguished three dimensional representation $\V$ of $K_{\infty}$. To define $\V$, we make a Lie algebra argument.  Following \cite[\S 3.4]{JMNPR24}, consider the $\SL_2$ triple in $\g$:
\begin{compactenum}[(i)]
	\item $e^+ = \frac{1}{2}(u_1-iu_2) \wedge (v_1 - iv_2)$
	\item $h^+ = i (u_1 \wedge u_2 + v_1 \wedge v_2) = \frac{1}{2} (u_1-iu_2) \wedge (u_1+iu_2) + \frac{1}{2}(v_1-iv_2) \wedge (v_1+iv_2)$
	\item $f^+ = -\frac{1}{2} (u_1+iu_2) \wedge (v_1 + iv_2)$.
\end{compactenum}
Then $K_{\infty}$ acts on $\sl_2=\C\linspan\{e^+, h^+, f^+\}$ via the adjoint representation, which defines the representation $\V$.  We choose a basis $x,y$ of $\C^2 = V_2$ so that $\V$ is identified with $\mathrm{Sym}^2V_2$ via $e^+ = -x^2$, $h^+ = 2xy$, $f^+ = y^2$.  For $\ell\in \Z_{\geq 1}$, we write $\mathbf{V}_{\ell}$ for the $\ell^{th}$ symmetric power of $\V$.  So \begin{equation}
\label{defn-of-Vell}
\mathbf{V}_{\ell} = \mathrm{Sym}^{2\ell}(V_2),
\end{equation} 
which has a basis $x^{2\ell}, x^{2\ell-1}y, \ldots, xy^{2\ell-1}, y^{2\ell}$. 
\indent The Lie algebra $\so(V^+)$ contains a second $\SL_2$ triple $\sl_2'$,  which is obtained by replace $v_2$ with $-v_2$ in the definition of $\sl_2$.  Then $\so(V^+)=\sl_2\oplus \sl_2'$.  Similarly, $\so(V^-)$ admits an orthogonal decomposition as $\so(V^-)=\sl_2''+\sl_2'''$ where $\sl_2''$ (resp. $\sl_2'''$) is obtain by replacing $u_i$ with $u_{-i}$ and $v_{i}$ with $v_{-i}$ in the definition of $\sl_2$ (resp.  $\sl_2'$).  In this way,  we present $\k=\k_0\otimes_{\R} \C$ as
\begin{equation}
\label{eqn-presentation-of-Lie(k)}
\k=\sl_2\oplus \sl_2'\oplus \sl_2''\oplus \sl_2'''.
\end{equation}
Write $L\leq K_{\infty}$ to denote the subgroup of $K_{\infty}$ with Lie algebra $ \sl_2'+\sl_2''+\sl_2'''$.  Then 
$$
K_{\infty}=(\SU(2)\times L)/\mu_2. 
$$
and $L=\SU(2)\times \SU(2)\times \SU(2)$.  The embedding $\mu_2\hookrightarrow \SU(2)\times L$ sends $-1$ to $(-1,-1,-1,-1)$. Thus if $V_2$ denotes the standard two dimensional representation of $\SU(2)$ and $W=V_2\otimes V_2\otimes V_2$,  then $V_2\boxtimes W$ gives a representation representation of $K_{\infty}$, and as $K_{\infty}$-module, we have 
\begin{equation}
\label{eqn-tensor-decomposition-W}
\p=V_2\boxtimes W.  
\end{equation}
\section{Preliminaries on Holomorphic Modular Forms}
\label{Sec-Modular-Forms-on-Spin(8)-and-Sp4}
\subsection{Holomorphic modular forms on $M'$}
\label{subsec:HMFS} Suppose $y\in V_{3,3}$ satisfies $(y,y)>0$.  Recall the orthogonal parabolic subgroup $R=M_RN_R$ and the subgroup $M'=\Spin(V_{2,3}')$ of Subsection \ref{subsec:The Klingen Type Parabolic Q}.   Let $\widehat{y}=y/\sqrt{(y,y)}$ and let $g_{y}\in G(\R)$ be any element satisfying 
$$
g_y\cdot v_1=\widehat{y}, \quad g_y\cdot u_{\pm 1}=u_{\pm 1}, \quad  g_y\cdot u_{\pm 2}=u_{\pm 2}, \quad g_y\cdot v_{\pm 2}=v_{\pm 2}.
$$
Then $g_yK_{\infty}g_y^{-1}$ is the maximal compact subgroup in $G(\R)$ stabilizing the $4$-plane spanned by $\{u_1, u_2, \widehat{y}, v_2\}$.  Let $K_y:=(g_yK_{\infty}g_y^{-1})\cap M'(\R)$ be the maximal compact subgroup of $M'(\R)$ stabilizing the subspace $\R\linspan\{u_2, v_2\}$ and write $\V_{\ell}^{y}$ to denote the representation of $K_y$ on the vector space $\V_{\ell}$ obtained by $k\cdot v= g_y^{-1}kg_y\cdot v$.  \\
\indent Recall that $V_{1,2}'$ denotes the complement of $\Q b_2+\Q y+\Q b_{-2}$ in $V_{3,3}$.  Then the symmetric domain $M'(\R)/K_y$ is identified with the complex manifold
\[\mathfrak{h}_{y}' = \{X+ iY \in V_{1,2}' \otimes_{\Q} \C: (Y,-\sqrt{2}v_2) > 0 \text{ and } (Y,Y) > 0\}.\]
To identify $M'(\R)/K_y$ with $\h_y'$,  recall that $V_{2,3}'$ denotes the complement of $\Q\linspan\{y\}$ in $V_{3,3}$. Then $\mathfrak{h'}_{y}$ maps into the subspace of isotropic elements in $V_{2,3}'\otimes_{\Q}\C$ via the map
\begin{equation}
\label{eqn-defn-action-on-h}
\mathfrak{h}_{y}'\to V_{2,3}'\otimes_{\Q}\C, \qquad Z \mapsto r(Z):= -q(Z) b_2 + Z + b_{-2}.
\end{equation}
This yields an action of the identity component $M'(\R)$ on $\mathfrak{h}_{y}'$ as follows: If $g \in M'(\R)$, then there exists a unique element $j_{y}(g,Z)\in \C^{\times}$ and a unique element $gZ \in \mathfrak{h}_{y}'$ so that 
$$
g r(Z) = j_{y}(g,Z) r(gZ).
$$Observe that $j_{y}(g,Z) = (g r(Z), b_2)$ and $K_y$ is the stabilizer of $i(-\sqrt{2}v_2)$ in $M'(\R)$. Moreover, $j_y(\cdot, -i\sqrt{2}v_2): K_{y} \rightarrow \C^\times$ is a character. \\
\indent Later, we will specialize to the case when $y=b_3+\alpha b_{-3}/2$ with $\alpha\in 2\Z_{>0}$. In this case we have explicit coordinates on $\h_{y}'$ given as follows.  Let  $y_{\alpha}^{\vee}=b_3-\alpha b_{-3}/2$ so that 
$$
V_{1,2}'=\Z\linspan\{b_4, y_{\alpha}^{\vee}, b_{-4}\}. $$ Then a general vector $Z\in V_{1,2}'\otimes_{\Q}\C$ takes the form 
\begin{equation}
\label{eqn-coordinates-on-h}
Z=-\tau'b_4+zy_{\alpha}^{\vee}-\tau b_{-4}
\end{equation}
where $\tau',\tau,z\in \C$.  A short computation shows that with notation as in \eqref{eqn-coordinates-on-h},  $Z\in \h_{y}'$ if and only if $\mathrm{Im}(\tau)>0$,  $\mathrm{Im}(\tau')>0$,  and $\mathrm{Im}(\tau)\cdot \mathrm{Im}(\tau')-\dfrac{\alpha}{2} \mathrm{Im}(z)^2>0$.
\begin{definition}
\label{defn-HMF-on-Sp(4)}
Suppose $\ell \in \Z$, and $\Gamma \subseteq M'(\R)$ is an arithmetic subgroup.  We say that $f: \mathfrak{h}_y \rightarrow \C$ is a holomorphic modular form of weight $\ell$ and level $\Gamma$ if:
	\begin{compactenum}
		\item $f$ is holomorphic,
		\item$f(\gamma Z) = j_y(\gamma, Z)^{\ell} f(Z)$ for all $Z \in \mathfrak{h}_y$ and $\gamma \in \Gamma$, and
		\item $\xi(g):=j_y(g,-i\sqrt{2}v_2)^{-\ell} f(g \cdot (-i\sqrt{2}v_2)): M'(\R) \rightarrow \C$ is of moderate growth.
	\end{compactenum}
\end{definition}
\begin{definition}
\label{defn-automorphic-function-associated-to-holomorphic-mf}
Suppose $\ell \in \Z$, and $\Gamma \subseteq M'(\R)$ is an arithmetic subgroup. We say that a function $\xi: M'(\R) \rightarrow \C$ is the automorphic function associated to a holomorphic modular form of weight $\ell$ and level $\Gamma$ if $\xi$ is of moderate growth and satisfies the following conditions: 
\begin{compactenum}
\item if $g_{\infty}\in M'(\R)$ and $\gamma\in \Gamma$ then $\xi(\gamma g_{\infty}) =\xi(g_{\infty})$,
\item if $g_{\infty}\in M'(\R)$ and $k\in K_y$ then $\xi(g_{\infty}k) = j_y(k,-i\sqrt{2}v_2)^{-\ell}\xi(g_{\infty})$, and 
\item   if $g_{\infty}\in M'(\R)$ and $Z:=g_{\infty} \cdot (-i\sqrt{2}v_2)\in \h_y$ then the formula 
$$
f_{\xi}(Z)=j_y(g_{\infty}, -i\sqrt{2}v_2)^{\ell}\xi(g_{\infty})
$$ 
is holomorphic in $Z$. 
\end{compactenum}
\end{definition}
\subsection{The Fourier Expansions of Modular Forms on $M'$} 
\label{subsec:FE-HMFS} 
We continue with the notation of the previous subsection.  Thus $y\in V_{3,3}$ denotes a vector  in the orthogonal complement to $\R\linspan\{u_2, v_2\}$ satisfying $(y,y)>0$.  Furthermore,  we assume that $y$ is orthogonal to $b_2$,  in which case $M'$ contains the parabolic subgroup $Q'=M_{Q'}N_{Q'}$ of Subsection \ref{subsec-The-Siegel-Parabolic-$Q'$}.  

\indent Since the unipotent radical $N_{Q'}$ is abelian,  we may apply the identification \eqref{defn-characters-Q'} to Fourier expand an automorphic function $\xi \colon M'(\Q)\backslash M'(\A)\to \C$ in characters of $N_{Q'}(\Q)\backslash N_{Q'}(\A)$ as 
\begin{equation}
\label{FE-GENERAL-AUT-FORM-M'} 
\xi(g)=\xi_{N_{Q'}}(g)+\sum_{S\in V_{1,2}'(\Q) \colon S\neq 0} \xi_S(g). 
\end{equation}
Here $g\in M'(\A)$ and $ \xi_S(g)=\displaystyle{\int_{N_{Q'}(\Q)\backslash N_{Q'}(\A)}}\xi(ng)\varepsilon_{[0,S]}(n)^{-1}\,dn$. \\
\indent For the remainder of this subsection, we specialize to the case when $y=nb_3+n\alpha b_{-3}/2$, where $n\in \Z_{\geq 1}$ and $\alpha \in 2\Z_{>0}$.  For this choice of $y$,  the complement of $\Q y$ in $V_{3,3}$ is 
$$
V_{2,3}'=\Q\linspan\left\{b_2, b_4, b_{-4}, y_{\alpha}^{\vee}, b_{-2}\right\}
$$
where $$
y_{\alpha}^{\vee}=b_3-\dfrac{\alpha}{2}b_{-3}.
$$Let $M'(\Z)=M'(\Q) \cap G(\widehat{\Z})$ viewed as a discrete subgroup of $M'(\R)$.  Regarding the stucture of $M'(\Z)$, we record the following elementary result. 
\begin{lemma} Let $V_{1,2}'(\Z)=\Z\linspan\left\{b_4, y_{\alpha}^{\vee}, b_{-4}\right\}$. Then 
\begin{equation}
\label{eqn-identification-N_Q'(Z)}
N_{Q'}(\Q)\cap M'(\Z)=\{\exp(b_2\wedge v) \colon v\in V_{1,2}'(\Z)\}. 
\end{equation}
\end{lemma}
Suppose $\xi\colon M'(\R)\to \C$ is the automorphic functions associated to a holomorphic modular form of weight $\ell$ and level $M'(\Z)$.  Writing $N_{Q'}(\Z)=N_{Q'}(\R)\cap M'(\Z)$,  the character lattice of $N_{Q'}(\Z)\backslash N_{Q'}(\R)$ is identified with the $\Z$-linear dual of $V_{1,2}'(\Z)$, i.e.
$$
V_{1,2}'(\Z)^{\vee}=\Z\linspan\left\{b_4, \dfrac{1}{\alpha}y_{\alpha}^{\vee}, b_{-4}\right\}.
$$
More precisely,  we have an $M_{Q'}(\Z)$ equivariant identification 
\begin{equation}
\label{eqn-defn-characters-of-N_Q(Z)}
V_{1,2}'(\Z)^{\vee} \xrightarrow{\sim}\mathrm{Hom}(N_{Q'}(\Z)\backslash N_{Q'}(\R), \C^1), \quad S\mapsto \varepsilon_{[0,S]}. 
\end{equation}
Thus if $g_{\infty}\in M'(\R)$, $Z=g_{\infty} \cdot (-i\sqrt{2}v_2)$,  and 
$$
f_{\xi}(Z)=j_y(g_{\infty}, -i\sqrt{2}v_2)^{\ell}\xi(g_{\infty}),
$$then $f_{\xi}$ Fourier expands along $N_{Q'}(\Z)\backslash N'(\R)$ as 
\begin{equation}
\label{FE-HMF-on-M'}
f_{\xi}(Z)=\sum_{S\in V_{1,2}'(\Z)^{\vee}_{\geq 0}} A_{\xi}[S]\exp(2 \pi i (S,Z)). 
\end{equation}
Here the scalars $A_{\xi}[S]$ are the Fourier coefficients of $\xi$ and
$$
V_{1,2}'(\Z)^{\vee}_{\geq 0} =\{S\in V_{1,2}'(\Z)^{\vee}\colon (S,S)\geq 0,  (S,-\sqrt{2}v_2)\geq 0\}.$$The condition $S\in V_{1,2}'(\Z)^{\vee}_{\geq 0}$ is a consequence of the Koecher principle.  \\
\indent Using \eqref{eqn-coordinates-on-h},  and writing 
$$S=-nb_4-mb_{-4}-\dfrac{r}{\alpha}y_{\alpha}^{\vee}$$ with $n,r,m\in \Z$,  the expansion \eqref{FE-HMF-on-M'} takes the form
\begin{equation}
\label{FE-HMF-EXPLICIT-COORDINATES}
f_{\xi}(\tau', z,\tau)=\sum_{\tiny \begin{array}{c} n,r,m\in \Z \\ n,m,2\alpha nm-r^2\geq 0\end{array}}A_{\xi}[-nb_4-mb_{-4}-\dfrac{r}{\alpha}y_{\alpha}^{\vee}]e^{2\pi i (\tau'm+n\tau+rz)}.
\end{equation}
Grouping terms in \eqref{FE-HMF-EXPLICIT-COORDINATES} gives the classical Fourier-Jacobi expansion (see for example \cite{EichZag}),  $f_{\xi}(\tau',  z, \tau)=
\sum_{m\in \Z_{\geq 0}} \phi_m(\tau, z) e^{2\pi i m \tau'}$,  where
\begin{equation}
\label{FJ-HMF-M'}
\phi_m(\tau, z)=\sum_{\tiny \begin{array}{c}n,r\in \Z \\ n, 2\alpha nm-r^2\geq 0\end{array}} A_{\xi}[-nb_4-mb_{-4}-\dfrac{r}{\alpha}y_{\alpha}^{\vee}]e^{2\pi i (n\tau+rz)}.
\end{equation}
The coefficients $\phi_m$ satisfy the following well known transformation law. 
\begin{lemma}
\label{lemma-transformation-law-Jacobi-Forms}
Suppose $\left(\begin{smallmatrix} a &b \\ c & d\end{smallmatrix}\right)\in \SL_2(\Z)$,  $z\in \C$ and $\tau\in \C$ satisfies $\mathrm{Im}(\tau)>0$. Then 
\begin{equation}
\label{eqn-lemma-transformation-law-Jacobi-Forms}
\phi_m\left(\dfrac{ a\tau +b}{c\tau +d}, \dfrac{z}{c\tau+d}\right)=\exp\left(\dfrac{2\pi i mcz^2\alpha/2}{c\tau+d}\right)(c\tau+d)^{\ell}\phi_m(\tau, z). 
\end{equation}
\end{lemma}
\begin{proof}
The transformation law \eqref{eqn-lemma-transformation-law-Jacobi-Forms} follows from property 2 of Definition \ref{defn-HMF-on-Sp(4)}.  In more detail,  recall the  $\Q$-rational morphism from $M'$ to $\SO(V_{2,3}')$ defined via the action of $M'$ on $V_{2,3}'$. Writing elements of $\SO(V_{2,3}')$ as matrices relative to the basis $\{b_2, b_4, y_{\alpha}^{\vee}, b_{-4}, b_{-2}\}$,  we consider 
\begin{equation}
\label{eqn-definition-iota}
\iota\colon \SL_2\to \SO(V_{2,3}'), \qquad \gamma= \begin{pmatrix} a &b \\ c &d \end{pmatrix} \mapsto \iota(\gamma):=\left( \begin{smallmatrix}  
a &b & & & \\ 
c &d & & & \\ 
   &   &1 & & \\ 
   &   &   & a &-b \\ 
   &   &   &-c &d 
   \end{smallmatrix}\right).
   \end{equation}
Then $\iota$ lifts to give a map $\tilde{\iota}\colon \SL_2\to M'$.  \\
\indent 
Suppose $\gamma=\left(\begin{smallmatrix} a &b \\ c & d \end{smallmatrix}\right)\in \SL_2(\Z)$. Applying the description of the action of $M'$ on $\h_{y_{\alpha}}'$ given in \eqref{eqn-defn-action-on-h},  one calculates that with respect to the coordinates \eqref{eqn-coordinates-on-h},  $j_{y_{\alpha}}(\tilde{\iota}(\gamma), Z)=c\tau+d$ and 
\begin{equation}
\label{eqn-action-of-SL2-on-h}
\tilde{\iota}(\gamma)\cdot (\tau', z, \tau) = \left(\dfrac{c(\tau\tau'-z^2\alpha/2)+d\tau'}{c\tau +d}, \dfrac{z}{c\tau+d}, \dfrac{a\tau+b}{c\tau+d}\right).  
\end{equation}
Using \eqref{eqn-action-of-SL2-on-h} in tandem with the Fourier-Jacobi expansion of $f_{\xi}$, we obtain  
\begin{equation}
\label{eqn-expansion-fxi}
f_{\xi}(\tilde{\iota}(\gamma)\cdot (\tau', z, \tau))=\sum_{m\geq 0}\phi_m\left(\frac{a\tau +b}{c\tau +d}, \frac{z}{c\tau+d}\right)\exp\left(\frac{-2\pi i mcz^2\alpha /2}{c\tau+d}\right) e^{2\pi i m\tau' }.
\end{equation}
Since $\xi$ is of level $M'(\Z)$,  Definition \ref{defn-HMF-on-Sp(4)} implies$f_{\xi}(\tilde{\iota}(\gamma)\cdot (\tau', z, \tau))=(c\tau+d)^{\ell}f_{\xi}(\tau', z, \tau)$ for all $\gamma\in \SL_2(\Z)$ and $Z\in \h'_{y_{\alpha}}$.  Thus,  \eqref{eqn-lemma-transformation-law-Jacobi-Forms} follows by equating the coefficient of $e^{2 \pi i m\tau'}$ in \eqref{eqn-expansion-fxi} with the coefficient of $e^{2 \pi i m\tau'}$ in $(c\tau+d)^{\ell}f_{\xi}(\tau', z, \tau)$. 
\end{proof}
\subsection{A Primitivity Theorem for Modular Forms on $M'$.}
\label{subsec:A Primitivity Theorem for Modular Forms on $M'$}
In this subsection, we continue with the notation of the preceding subsection. Our goal is to establish Theorem \ref{thm-3-yamana}, which has its origins in a result of Zagier \cite[pg.  387]{ZaFrench}.  Theorem \ref{thm-3-yamana} is closely related to a special case of \cite[Theorem 3]{Yamana09},  except for the fact that the level subgroup $M'(\Z)$ is different from the level subgroups considered in (loc. cit.).  In spite of this difference,  the proof of Theorem \ref{thm-3-yamana} is essentially a special case of the proof given in (loc. cit.).
\begin{theorem}\label{thm-3-yamana}
Suppose $\xi\colon M'(\R)\to \C$ is the automorphic functions associated to a holomorphic modular form of weight $\ell>0$ and level $M'(\Z)$.  Suppose $A_{\xi}[S]=0$ for all vectors $S\in V_{1,2}'(\Z)^{\vee}$ such that $\Q\linspan\{S\}\cap V_{1,2}'(\Z)^{\vee}=\Z\linspan\{S\}$.  Then $\xi\equiv 0$. 
\end{theorem}
\begin{proof}
For the sake of contradiction, we suppose $\xi\not \equiv 0$.   Then there exists $m_0>0$ such that the Fourier-Jacobi coefficient $\phi_{m_0}(\tau, z) \neq 0$.  Since $\phi_{m_0}$ is holomorphic in the variable $z$, we may develop $\phi_{m_0}(\tau, z)$ into a Taylor series as $\phi_{m_0}(\tau,z)=\sum_{\nu\geq 0}\lambda_{\nu}(\tau)z^{\nu}$ where 
\begin{equation}
\label{eqn-nu-Taylor-coefficient-of-phi}
\lambda_{\nu}(\tau)=\sum_{n\geq 0}\left(\sum_{\tiny \begin{array}{c} \hbox{$r\in \Z$ such that} \\ 2\alpha nm_0-r^2\geq 0\end{array}} \dfrac{(2\pi i )^{\nu}A_{\xi}\left[-nb_4-m_0b_{-4}-\dfrac{r}{\alpha}y_{\alpha}^{\vee}\right]}{\nu !} \right)e^{2\pi i n\tau}.   
\end{equation}
Applying the transformation law \eqref{eqn-lemma-transformation-law-Jacobi-Forms},  we conclude that for all $\left(\begin{smallmatrix} a & b \\ c & d\end{smallmatrix}\right)\in \SL_2(\Z)$,
\begin{equation}
\label{eqn-leading-order-transformation-law}
\sum_{\nu\geq 0}\sum_{j\geq 0}\dfrac{1}{j!}\left(\dfrac{2\pi i c\alpha/2}{c\tau+d}\right)^j\lambda_{\nu}(\tau)z^{\nu+2j}=\sum_{\nu\geq 0}\dfrac{1}{(c\tau+d)^{\nu+\ell}}\lambda_{\nu}\left(\frac{a\tau+b}{c\tau+d}\right)z^{\nu}. 
\end{equation}
Let $\nu_0\geq 0$, be minimal such that $\lambda_{\nu_0}(z)\not \equiv 0$.  By equating the coefficients of $z^{\nu_0}$ on the left and right hand sides of \eqref{eqn-leading-order-transformation-law},  we deduce that $$
\lambda_{\nu_0}\left(\dfrac{a\tau+b}{c\tau+d}\right)=(c\tau+d)^{\nu_0+\ell}\lambda_{\nu_0}(z).$$It follows that $\lambda_{\nu_0}(\tau)$ is an elliptic modular form of weight $\ell+\nu_0$ and level $\SL_2(\Z)$.  Inspecting \eqref{eqn-nu-Taylor-coefficient-of-phi}, we determine that the numbers 
$$
a(n):=\sum_{\tiny \begin{array}{c} \hbox{$r\in \Z$ such that} \\ 2\alpha nm_0-r^2\geq 0\end{array}} \dfrac{(2\pi i )^{\nu_0}A_{\xi}\left[-nb_4-m_0b_{-4}-\dfrac{r}{\alpha}y_{\alpha}^{\vee}\right]}{\nu_0 !}
$$ are the Fourier coefficients of an elliptic modular form on $\SL_2(\Z)$ of weight $\ell+\nu>0$.  The proof now follows from the argument given in the proof of \cite[Theorem 3]{Yamana09}.
\end{proof}
\subsection{The Hecke Bound Characterization of Cusp Forms on $M'$} Continuing with the notation of the previous subsection, we assume $$
y=y_{\alpha}=b_3+\dfrac{\alpha}{2}b_{-3}.
$$
The purpose of this subsection is to prove the following cuspidality criterion. 
\begin{theorem}
\label{thm-hecke-bound-IFF-cuspidal-holomorphic}
Suppose $\ell\geq 5$ and let $f_{\xi}$ be a weight $\ell$ holomorphic modular form of level $M'(\Z)$. Then $f_{\xi}$ is cuspidal if and only if,  for all $S\in V_{1,2}'(\Z)^{\vee}_{\geq 0}$ satisfying $(S,S)>0$,
\begin{equation}
\label{eqn-hecke-bound-holomorphic}
A_{\xi}[S]\ll_{f_{\xi}}  (S,S)^{\frac{\ell+1}{2}}. 
\end{equation}
\end{theorem} 
\begin{remark}
\label{rmk-hecke-IFF-cuspidal-holomorphic}
As we shall see,  the proof of Theorem \ref{thm-hecke-bound-IFF-cuspidal-holomorphic} is essentially the same as the proof of \cite[Theorem 2.1]{KoMa14}.  With that said, the statement of Theorem \ref{thm-hecke-bound-IFF-cuspidal-holomorphic} differs from that of (loc. cit.) in at least two respects. For one, \cite[Theorem 2.1]{KoMa14} pertains to modular forms on $\Sp_4$ of level one. However,  though there is a connection between modular forms on $M'(\Z)$ and level one Siegel modular forms in the case when $\alpha=2$,  these two families of modular forms are, in general, distinct from one another.  Secondly,  in Theorem \ref{thm-hecke-bound-IFF-cuspidal-holomorphic}, $\ell\geq 5$ and the exponent in \eqref{eqn-hecke-bound-holomorphic} equals $(\ell+1)/2$.  However, in (loc. cit.), the authors assume $\ell\geq 4$ and their bound takes the form 
\begin{equation}
\label{eqn-hecke-bound2-holomorphic}
A_{\xi}[S]\ll_{f_{\xi}}  (S,S)^{\frac{\ell}{2}}. 
\end{equation} 
For our application to quaternionic modular forms, we require the additional flexibility afforded by the bound \eqref{eqn-hecke-bound-holomorphic}, though this comes at the cost of excluding $\ell=4$.  In the setting of modular forms on $\Sp_{2g}$, the tradeoff between the exponent in the Hecke bound, and the range of weights is well understood (see \cite[Theorem 4.1]{BoSo14}). 
\end{remark}
The proof of Theorem \ref{thm-hecke-bound-IFF-cuspidal-holomorphic} requires a close analogue of \cite[Theorem 2.2]{KoMa14}.
\begin{prop}
\label{prop-Hecke-IFF-Cusp-Jacobi-Forms}
Suppose $\ell\geq 5$ and let $\phi$ be a Jacobi form of level one, weight $\ell$ and index $m>0$.  Following \cite{EichZag}, we record the Fourier expansion of $\phi$ as 
$$
\phi(\tau,z)=\sum_{\tiny \begin{array}{c} n\geq 0, r\in \Z \\   \hbox{such that $r^2\leq 4mn$}\end{array}}c(n,r)e^{2\pi i (n\tau+rz)}. 
$$
Then $\phi$ is cuspidal if and only if the Fourier coefficients $c(n,r)$ satisfy the following condition: if $n\geq 0$, $r\in \Z$, and $D:=r^2-4mn< 0$, then 
\begin{equation}
\label{Hecke-IFF-Cusp-Jacobi-Forms}
c(n,r)\ll_{\phi}|D|^{\frac{\ell+1}{2}}.
\end{equation}
\end{prop}
\begin{proof}
For the proof of the ``only if" implication see \cite[Lemma 4.1]{KoMa14}. Conversely, suppose the non-degenerate Fourier coefficient of $\phi$ satisfy the bound \eqref{Hecke-IFF-Cusp-Jacobi-Forms}. Let $f\in \Z_{>0}$ be such that $f^2\mid m$ and write $m/f^2=ab^2$ with positive integers $a,b$ such that $a$ is the square-free part of $m/f^2$.  Then applying the argument in the proof of \cite[Theorem 2.2]{KoMa14}, we may assume that $\phi$ is a linear combination of Eisenstein series, and there exists a primitive Dirichlet character $\chi\colon (\Z/f\Z)^{\times} \to \C$ such that 
\begin{equation}
\label{eqn-expression-for-FC}
c(n,r)=\sum_{l\mid b}\lambda_{l}\left(\sum_{\tiny \begin{array}{c} d\mid (n,r,ab^2/l^2) \\ r/d \equiv 0\hspace{.1cm} (\mathrm{mod }\hspace{.1cm} l) \end{array}}d^{\ell-1}c_{\ell,f^2}^{\chi}\left(\dfrac{nab^2/l^2}{d^2}, \dfrac{r}{dl}\right)\right).
\end{equation}
Here $\lambda_{l}\in \C$ and the numbers $c_{k,f^2}^{\chi}(n,r)$ are Fourier coefficients of the Eisenstein series $E_{\ell,f^2}^{(\chi)}$ in \cite[pg. 26]{EichZag}.  By \cite[Lemma 4.3]{KoMa14},  if $r\geq 0$ is coprime to $f$, and $D=r^2-4nf^2$ is a fundamental discriminant, then there exists a constant $A_{\ell,f}>0$ such that 
$$
|c_{\ell,f^2}^{\chi}(n,r)|>A_{\ell,f}|D|^{\ell-3/2}.
$$
On the other hand,  since $\ell\geq 5$, $\frac{\ell+1}{2}<\ell-\frac{3}{2}$,  and so the inductive argument of \cite[pg. 1330]{KoMa14} implies $\lambda_{\ell}=0$ for all $\ell\mid b$.  Hence,  $\phi$ is necessarily zero, which completes the proof. 
\end{proof} 
\begin{proof}[Proof of Theorem \ref{thm-hecke-bound-IFF-cuspidal-holomorphic}:] The ``only if" implication follows directly from the Hecke bound \eqref{eqn-hecke-bound2-holomorphic}, which is satisfied for all holomorphic cusp forms on $M'$.  The converse implication is proven by the argument presented in \cite[\S 3]{KoMa14}. In more detail,  assuming the Fourier coefficients $A_{\xi}[S]$ satisfy \eqref{eqn-hecke-bound-holomorphic} for $(S,S)>0$, the results of Proposition \ref{prop-Hecke-IFF-Cusp-Jacobi-Forms} and Lemma \ref{lemma-transformation-law-Jacobi-Forms} imply that $A_{\xi}[S]=0$ whenever $S=-nb_{4}-mb_{-4}-\frac{r}{\alpha} y_{\alpha}^{\vee}$ satisfies $m> 0$ and $(S,S)=0$.  Since $M_{Q'}$ acts transitively on the isotropic lines in $V'_{1,2}$, it follows that $A_{\xi}[S]=0$ for all non-zero $S\in V'_{1,2}(\Z)$ satisfying  $(S,S)=0$.  It remains to show that the constant term of $f_{\xi}$ vanishes,  which is achieved by examining the $m=0$ term in \eqref{eqn-expansion-fxi}. 
\end{proof}
\subsection{Relation to Modular Forms on $\Sp_4$}
\label{subsec-Relation-to-Modular-Forms-on-Sp4} 
In this subsection we review the relationship between the theory of modular forms on $M'=M'_{y_{\alpha}}$ in the case when $\alpha=2$, and the theory of genus two Siegel modular forms.  This relationship is also explained \cite[\S 7.3]{JMNPR24}. 

\indent Let $\Sp_4$ denote the split symplectic group of rank $2$ over $\Q$, and let $\pi \colon \Sp_4\to \SO(V_{2,3}')$ denote the map constructed in (loc. cit.).  Since $\Sp_4$ and $M'$ are simply connected,  $\pi$ lifts to give an isomorphism $\widetilde{\pi} \colon \mathrm{Sp}_4\to M'$.  Hence, if $\xi$ is the automorphic function on $M'$ associated to a holomorphic modular form, we define 
$$
\xi^*:=\varphi\circ \widetilde{\pi}.
$$
\begin{prop}\emph{\cite[\S 7.3]{JMNPR24}}
If $\xi$ is the automorphic function corresponding to a weight $\ell$ holomorphic modular form on $M'$ of level $M'(\Z)$, then $\xi^*$ is the automorphic function on $\Sp_4$ corresponding to a genus $2$, Siegel modular form $F_{\xi^{\ast}}$ of weight $\ell$ and level $1$.  Moreover,  the Fourier expansion of $F_{\xi^{\ast}}$ takes the form 
\begin{equation}
\label{eqn-FE-Classical-Siegel-Modular-Form}
F_{\xi^{\ast}}(Z)=\sum_{T\geq 0}B_{\xi^{\ast}}[T]\exp(2\pi i \tr(TZ)), \qquad (Z\in \h_{\Sp_4}).
\end{equation}
Here $T$ runs over half-integral positive semi-definite matrices $T=\left(\begin{smallmatrix} a &b/2 \\ b/2 &c\end{smallmatrix}\right)$, and the Fourier coefficients $B_{\xi^{\ast}}[T]$ satisfy 
\begin{equation}
\label{eqn-relation-between-FCS}
B_{\xi^{\ast}}\left[\begin{pmatrix} a &b/2 \\ b/2 &c \end{pmatrix}\right]=A_{\xi}[-cb_4-\dfrac{b}{2}(b_3-b_{-3})-ab_{-4}]. 
\end{equation}
\end{prop}
\subsection{Holomorphic Modular Forms on $M_P^{\mathrm{der}}$}
\label{subsec:holomorphic-modular-forms-on-MP}
In Subsection \ref{subsec:The Heisenber Type Parabolic P}, we gave an identification between the derived subgroup of the Heisenberg Levi factor $M_P$, and the group $\SL_2\times \SL_2\times \SL_2$.  In this subsection we give preliminaries regard the holomorphic modular forms on the derived subgroup $M_P^{\mathrm{der}}$.  \\
\indent Let $\h_{\SL_2}$ be the complex upper half-plane and write $\h_{M_P^{\mathrm{der}}}=\h_{\SL_2}\times \h_{\SL_2}\times \h_{\SL_2}$.  Following \cite[Proposition 2.3.1]{pollackQDS}, we define an automorphy factor 
$$
j_{M_P^{\mathrm{der}}}\colon M_P^{\mathrm{der}}(\R)\times \h_{M_P^{\mathrm{der}}}\to \C, \quad ((g_1,g_2,g_3), (z_1,z_2,z_3))\mapsto j_{\SL_2}(g_1,z_1)j_{\SL_2}(g_2,z_2)j_{\SL_2}(g_3,z_3)
$$
where $j_{\SL_2}\colon \SL_2(\R)\times \h_{\SL_2}\to \C$ is the standard factor of automorphy $j(\left(\begin{smallmatrix} a& b\\ c &d \end{smallmatrix}\right), z)=cz+d$. 
\begin{definition}
\label{defn-holomorphic-modular-forms-on-MP}
  Write $\h_{M_P^{\mathrm{der}}}=\h_{\SL_2}\times \h_{\SL_2}\times \h_{\SL_2}$ so that $M_P^{\mathrm{der}}(\R)$ acts on $\h_{M_P^{\mathrm{der}}}$.  Let $\ell\in \Z_{>0}$ and suppose $\Gamma\leq M_P^{\mathrm{der}}(\R)$ is an arithmetic subgroup. A holomorphic modular form on $\h_{M_P^{\mathrm{der}}}$ of weight $\ell$ and level $\Gamma$ is a function $f\colon \h_{M_P^{\mathrm{der}}}\to \C$ such that 	\begin{compactenum}
		\item $f$ is holomorphic,
		\item if $z\in \h_{M_P^{\mathrm{der}}}$ and $\gamma\in \Gamma$ then 
		$
		f(\gamma\cdot z) =j_{M_P^{\mathrm{der}}}(\gamma,z)^{\ell}f(z)
		$, and 
		\item
		$
		\xi\colon M_P^{\mathrm{der}}(\R) \rightarrow \C, \quad \xi(g):=j_{M_P^{\mathrm{der}}}(\gamma,z)^{-\ell} f(g\cdot (i,i,i))$
		is of moderate growth.
	\end{compactenum}
\end{definition}
\begin{lemma}
\label{primitivity-holomorphic-modular-forms-on-MP}
Assume $\ell\in \Z_{>0}$. 
Suppose $f\colon \h_{M_P^{\mathrm{der}}}\to \C$ is a weight $\ell$ holomorphic modular form  of level $M_P^{\mathrm{der}}(\Z)$.  Write the classical Fourier expansion of $f$ as 
\begin{equation}
\label{eqn-FE-3-modular-forms-on-SL2}
f(z_1,z_2,z_3)=\sum_{n_1,n_2,n_3\geq 0}a(n_1,n_2,n_3)e^{2\pi i (n_1z_1+n_2z_2+n_3z_3)}. 
\end{equation}
If $a(n_1,n_2,n_3)=0$ for all $(n_1,n_2,n_3)\in \Z_{>0}^3$ satisfying $\mathrm{gcd}(n_1,n_2,n_3)=1$, then $f\equiv 0$. 
\end{lemma}
\begin{proof}
Suppose for a contradiction that $f\not\equiv 0$.  Then there exists $w=(w_1,w_2,w_3)\in\h_{M_P^{\mathrm{der}}}$ such that $f(w)\neq 0$.   Therefore, 
$$
f_1(z)=\sum_{n_1\geq 0}\left(\sum_{n_2,n_3\geq 0}a(n_1,n_2,n_3)e^{2\pi i (n_2w_2+n_3w_3)}\right)e^{2\pi i n_1z.} 
$$
is a non-zero elliptic modular form of weight $\ell>0$. So there exists $m>0$,  such that $g(z_2,z_3)=\sum_{n_2,n_3\geq 0}a(m,n_2,n_3)e^{2\pi i (n_2z_2+n_3z_3)}$ is a non-zero, weight $\ell$ holomorphic modular form on $\h_{\SL_2}\times \h_{\SL_2}$ of level $\SL_2(\Z)\times \SL_2(\Z)$.  Applying the same logic to $g$,  there exists an integer $m'>0$ such that $h(z_3)=\sum_{n_3\geq 0}  a(m,m',n_3)e^{2\pi i n_3z_3}$ is non-zero, weight $\ell>0$ modular form on $\h_{\SL_2}$ of level $\SL_2(\Z)$.  Thus if $d=\mathrm{gcd}(m,m')$, then there exists  an integer $m''>0$ such that $\gcd(d,m'')=1$ and $a(m,m',m'')\neq 0$,  completing the proof. 
\end{proof}
\section{Preliminaries on Quaternionic Modular Forms}
\label{Sec-Prelimns-on-Quat-Modular-Forms}
\subsection{Quaternionic Modular Forms on $G$}
\label{subsec: Quaternionic Modular Forms on G}
Let $\ell\in \Z_{\geq 1}$ and recall the representation $\mathbf{V}_{\ell}$ defined in \eqref{defn-of-Vell}. Following \cite{pollackQDS}, we now define quaternionic modular forms on $G$. \\
\indent We begin by reviewing the construction of a differential operator $D_{\ell}$ which goes back to the work of Schmid \cite{Schmid89}.  To specify $D_{\ell}$,  let $\g = \k \oplus \p$ be the Cartan decomposition of the complexified Lie algebra $\g$ of $G(\R)$ from Subsection \ref{subsec-compact-subgroups}.  Let $\{X_\alpha\}$ be a basis of $\p$ and $\{X_\alpha^\vee\}$ the dual basis of $\p^\vee$. One has that, as a representation of $K_{\infty}$, $\p \simeq \p^\vee \simeq V_2 \boxtimes W$, where the distinguished $\SU_2$ acts trivially on $W$ (see \eqref{eqn-tensor-decomposition-W}).

\indent  For a $K_{\infty}$-equivariant function $F:G(\R) \rightarrow \mathbf{V}_{\ell}$, set $\widetilde{D}_{\ell}F = \sum_{\alpha}{X_\alpha F \otimes X_\alpha^\vee}$.  The sum is independent of the choice of basis and $\widetilde{D}_{\ell}F$ takes values in 
$$
\mathbf{V}_{\ell}\otimes \p^\vee \simeq (\mathrm{Sym}^{2\ell-1}(V_2) \otimes W) \oplus (\mathrm{Sym}^{2\ell+1}(V_2) \otimes W).
$$ Let $\mathrm{pr}$ be the projection $\mathbf{V}_{\ell}\otimes \p^\vee \rightarrow \mathrm{Sym}^{2\ell-1}(V_2) \otimes W$.  Then $D_{\ell} = \mathrm{pr} \circ \widetilde{D}_{\ell}$.  
\begin{definition}
    Suppose $\ell\in \Z_{\geq 1}$. The space of weight $\ell$ (quaternionic) modular forms $M_{\ell}$ on $G$ is the space of smooth, moderate growth functions 
    $
    \varphi\colon G(\A) \to \V_{\ell}
    $ such that: 
  \begin{compactenum}
  \item if $\gamma\in G(\Q)$ and $g\in G(\A)$ then $\varphi(\gamma g)=\varphi(g)$, 
  \item there exists an open compact subgroup $K_f\leq G(\A_f)$ such that $\varphi$ is right $K_f$-invariant, 
  \item if $k\in K_{\infty}$, and $g\in G(\A)$, then $\varphi(gk)=k^{-1}\varphi(g)$, 
  \item $D_{\ell}\varphi \equiv 0$,  and 
  \item $\varphi$ is $Z(\g)$ finite.  
  \end{compactenum}
Let $S_{\ell}$ be the subspace of $M_{\ell}$ consisting of cusp forms. So $S_{\ell}$ consists of forms $\varphi\in M_{\ell}$ such that if $\mathcal{N}\leq G$ is the unipotent radical of a proper $\Q$-rational parabolic subgroup then 
$$
\varphi_{\mathcal{N}}(g):=\displaystyle{\int_{\mathcal{N}(\Q)\backslash \mathcal{N}(\A)}}\varphi(ng)\,dn
$$ 
is identically zero. 
Write $M_{\ell}(1)$ (resp. $S_{\ell}(1)$) to denote the subspace of $M_{\ell}$ (resp. $S_{\ell}$) consisting of forms $\varphi$ such that $\varphi(gk)=\varphi(g)$ for all $g\in G(\A)$ and $k\in G(\widehat{\Z})$. 
\end{definition}
\subsection{The Fourier Expansion of Quaternionic Modular Forms on $G$}
\label{subsec: FE Quaternionic Modular Forms on G}
The Heisenberg parabolic subgroup $P=M_PN_P$ is defined in  Subsection \ref{subsec:The Heisenber Type Parabolic P},  and the unipotent radical $N_P$ of $P$ is two-step nilpotent with center $Z$.  \\
\indent Recall the identification of the character lattice $\mathrm{Hom}(N_P(\Q) \backslash N_P(\A), \C^1)$ given in \eqref{defn-characters-NP}.  If $T_1,T_2\in V_{2,2}(\Q)$,  and $\varphi\in M_{\ell}$, we define the Fourier coefficient $\varphi_{[T_1,T_2]}$ through the formula 
\begin{equation}
\label{defn-FE-General-Aut-Form}
\varphi_{[T_1,T_2]}\colon G(\A) \to \V, \qquad \varphi_{[T_1,T_2]}(g)=\displaystyle{\int_{N_P(\Q)\backslash N_P(\A)}}\varphi(ng)\varepsilon_{[T_1,T_2]}^{-1}(n)\,dn. 
\end{equation}
Then since $Z=[N_P, N_P]$, we may Fourier expand the constant term $\varphi_Z$ as
\begin{equation}
\label{FE-general-phiZ}
\varphi_Z(g)=\varphi_{N_P}(g)+\sum_{T_1,T_2\in V_{2,2}(\Q)\colon [T_1,T_2]\neq [0,0] }\varphi_{[T_1,T_2]}(g). 
\end{equation}

\indent Write $\langle \cdot, \cdot \rangle $ for the $\GL(U)\times \SO(V_{2,2})$-invariant pairing between $U^{\vee}\otimes_{\Q} V_{2,2}$ and $U\otimes_{\Q} V_{2,2}$. So if $T_1, T_1',T_2,T_2'\in V_{2,2}(\R)$, then 
$$
\langle b_{-1}\otimes T_1+b_{-2}\otimes T_2, b_{1}\otimes T_1'+b_{2}\otimes T_2'\rangle =(T_1,T_1')+(T_2,T_2').
$$ For $T_1,T_2\in V_{2,2}(\R)$, define $\beta_{[T_1,T_2]}\colon M_P(\R)\to \C$ by the formula
\begin{equation}
\label{definition of beta}
\beta_{[T_1,T_2]}(r):=\sqrt{2} i\langle r^{-1}\cdot (b_{-1}\otimes T_1+b_{-2}\otimes T_2), b_{1}\otimes(v_1+iv_2)+b_{2}\otimes i(v_1+ iv_2)\rangle.
\end{equation}
For the convenience of the reader we recall that $v_1=(b_3+b_{-3})/\sqrt{2}$ and $v_2=(b_4+b_{-4})/\sqrt{2}$. 
\begin{definition}
\label{defn-of-beta}
Let $[T_1,T_2]\in V_{2,2}(\R)^{\oplus 2}$. We say $[T_1, T_2]$ is positive semi-definite if $\beta_{[T_1,T_2]}(r)\neq 0$ for all $r\in M(\R)^0$. We write $[T_1,T_2]\succeq 0$ to mean that the pair $[T_1,T_2]$ is positive semi-definite. Moreover, we write $[T_1,T_2]\succ 0$ if $[T_1,T_2]\succeq 0$ and $(T_1,T_1)(T_2,T_2)-(T_1,T_2)^2>0$. 
\end{definition}
 Given $r\in M_P(\R)$ we write the image of $m$ under the projection $p\colon G(\R)\to \SO(V)(\R)^0$ as $p(r)=(m,h)$ with the understanding that $m\in \GL(U)(\R)$ and $h\in \SO(V_{2,2})(\R)^0$.  \\
 \indent Modulo a technicality regarding the difference between $G$ and $G^{\mathrm{ad}}$, the following result is proven in \cite{pollackQDS}. With that said, a version of Theorem \ref{Thm 1.2.1 Aaron Paper} is available in \cite[Theorem 16]{wallach}, though (loc. cit.) does not establish \eqref{K-Bessel-Magic}, which is of crucial importance to us.
\begin{theorem}\label{Thm 1.2.1 Aaron Paper}
Fix $\ell\in \Z_{\geq 1}$ and suppose $[T_1,T_2]\in V_{2,2}(\R)^{\oplus2}$. Then up to scalar multiple, there is a unique moderate growth function $\mathcal{W}_{[T_1,T_2]}\colon G(\R)\to \mathbf{V}_{\ell}$ such that:
\begin{compactenum}
\item If $g\in G(\R)$ and $k\in K_{\infty}$ then $\mathcal{W}_{[T_1,T_2]}(gk)=k^{-1}\mathcal{W}_{[T_1,T_2]}(g)$.
\item If $g\in G(\R)$ and $n \in N_P(\R)$ satisfy $\log(n) = b_1 \wedge w_1 + b_2 \wedge w_2 + z b_1 \wedge b_2$ for $w_1,w_2\in V_{2,2}(\R)$ and $z\in \R$, then $\mathcal{W}_{[T_1,T_2]}(ng)=e^{i(T_1,w_1) + i (T_2,w_2)}\mathcal{W}_{[T_1,T_2]}(g)$.
\item If $g\in G(\R)$ then $D_{\ell}\mathcal{W}_{[T_1,T_2]}(g)=0$.  
\end{compactenum}
Moreover $\mathcal{W}_{[T_1,T_2]}(g)\equiv 0$ unless $[T_1,T_2]\succeq 0$, and if $[T_1,T_2]\succeq 0$ then the function $\mathcal{W}_{[T_1,T_2]}(g)$ is uniquely characterized by requiring that for all $r\in M_P(\R)^0$,
\begin{equation}
\label{K-Bessel-Magic}
\mathcal{W}_{[T_1,T_2]}(r)=\det(m)^{\ell}|\det(m)|\sum_{-\ell\leq v\leq \ell}\left(\frac{\beta_{[T_1,T_2]}(r)}{|\beta_{[T_1,T_2]}(r)|}\right)^vK_v(|\beta_{[T_1, T_2]}(r)|)\frac{x^{\ell+v}y^{\ell-v}}{(\ell+v)!(\ell-v)!}.
\end{equation}
Here $K_{v}\colon \R_{>0}\to \R$ denotes the modified K-Bessel function $K_v(x)=\dfrac{1}{2}\displaystyle{\int_0^{\infty}}t^{v-1}e^{-x(t+t^{-1})}\,dt$.
\end{theorem}
\begin{proof}
The statement is identical to that of \cite[Theorem 4.5]{JMNPR24} except for the fact $G=\Spin(V)$ as opposed to $\SO(V)$.  However,  since the central $\mu_2$ in $\Spin(V)$ is contained in $K_{\infty}$ and this subgroup acts trivially on $\V_{\ell}$, property (1) of the theorem statement implies that $\mathcal{W}_{[T_1,T_2]}$ factors across $\SO(V)(\R)^0$.  Hence, the result follows directly from (loc. cit.).
\end{proof}
As a corollary to Theorem \ref{Thm 1.2.1 Aaron Paper} we have the following.
\begin{corollary}
\label{cor:FE of phiZ}
Suppose $\ell\in \Z_{\geq 1}$ and let $\varphi\colon G(\R)\to \mathbf{V}_{\ell}$ be a weight $\ell$ quaternionic modular form on $G(\A)$.  If $[T_1,T_2]\in V_{2,2}(\Q)^{\oplus 2}$ is non-zero and not positive semi-definite then $\varphi_{[T_1,T_2]}\equiv 0$. Moreover, there exists a unique family of locally constant functions  
$$
\{a_{[T_1,T_2]}(\varphi,\cdot)\colon G(\A_f)\to \C\colon \hbox{$[T_1,T_2]\in V_{2,2}(\Q)^{\oplus 2}$ such that $[T_1,T_2]\succeq 0$}\}
$$ such that $\varphi_{[T_1,T_2]}(g_fg_{\infty})=a_{[T_1,T_2]}(\varphi,g_f)\mathcal{W}_{[2\pi T_1,2\pi T_2]}(g_{\infty})$ for all $[T_1,T_2]\succeq 0$. In particular, the Fourier expansion of $\varphi_Z$ along $Z(\A)N_P(\Q)\backslash N_P(\A)$ takes the form
\begin{equation}
\label{FE-non-cuspidal-QMF-phiZ} 
\varphi_{Z}(g_fg_{\infty})=\varphi_{N_P}(g_fg_{\infty})+\sum_{T_1,T_2\in V_{2,2}(\Q) \colon [T_1,T_2]\succeq 0}a_{[T_1,T_2]}(\varphi,g_f)\mathcal{W}_{[2\pi T_1,2\pi T_2]}(g_{\infty}).  
\end{equation}
\end{corollary}
\begin{definition}
Suppose $\varphi\in M_{\ell}$ and $B=[T_1,T_2]\in V_{2,2}(\Q)^{\oplus 2}$ satisfies $B]\succeq 0$. The Fourier coefficients of $\varphi$ indexed by $B$ is defined as $\Lambda_{\varphi}(B)=\Lambda_{\varphi}[T_1,T_2]=a_{[T_1,T_2]}(\varphi,1)$.  
\end{definition}
\indent The constant term $\varphi_{N_P}$ is essentially a holomorphic modular form on $M_{P}$. More precisely, following \cite[Proposition 11.1.1]{pollackQDS},  the function 
\begin{equation}
\label{eqn-structure-of-Heisenberg-constant-term}
\Phi\colon M_P^{\mathrm{der}}(\R) \to \C,  \quad \Phi(m)=j_{M_P^{\mathrm{der}}}(m,(i,i,i))^{\ell}\left\{ \varphi_{N_P}(m), y^{2\ell}\right\}_{K_{\infty}}
\end{equation}
descends to a holomorphic modular form on $\h_{M_P^{\mathrm{der}}}$ in the sense of Definition \ref{defn-holomorphic-modular-forms-on-MP}. Here $\{\, , \,\}_{K_{\infty}}$ denotes the unique $K_{\infty}$ invariant symmetric bilinear form on $\V_{\ell}$ satisfying
$$
\{ x^{\ell+v}y^{\ell-v}, x^{\ell-w}y^{\ell+w}\}_{K_{\infty}}=(-1)^{\ell+v}\delta_{v,w}(\ell+v)!(\ell-v)!
$$
where $\delta_{v,w}$ denotes the Kronecker delta function. \\
\indent 
Next we recall a result of  \cite[\S 4.4]{JMNPR24},  which establishes a positive semi-definiteness support property for the Fourier coefficients of quaternionic modular forms on $G$.  
 \begin{proposition}\emph{\cite[Proposition 4.9]{JMNPR24}}
\label{Properties-of-beta}
Let $T_1, T_2\in V_{2,2}(\R)$, $W=\R\linspan\{T_1,T_2\}$, and $V^+_2(\R)=\R\linspan\{v_1, v_2\}$.
\begin{compactenum}[(i)] 
\item If $\R\linspan\{T_1,T_2\}$ is an indefinite two plane,  a negative definite two plane, or a negative definite line, then there exists $r\in M_P(\R)^0$ such that $\beta_{[T_1,T_2]}(r)=0$. 
\item If $|\beta_{[T_1,T_2]}(r)|$ is bounded away from zero on $M_R^{\mathrm{der}}(\R)$ then $(T_1,T_1)(T_2,T_2)-(T_1,T_2)^2>0$. In particular $T_1$ and $T_2$ span a two plane in $V_{2,2}(\R)$. 
\end{compactenum}
\end{proposition}
\indent We conclude with an application of Proposition \ref{Properties-of-beta} to the study of cusp forms $\varphi\in S_{\ell}$, which will be applied during the proof of Corollary \ref{slice-primitivity-corollary}.
\begin{corollary}
\label{FE-CUSPIDAL-QMF} 
 Suppose $\ell\in \Z_{\geq 1}$ and $\varphi\colon  G(\A)\to \mathbf{V}_{\ell}$ is a weight $\ell$ quaternionic modular form.  Then $\varphi$ is a cusp form if and only if the Fourier expansion \eqref{FE-non-cuspidal-QMF-phiZ} takes the form 
  \begin{equation}
  \label{equation-FE-CUSPIDAL-QMF}
\varphi_{Z}(g_fg_{\infty})=\sum_{T_1,T_2\in V_{2,2}(\Q) \colon [T_1,T_2]\succ 0}a_{[T_1,T_2]}(\varphi,g_f)\mathcal{W}_{[2\pi T_1,2\pi T_2]}(g_{\infty}).  
\end{equation}
\end{corollary}
\begin{proof}
For the proof of the forward implication, we refer the reader to \cite[Corollary 4.10]{JMNPR24}.  To handle the backward implication, suppose $\varphi$ is such that $\varphi_{N_P}\equiv 0$ and $a_{[T_1,T_2]}(\varphi,g_f)\equiv 0$ whenever $[T_1,T_2]$ does not satisfy $[T_1,T_2]\succ 0$.  \\ 
\indent The group $G$ contains $4$ conjugacy classes of maximal parabolic subgroups.  Two of these conjugacy classes are represented by the Heisenberg parabolic subgroup $P$,  and the orthogonal parabolic subgroup $R$.  Since $R$ corresponds to an outer node in the Dynkin diagram of $G$, representatives for the remaining two conjugacy classes may be obtain by translating $R$ by the triality outer automorphisms described in \cite[Theorem A.8]{JMNPR24}.  By \cite[Theorem A.1]{JMNPR24}, the triality outer automorphism of $G$ stabilize $N_P$, and induce an action of the symmetric group $S_3$ on the space $M_{\ell}$. Moreover, the induced action of $S_3$ on $M_{\ell}$ preserves the subspace consisting of forms whose Fourier coefficients are supported on positive definite indices.   Hence,  to prove that $\varphi$ is cuspidal, it suffices to show that $\varphi_{N_P\cap N_R}\equiv 0$.  \\
\indent Writing $g\in G(\A)$ as $g=g_fg_{\infty}$ , a standard manipulation gives that 
\begin{equation}
\label{Eqn-constant-term analysis}
\varphi_{N_P\cap N_R}(g_fg_{\infty})=\sum_{w\in V_{2,2}(\Q)^{\oplus 2} \colon \varepsilon_w\rvert_{N_P\cap N_R}=1}a_w(\varphi,g_f)\mathcal{W}_{2\pi w}(g_{\infty}). 
\end{equation}
Since $\Lie(N_P\cap N_R)$ contains the subspace $b_1\wedge V_{2,2}$,  if $w=[T_1,T_2]$ satisfies $\varepsilon_w\rvert_{N_P\cap N_R}=1$, then $T_1=0$.  Hence, the only terms appearing in \eqref{Eqn-constant-term analysis} are those for which $w=[T_1,T_2]$ does not satisfy $[T_1,T_2]\succ 0$.  Therefore,  $\varphi_{N_P\cap N_R}\equiv 0$, and $\varphi$ is cuspidal. 
\end{proof} 
\subsection{The Hecke Bound for Quaternionic Modular Forms on $G$}
If $B=[T_1,T_2]\in V_{2,2}(\Q)^{\oplus 2}$, we define 
$$
Q(B)=\det\left(\begin{pmatrix} (T_1,T_1) &(T_1,T_2) \\ (T_2,T_2) &(T_2,T_2)\end{pmatrix}\right). 
$$
Then $r\in M_P$ acts on $V_{2,2}(\Q)^{\oplus 2}$ preserving $Q$ up to a similitude character $\nu \colon M_P\to \G_m$ i.e. 
$$
Q(r\cdot B)=\nu(r)^2Q(B). 
$$
To define $\nu$, write the image of $r\in M_P$ under the projection $M_P\to \GL(U)\times \SO(V_{2,2})$ as $(m,h)$. Then $\nu(r)=\det(m)$.  The purpose of this subsection is to prove an analogue of the Hecke bound for quaternionic cusp forms on $G$.  Our proof is adapted from \cite[Proposition 8.6]{ganGrossSavin}, in which the authors prove an analogous bound for modular forms on $G_2$. 
\begin{prop}
\label{hecke-bound-quaternionic-modular-forms}
Suppose $\varphi\in S_{\ell}$ is a weight $\ell$ cuspidal quaternionic modular form on $G$. If $B\in V_{2,2}(\Q)^{\oplus 2}$ satisfies $B\succ 0$, then 
$$
\Lambda_{\varphi}[B]\ll_{\varphi} Q(B)^{\frac{\ell+1}{2}}. 
$$
\end{prop}
\begin{proof}
Assume $B\in V_{2,2}(\Q)^{\oplus 2}$ satisfies $B\succ 0$. Independent of the choice of $B$, we fix an element $B_0\in V_{2,2}(\R)^{\oplus 2}$ such that $Q(B_0)=1$.  The group $M_P(\R)$ acts transitively on the elements $B'\in V_{2,2}(\R)^{\oplus 2}$ such that $Q(B')>0$. Therefore, we may fix an element $m_0\in M_P(\R)$ such that $m_0\cdot B_0=B$.  It follows that $Q(B)=\nu(m_0)^2$.  \\
\indent By definition,  if $m_{\infty}\in M_P(\R)$, then 
\begin{equation}
\label{eqn-int-expression}
\int_{[N_P]}\varphi(nm_{\infty})\varepsilon_{B}(n)^{-1}\,dn=\Lambda_{\varphi}[B]\cdot \mathcal{W}_B(m_{\infty})
\end{equation}
where $\mathcal{W}_B$ is given by \eqref{K-Bessel-Magic}.  Fix an element $u_0\in \V_{\ell}$ so that if $\{\cdot , \cdot \}_{K_\infty}$ is the $K_{\infty}$ invariant bilinear form on $\V_{\ell}$ then  
$$
\{x^{\ell-v}y^{\ell+v}, u_0\}_{K_{\infty}}=\begin{cases} 2\ell !&\hbox{if $v=0$,} \\ 0 &\hbox{else.} \end{cases}
$$
Since $M_P$ is a split, connected, and reductive,  $M_P$ is an almost direct product of $M_P^{\mathrm{der}}$ and the radical of $M_P$,  $R(M_P)$, which is a central torus \cite[Corollary 8.1.6]{MR1642713}. As $M_P$ has rank $4$, and $M_P^{\mathrm{der}}=\SL_2\times \SL_2\times \SL_2$,  $R(M_P)$ is isomorphic to the multiplicative group $\G_m$.  We choose this isomorphism so that the natural action of $R(M_P)$ on $U^{\vee}\otimes V_{2,2}$ is identified with the action of $\G_m$ on $U^{\vee}\otimes V_{2,2}$ by scalar multiplication.  In this way, we may write $\nu(m_0)^{-1}m_0\in M_P(\R)$ and 
\begin{align*}
\left\{\mathcal{W}_B(\nu(m_0)^{-1}m_0), u_0\right\}_{K_\infty}
&=\nu(\nu(m_0)^{-1}m_0)^{\ell}|\nu(\nu(m_0)^{-1}m_0)| K_0\left(|\beta_{B}(\nu(m_0)^{-1}m_0)|\right) \\
&= Q(B)^{\frac{-\ell-1}{2}}\left\{\mathcal{W}_{B_0}(1), u_0\right\}_{K_\infty}.
\end{align*}
Since $\varphi$ is cuspidal, $|\varphi(g)|$ is bounded on $G(\Q)\backslash G(\A)$.  Therefore, since the domain of integration in  \eqref{eqn-int-expression} is compact,  we may set $m_{\infty}=\nu(m_0)^{-1}m_0$ to obtain 
$$
|\Lambda_{\varphi}[B]|\leq \frac{\|\varphi\|_{\infty}}{\left\{\mathcal{W}_{B_0}(1), u_0\right\}_{K_\infty}} Q(B)^{\frac{\ell+1}{2}}.
$$
\end{proof}
\section{The Fourier-Jacobi Expansion of Modular Forms on $G$}
\label{chap:THE FOURIER-JACOBI EXPANSION OF MODULAR FORMS ON $G$}
Given a vector valued automorphic form $\varphi\colon G(\A) \to \mathbf{V}$, and a vector $y\in V_{3,3}(\Q)$, define
$$
\mathcal{F}(\varphi;y)\colon G(\A)\to \mathbf{V}, \qquad g\mapsto \int_{N_R(\Q)\backslash N_R(\A)}\varphi(ng)\chi_y(n)^{-1}\,dn. 
$$
Here the the character $\chi_y$ is as defined in \eqref{eqn-character-lattice-R}. 
\subsection{Unfolding $\mathcal{F}(\varphi;y)$: The case of isotropic $y$}
\label{subsec-degenerate-terms}
Suppose $\varphi$ is a vector valued automorphic function of $G(\A)$ and let $y\in V_{3,3}(\Q)$ be non-zero and isotropic.  
In this subsection,  we establish  Lemma \ref{lemma-degenerate-FCs},  which gives a relationship between the degenerate Fourier-Jacobi coefficient $\mathcal{F}(\varphi;y)$, and Fourier coefficients of $\varphi$ along the Heisenberg unipotent radical $N_P$.  \\
\indent We are considering the case of isotropic $y$, and since $M_R^{\mathrm{der}}\simeq \Spin(V_{3,3})$,  we restrict to considering the case of $y=nb_2$ where $n\in \Z\backslash\{0\}$.  Then $\mathcal{F}(\varphi;y)\rvert_{G(\R)}$ is left invariant under the rational points of the parabolic subgroup in $\Spin(V_{3,3})$ stabilizing $b_2$.  Let $N$ denote the unipotent radical  of this parabolic subgroup.  So, 
$$
N\simeq\left\{
=\left(\tiny
\begin{matrix} 
1&0 &\mathbf{0} &0 &0 \\ 
0 &1 &v &\ast &0  \\ 
\mathbf{0} &\mathbf{0} &I_{4} &\ast &\mathbf{0} \\
0 &0 &\mathbf{0} &1 &0 \\ 
0 &0 &\mathbf{0} &0 &1 
\end{matrix}
\right)\in \SO(V)
\colon v\in V_{2,2}\right\},
$$
and $\mathrm{Hom}(N(\Q)\backslash N(\A), \C^{\times})=\{\varepsilon_{[0,T]}\colon T\in V_{2,2}(\Q)\}$,
where $\varepsilon_{[0,T]}$ is defined in \eqref{eqn defn epsilon}.
Since $N$ is abelian, we may Fourier expand $\mathcal{F}(\varphi;nb_2)$ in terms of the coefficients 
\begin{equation}
\label{eqn-defn-FC-of-degenerate-FJ-coefficient}
\mathcal{F}^N_T(\varphi;nb_2)(g):=\int_{N(\Q)\backslash N(\A)}\mathcal{F}(\varphi;nb_2)(rg)\varepsilon_{[0,T]}^{-1}(r)\,dr.
\end{equation}
Here $g\in G(\A)$ and $
T\in V_{2,2}(\Q)$.
Regarding $\mathcal{F}^N_T(\varphi;nb_2)$ we have the following.  
\begin{lemma}
\label{lemma-degenerate-FCs}
Suppose $g\in G(\A)$ and $T\in V_{2,2}(\Q)$. Let $\varepsilon_{[0,T]}$ be the character of $N_P(\Q)\backslash N_P(\A)$ defined in \eqref{eqn defn epsilon} and write $\varphi_{[0,T]}(g)=\displaystyle{\int_{N_P(\Q)\backslash N_P(\A)}}\varphi(ng)\varepsilon_{[0,T]}(n)^{-1}\,dn$. Then 
    \begin{equation}
    \label{eqn-FC-of-degenerate-FJ-coefficient}
   \mathcal{F}^N_T(\varphi;nb_2)(g)=\int_{\Q\backslash \A}\psi^{-1}(ns)\varphi_{[0,T]}(\exp(sb_1\wedge b_{-2})g)\,ds.
    \end{equation}
\end{lemma}
\begin{proof}
Throughout the proof we abbreviate notation by writing $[\mathcal{G}]:=\mathcal{G}(\Q)\backslash \mathcal{G}(\A)$ to denote the adelic quotient of an algebraic group $\mathcal{G}$. 
    Plugging in definitions, one obtains  
    \begin{equation}
    \label{lemma-degerate-FCs-First-Step}
    \mathcal{F}^N_T(\varphi;n b_2)(g)=\int_{[N]}\int_{[N_R]}\varphi(urg)\chi_{nb_2}^{-1}(u)\varepsilon_{[0,T]}^{-1}(r)\,du\,dr.
    \end{equation}
    The character $\chi_{nb_2}$ is trivial on $Z\leq N_R$. Since $N_R$ is abelian, we may thus factor the inner integral in \eqref{lemma-degerate-FCs-First-Step} across $Z$ to obtain
    \begin{equation}
    \label{lemma-degerate-FCs-Second-Step}
   \mathcal{F}^N_T(\varphi;n b_2)(g)=\int_{[N]}\int_{[N_R/Z]}\varphi_Z(urg)\chi_{nb_2}^{-1}(u)\varepsilon_{[0,T]}^{-1}(r)\,du\,dr.
    \end{equation}
    Let $X=M_P\cap N_R=\{\exp(tb_1\wedge b_{-2})\colon t\in \G_a\}$. Since $N_R$ is abelian, the integral over $[N_R/Z]$ in \eqref{lemma-degerate-FCs-Second-Step} is an iterated integral over $X(\Q)\backslash X(\A)$ and $[N_R/(ZX)]$. Moreover, the commutator $[X,N]$ is contained in $Z$. Hence, moving the integral over $[X]$ to the left of the integration over $[N]$,  one arrives at the expression
    \begin{equation}
    \label{lemma-degerate-FCs-Third-Step} 
   \mathcal{F}^N_T(\varphi;n b_2)(g)=\int_{[X]}\chi_{nb_2}^{-1}(x)\int_{[N]}\int_{[N_R/(ZX)]}
    \varphi_Z(xurg)\varepsilon_{[0,T]}^{-1}(r)\,du\,dr\,dx.
    \end{equation}
    Since $N\times N_R/(ZX)=N_P^{\mathrm{ab}}$,  \eqref{lemma-degerate-FCs-Third-Step} may be rewritten as 
    \begin{equation}
    \label{lemma-degerate-FCs-Fourth-Step} 
    \mathcal{F}^N_T(\varphi;n b_2)(g)=\int_{[X]}\chi_{nb_2}^{-1}(x)\int_{[N_P^{\mathrm{ab}}]}
    \varphi_Z(xng)\varepsilon_{[0,T]}^{-1}(n)\,dn\,dx.
    \end{equation} 
   Fourier expanding $\varphi_Z$ along $N_P^{\mathrm{ab}}(\Q)\backslash N_P^{\mathrm{ab}}(\A)$, we apply the inclusion $X\leq M_P$ and the fact that $M_P$ normalized $N_P$ to simplify the inner integral in \eqref{lemma-degerate-FCs-Fourth-Step} as 
    \begin{align}
    \label{lemma-degerate-FCs-Sixth-Step}
    \int_{[N_P^{\mathrm{ab}}]}
    \varphi_Z(xng)\varepsilon_{[0,T]}^{-1}(n)\,dn 
    &=\int_{[N_P^{\mathrm{ab}}]}\sum_{T_1,T_2\in V_{2,2}(\Q)}\varphi_{[T_1,T_2]}(xng)\varepsilon_{[0,T]}^{-1}(n)\,dn \notag \\ 
    &=\sum_{T_1,T_2\in V_{2,2}(\Q)}\varphi_{[T_1,T_2]}(xg)\int_{[N_P^{\mathrm{ab}}]}\varepsilon_{[T_1,T_2]}(xnx^{-1})\varepsilon_{[0,T]}^{-1}(n)\,dn.
    \end{align}
    Fix $s\in \A$, $x=\exp(sb_{1}\wedge b_{-2})$, and suppose $n=\exp(b_1\wedge v+b_2\wedge v')$ with $v,v'\in V_{2,2}(\A)$. Then $xnx^{-1}=\exp(b_1\wedge (v+sv')+b_2\wedge v)$ and so if $T_1,T_2\in V_{2,2}(\Q)$, 
    \begin{equation}
    \label{lemma-degerate-FCs-Fifth-Step}
   \int_{[N_P^{\mathrm{ab}}]}\varepsilon_{[T_1,T_2]}(xnx^{-1})\varepsilon_{[0,T]}^{-1}(n)\,dn
   =
   \int_{[V_{2,2}]}\int_{[V_{2,2}]}\psi((T_1,v+sv')+(T_2,v)-(v',T))\,dv\,dv'. 
    \end{equation}
   The double integral above is non-zero if and only if $(T_1,v+sv')+(T_2,v')-(v',T)=0$ for all $v,v'\in V_{2,2}(\A)$. Setting $v'=0$ this condition reduces to $(T_1,v)=0$ for all $v\in V_{2,2}(\A)$. Since $V_{2,2}$ is non-degenerate, it follows that the integral \eqref{lemma-degerate-FCs-Fifth-Step} is non-vanishing if and only if $T_1=0$ and $T_2=T$. Thus the sum in \eqref{lemma-degerate-FCs-Sixth-Step} reduces to a single term equal to $\varphi_{[0,T]}(xg)$. Therefore, \eqref{lemma-degerate-FCs-Fourth-Step} simplifies to
\eqref{eqn-FC-of-degenerate-FJ-coefficient}
as required.
\end{proof}
  \subsection{Unfolding $\mathcal{F}(\varphi;y)$: The case of non-isotropic $y$}
\label{subsec:nondegenerate}
Let $\varphi$ is a vector valued automorphic function on $G(\A)$,  and suppose $y\in V_{3,3}(\Q)$ is non-zero and non-isotropic. 
The goal of this subsection is to give companion result to Lemma \ref{lemma-degenerate-FCs}, that analyses $\mathcal{F}(\varphi;y)$ in the case when $y$ is non-isotropic.  By the result of \cite[Lemma B.2]{JMNPR24}, the orbits of $M_R^{\mathrm{der}}(\Z)$ on the space of non-isotropic vectors in $V_{3,3}(\Z)$ are exhausted by representatives 
\begin{equation}
\label{eqn-defn-y_n,alpha}
ny_{\alpha}:=nb_3+\dfrac{n\alpha}{2}b_{-3}
\end{equation}
where $n\in \Z_{\geq 1}$ and $\alpha \in 2\Z\backslash\{0\}$.   For ease of notation we write $y_{\alpha}=b_3+\alpha b_{-3}/2$ \\
\indent 
To analyze $\mathcal{F}(\varphi;ny_{\alpha})$,  we study its Fourier expansion in characters of the $[N_{Q'}]$. Here $Q'$ denotes the Siegel parabolic subgroup of $M'=\mathrm{Stab}_{M_R^{\mathrm{der}}}(y)$ (see Subsection \ref{subsec-The-Siegel-Parabolic-$Q'$}).  \\
\indent Since $M'$ stabilizes $ny_{\alpha}$, the coefficient $\mathcal{F}(\varphi;ny_{\alpha})$ is left invariant by $N_{Q'}(\Q)$.  Hence, we may Fourier expand $\mathcal{F}(\varphi;ny_{\alpha})$ in characters of $N_{Q'}(\Q)\backslash N_{Q'}(\A)$ as
$$
\mathcal{F}(\varphi;ny_{\alpha})(g)=\sum_{S\in V'_{1,2}(\Q)}\mathcal{F}^{Q'}_S(\varphi;ny_{\alpha})(g)
$$
where $g\in G(\A)$ and 
$$\mathcal{F}^{Q'}_S(\varphi;ny_{\alpha})(g):=\displaystyle{\int_{N_{Q'}(\Q)\backslash N_{Q'}(\A)}}\mathcal{F}(\varphi;ny_{\alpha})(wg)\varepsilon_{[0,S]}^{-1}(w)\,dw.$$
In analogy with Lemma \ref{lemma-degenerate-FCs}, we have the following relationship between the Fourier coefficients $\mathcal{F}^{Q'}_S(\varphi;ny_{\alpha})(g)$ and Fourier coefficients of $\varphi$ along the $N_P$.
\begin{lemma}
\label{lemma-non-degerate FCs}
Let $\alpha\in 2\Z\backslash\{0\}$ and $n\in \Z_{\geq 1}$.
    Suppose $S\in V_{1,2}'(\Q)$ and $g\in G(\A)$.  Let $\varphi_{[ny_{\alpha},S]}$ denote the Fourier coefficient of $\varphi$ along $N_P$ corresponding to the character $\varepsilon_{[ny_{\alpha},S]}$.  Then 
    $$
    \mathcal{F}^{Q'}_S(\varphi;ny_{\alpha})(g)=\int_{\A}\varphi_{[ny_{\alpha},S]}(\exp(sb_1\wedge b_{-2})g)\, ds, \qquad (g\in G(\A)). 
    $$
\end{lemma}
\begin{proof}
The statement is a mild generalization of a \cite[Lemma 7.1]{JMNPR24} with essentially the same proof. The details of the general proof are available in \cite[Lemma 4.0.3]{McG25}.
\end{proof}
\subsection{The Fourier-Jacobi Expansion of Quaternionic Modular Forms on $G$}
\label{subsec:The Fourier-Jacobi Expansion of Quaternionic Modular Forms on G}
In this subsection we combine Lemma {\ref{lemma-degenerate-FCs} and Lemma \ref{lemma-non-degerate FCs} with the results of Subsection \ref{subsec: Quaternionic Modular Forms on G} to refine the Fourier-Jacobi expansion, \begin{equation}
\label{eqn FJ expansion level N}
\varphi(g)=\varphi_{N_R}(g)+\sum_{y\in V_{3,3}(\Z) \colon y\neq 0}\mathcal{F}(\varphi;y)(g)., 
\end{equation}
 in the case when $\varphi\in M_{\ell}(1)$.
\begin{proposition}
\label{prop-vanishing-of-negative-FJ-coefficients}
 Suppose $\varphi\in  M_{\ell}(1)$ is a weight $\ell$ quaternionic modular form of level $1$. Then the Fourier-Jacobi expansion \eqref{eqn FJ expansion level N} takes the form 
    \begin{equation}
    \label{FJ-expansion-of-general-QMF}
    \varphi(g_{\infty})=\varphi_{N_R}(g_{\infty})+\sum_{y\in V_{3,3}(\Z)\backslash\{0\} \colon (y,y)\geq 0}\mathcal{F}(\varphi;y)(g_{\infty})
    \end{equation}
    for all $g_{\infty}\in G(\R)$. Moreover, if $\varphi\in S_{\ell}(1)$, then for $g_{\infty}\in G(\R)$, \eqref{FJ-expansion-of-general-QMF} takes the form 
    \begin{equation}
    \label{FJ-expansion-of-cuspidal-QMF}
    \varphi(g_{\infty})=\sum_{y\in V_{3,3}(\Z)\colon (y,y)> 0}\mathcal{F}(\varphi;y)(g_{\infty}).
    \end{equation}
\end{proposition}
\begin{proof}
Since $\varphi$ has level one, to prove that $\mathcal{F}(\varphi;y)\rvert_{G(\R)}\equiv 0$ for a given $y\in V_{3,3}(\Z)$,  it suffices to show that $\mathcal{F}(\varphi;y')\rvert_{G(\R)}\equiv 0$ for any $y'$ in the same $M_R^{\mathrm{der}}(\Z)$ orbit of $y$. \\
\indent 
    Initially, we suppose $\varphi\in M_{\ell}(1)$ and $y\in V_{3,3}(\Z)$ satisfies $\langle y,y\rangle<0$. Then by \cite[Lemma B.2]{JMNPR24}, there exist $\alpha\in 2\Z_{<0}$ and $n\in \Z_{\geq 1}$ such that $y$ is in the same $M_R^{\mathrm{der}}(\Z)$ orbit as $ny_{\alpha}=nb_3+n\alpha b_{-3}/2$.
  Combining Lemma \ref{lemma-non-degerate FCs} with Corollary \ref{cor:FE of phiZ},  if $S\in  V'_{1,2}(\Q)$ then $\mathcal{F}^{Q'}_S(\varphi;ny_{\alpha})(g_{\infty})$ is equal to
  \begin{equation}
  \label{eqn-factorization-non-degenerate-FJ}
  \int_{\A_f}a_{[ny_{\alpha},S]}(\varphi,\exp(s_fb_1\wedge b_{-2}))\,ds_f\cdot\int_{\R}\mathcal{W}_{[2\pi n y_{\alpha},2\pi S]}(\exp(s_{\infty}b_1\wedge b_{-2})g_{\infty})\,ds_{\infty}. 
  \end{equation}
  Since $ny_{\alpha}$ is a vector of negative norm, Proposition \ref{Properties-of-beta}(a) implies that the pair $[ny_{\alpha},S]$ is not positive semi-definite. Therefore,  $\mathcal{W}_{[2\pi n y_{\alpha},2\pi S]}\equiv 0$ according to Theorem \ref{Thm 1.2.1 Aaron Paper},  and $\mathcal{F}(\varphi;ny_{\alpha})$ vanishes identically on $G(\R)$. Hence, we have proven expression \eqref{FJ-expansion-of-general-QMF}. \\
   \indent We now address the proof of \eqref{FJ-expansion-of-cuspidal-QMF}. Thus suppose $\varphi$ is cuspidal. It remains to show that $\mathcal{F}(\varphi;y)$ vanishes identically on $G(\R)$ whenever $y\in V_{3,3}(\Z)$ is isotropic.  For this,  it suffices to show that $\mathcal{F}(\varphi;nb_2)$ vanishes identically on $G(\R)$ for all $n\in \Z_{\neq 0}$.  Applying Lemma \ref{lemma-degenerate-FCs},  $\mathcal{F}(\varphi;nb_2)$ vanishes identically on $G(\R)$ provided $\varphi_{[0,T]}$ vanishes identically on $G(\R)$ for all $T\in V_{2,2}(\Q)$.  When $T=0$, $\varphi_{[0,T]}=\varphi_{N_P}$ vanishes since $\varphi$ is cuspidal. If $T\neq 0$, then a consequence of definition \ref{defn-of-beta}, $[0,T]$ does not satisfy $[0,T]\succ 0$. Hence, the vanishing of $\varphi_{[0,T]}$ follows from Corollary \ref{FE-CUSPIDAL-QMF}.  
\end{proof}
\subsection{The Degenerate Fourier-Jacobi Coefficients of Modular Forms on G}
\label{subsec:Degenerate Fourier-Jacobi Coefficients of Modular Forms on G}
In this subsection, we further refine Lemma \ref{lemma-degenerate-FCs} in the case when $\varphi$ is a weight $\ell$ quaternionic modular form of level $1$.  The main results are Proposition \ref{corollary-Degenerate-QMF-FJs} and Corollary \ref{cor-FCs-constant-degenerate-Heisenberg-FCS-D4}.  The proofs of these result use Lemma \ref{lemma-left-invariance-by-X(A)} and Proposition  \ref{prop-FCs-of-Heisenberg-constant-term}, for which the reader should consult Section \ref{sec-The Hecke Bound Characterization of Cusp Form on $G$}.
\begin{prop}
\label{corollary-Degenerate-QMF-FJs}
Let $n\in \Z_{\neq 0}$,  and $g_{\infty}\in G(\R)$.  Recall the Fourier coefficients $\mathcal{F}_T^N(\varphi;nb_2)(g)$ defined in \eqref{eqn-defn-FC-of-degenerate-FJ-coefficient}.  If $T\in V_{2,2}(\Q)\backslash \{0\}$  then $\mathcal{F}_T^N(\varphi;nb_2)(g_{\infty})=0$  for all $g_{\infty}\in G(\R)$, and
\begin{equation}
\label{eqn-Degenerate-QMF-FJs}
\mathcal{F}(\varphi;nb_2)(g_{\infty})=\int_{\Q\backslash \A}\varphi_{N_P}(\exp(sb_1\wedge b_{-2})g_{\infty})\psi^{-1}(ns)\, ds.
\end{equation}
\end{prop}
\begin{proof}
In light of Lemma \ref{lemma-degenerate-FCs} and the fact that $S$ is abelian, expression \eqref{eqn-Degenerate-QMF-FJs} follows immediately from the vanishing of $\mathcal{F}_T^N(\varphi;nb_2)(g_{\infty})$ for all non-zero $T\in V_{2,2}(\Q)$.  Thus suppose $T\in V_{2,2}(\Q)$ is non-zero and $s\in \A$.  One checks that for $m=\exp(sb_1\wedge b_{-2})$ and $[T_1,T_2]=[0,T]$,  the hypotheses of Lemma \ref{lemma-left-invariance-by-X(A)} are satisfied, and thus the term $\varphi_{[0,T]}(\exp(sb_1\wedge b_{-2})g_{\infty})$ appearing in the integrand of \eqref{eqn-FC-of-degenerate-FJ-coefficient} is independent of $s$.  It follows that
$$
\mathcal{F}_T^N(\varphi;nb_2)(g_{\infty})=\varphi_{[0,T]}(g_{\infty})\int_{\Q\backslash \A}\psi^{-1}(ns)\,ds=\varphi_{[0,T]}(g_{\infty})\cdot 0=0.
$$
\end{proof}
     The lemma above shows that the degenerate Fourier-Jacobi coefficient $\mathcal{F}(\varphi;nb_2)$ factors across the constant term $\varphi_{N_P}$. In Subsection \ref{subsec: Quaternionic Modular Forms on G}, we saw that $\varphi_{N_P}$ is closely related to a holomorphic modular form on the group $M_P^{\mathrm{der}}=\SL_2^3$.  As such, it is natural to consider Fourier coefficients of $\varphi_{N_P}$ along the unipotent radical of a Borel subgroup of $M_P^{\mathrm{der}}$. Using the isomorphism of \cite[Theorem A.8]{JMNPR24},  the unipotent radical of such a subgroup has Lie algebra
 $$
e_2\otimes E =\Q\linspan\{b_{-3}\wedge b_{-4},b_2\wedge b_{-1},b_3\wedge b_{-4}\}.
$$
     \indent 
     Given a triple of $C=(a,b,c)\in \Q^3$, we let $\eta_{C}$ denote the character of $[\exp(e_2\otimes E)]$ corresponding to the linear functional 
     $$
     xb_2\wedge b_{-1}+yb_{-3}\wedge b_{-4}+xb_3\wedge b_{-4}\mapsto bz-ay-cx.$$ Proposition \ref{prop-FCs-of-Heisenberg-constant-term} shows that the Fourier coefficient $(\varphi_{N_P})_{\exp(e_2\otimes E),C}\colon G(\A) \to \V_{\ell}$, defined by
\begin{equation}
\label{defn-FC-Heisenber-D4}
      (\varphi_{N_P})_{\exp(e_2\otimes E),C}(g)=\int_{[\exp(e_2\otimes E)]}\varphi_{N_P}(ug)\eta_{C}^{-1}(u)\,du,
  \end{equation} 
  is equal to an integral transform of a Fourier coefficient of $\varphi$ along $N_P$.  More precisely,  we have the following corollary to Proposition  \ref{prop-FCs-of-Heisenberg-constant-term}.
     \begin{corollary}
     \label{cor-FCs-constant-degenerate-Heisenberg-FCS-D4}
     Let $\varphi\in M_{\ell}(1)$ be a weight $\ell$ modular form of level $1$ of $G=\Spin(V)$.  As in \eqref{eqn-structure-of-Heisenberg-constant-term},  write $\Phi$ for the holomorphic modular form on $M_P^{\mathrm{der}}$ associated to the constant term $\varphi_{N_P}$. Folllowing the notation in \eqref{eqn-FE-3-modular-forms-on-SL2}, write the Fourier expansion of the classical modular form associated to $\Phi$ as 
\begin{equation}
\label{eqn-FE-3-modular-forms-on-d4}
     f_{\Phi}(z_1,z_2,z_3)=\sum_{C=(a,b,c)\in \Z^3_{\geq 0}} a_{\varphi}(C)e^{2\pi i (az_1+bz_2+cz_3)}. 
 \end{equation} Then there exists a non-zero scalar $B\in \C$ such that if $C\in \Z^3$, then 
     \begin{equation}
     \label{eqn-cor-FCs-constant-degenerate-Heisenberg-FCS}
     a_{\varphi}(C)=B\cdot \int_{S_{C}(\A_f)\backslash \exp(e_2\otimes E)(\A_f)}a_{[ab_3+cb_{-3},-bb_{-4}]}(\varphi, u)\, du .
     \end{equation}
     Here $S_C$ denotes the stabilizer of the character $\varepsilon_{[ab_3+cb_{-3},-bb_{-4}]}$ in $\exp(e_2\otimes E)$. 
     \end{corollary}
     \begin{proof}
     This is a restatement of the equality \eqref{eqn-equality-of-finite-adelics} in the case $G_J=G^{\mathrm{ad}}$ and $g_f=1$.  The coefficient $a_{[ab_3+cb_{-3},-bb_{-4}]}$ is right invariant by $n_{\alpha}$ since $\varphi$ is of level one. 
     \end{proof}
\subsection{The Non-Degenerate Fourier-Jacobi Coefficients of Modular Forms on G}
\label{subsec:Non-Degenerate Fourier-Jacobi Coefficients of Modular Forms on G}
Our goal in this subsection is to refine the result of Lemma \ref{lemma-non-degerate FCs} in the special case when $\varphi\in M_{\ell}(1)$ is a modular form on $G$ of level $1$.  The main result is Proposition \ref{prop-existence-holomorphic-FJ-coefficient}, which shows that for such a $\varphi$, the non-degenerate Fourier-Jacobi coefficients of Lemma \ref{lemma-non-degerate FCs} give rise to holomorphic modular forms in the sense of Subsection \ref{subsec:HMFS}.  \\
\indent We adopt the notation of Lemma \ref{lemma-non-degerate FCs}, so $n\in \Z_{\geq 1}$, $\alpha\in 2\Z\backslash\{0\}$, and $ny_{\alpha}=nb_3+n\alpha b_{-3}/2$. By Proposition \ref{prop-vanishing-of-negative-FJ-coefficients}, if $\varphi\in M_{\ell}(1)$ then $\mathcal{F}(\varphi;ny_{\alpha})(g_{\infty})= 0$ whenever $\alpha<0$ and $g_{\infty}\in G(\R)$. Hence, we may assume $\alpha>0$, in which case $\mathcal{F}(\varphi;y)$ Fourier expands along $N_{Q'}$ as
\begin{equation}
\label{eqn-non-degenerate-FJ-coefficents-QMF-1}
\mathcal{F}(\varphi;ny_{\alpha})(g)=\sum_{S\in V'_{1,2}(\Q)}\mathcal{F}^{Q'}_S(\varphi;ny_{\alpha})(g). 
\end{equation}
In fact, if $S\in V'_{1,2}$ is a vector of negative norm, then the line of reasoning applied during the proof of Proposition \ref{prop-vanishing-of-negative-FJ-coefficients} implies that $\mathcal{F}^{Q'}_S(\varphi;ny_{\alpha})$ vanishes identically on $G(\R)$. Thus,  if $g=g_{\infty}\in G(\R)$, then \eqref{eqn-non-degenerate-FJ-coefficents-QMF-1} takes the form 
\begin{equation}
\label{FE-of-FJ-Coefficient}
\mathcal{F}(\varphi;ny_{\alpha})(g_{\infty})=\sum_{S\in V'_{1,2}(\Q)\colon (S,S)\geq 0}\mathcal{F}^{Q'}_S(\varphi;ny_{\alpha})(g_{\infty}). 
\end{equation}
The summation indices appearing in \eqref{FE-of-FJ-Coefficient} resemble the indices which appear in the Fourier expansion of holomorphic modular forms on $M'$ (see \eqref{FE-HMF-on-M'}). However, since $K_{\infty}\cap M'$  is not necessarily a maximal compact subgroup of $M'$, it is necessary to introduce a modified definition of $\mathcal{F}(\varphi;ny_{\alpha})$ in order to obtain holomorphic modular forms on $M'$ from $\mathcal{F}(\varphi;ny_{\alpha})$.  For this we recall the element $g_y\in G(\R)$ and the representation $\mathbf{V}_{\ell}^{\alpha}$ introduced in Subsection \ref{subsec:HMFS}.  Define
$$
\widetilde{\mathcal{F}}(\varphi;ny_{\alpha})\colon G(\A)\to \V^{\alpha}_{\ell}, \qquad g\mapsto \mathcal{F}(\varphi;ny_{\alpha})(gg_{y_{\alpha}}).
$$
Then $\widetilde{\mathcal{F}}(\varphi;ny_{\alpha})(gk)=k^{-1}\cdot \widetilde{\mathcal{F}}(\varphi;ny_{\alpha})(g)$ for all $g\in G(\A)$ and $k\in K_{y_{\alpha}}$. Furthermore,  Lemma \ref{lemma-non-degerate FCs} and Theorem \ref{Thm 1.2.1 Aaron Paper} imply that 
$$\widetilde{\mathcal{F}}(\varphi;ny_{\alpha})(g_{\infty})=\sum_{S\in V'_{1,2}(\Q)\colon (S,S)\geq 0}\widetilde{\mathcal{F}}^{Q'}_S(\varphi;ny_{\alpha})(g_{\infty}),$$ 
where for $S\in V'_{1,2}(\Q)$ satisfying $(S,S)\geq 0$, 
\begin{equation}
\label{eqn-factorization-non-degenerate-FJ-coefficient}
\widetilde{\mathcal{F}}^{Q'}_S(\varphi;ny_{\alpha})(g_{\infty})=\int_{X(\A_f)}a_{[ny_{\alpha},S]}(\varphi,x_f)\,dx_f\cdot \int_{X(\R)}\mathcal{W}_{[2\pi ny_{\alpha},2\pi S]}(x_{\infty}g_{\infty}g_{y_{\alpha}})\,dx_{\infty}. 
\end{equation}
Here $g_{\infty}\in G(\R)$ and $X=\{\exp(sb_1\wedge b_{-2})\colon s\in \mathbb{G}_a\}$. Given $g_{\infty}\in M_{Q'}(\R)$,  let $\overline{g}_{\infty}$ to be the image of $g_{\infty}$ in $\SO(V_{2,3}')$.   We express $\overline{g}_{\infty}$ as an ordered pair $\overline{g}_{\infty}=(t,u)$ with $t\in \R_{>0}$ and $u\in \mathrm{SO}(V'_{1,2})(\R)^0$.    The coordinate $t \in \R$ is normalized so that $(t, 1) \cdot  b_2 = tb_2$.  In \cite[Proposition 7.3]{JMNPR24},  the archimedean integral in \eqref{eqn-factorization-non-degenerate-FJ-coefficient} is evaluated when $\alpha=2$ and $n=1$. The computation in (loc. cit.) generalizes to the case of general $\alpha$ and general $n$ as follows.  
\begin{lemma}
\label{lemma-main-Archimedean Integral}
Suppose $\alpha\in 2\Z_{>0}$, $n\in \Z_{\geq 1}$, and $S\in V'_{1,2}(\Q)$ is such that $[ny_{\alpha},S]\succeq 0$.  Let $g_{\infty}\in M_{Q'}(\R)^0$ and write $\overline{g}_{\infty}=(t,u)$ with $t\in \R_{>0}$ and $u\in \SO(V_{1,2}')(\R)^0$.  Then 
 $$
 \int_{X(\R)}\mathcal{W}_{[2\pi ny_{\alpha},2\pi S]}(x_{\infty}g_{\infty}g_{y_{\alpha}})\,dx_{\infty}=\frac{t^{\ell} e^{-2\sqrt{2}\pi(n\alpha^{1/2}-t(S,u\cdot v_2))}}{2\sqrt{2}n\alpha^{1/2}}\sum_{-\ell\leq v\leq \ell}i^v\frac{x^{\ell+v}y^{\ell-v}}{(\ell+v)!(\ell-v)!}
 $$
\end{lemma}
\begin{proof}
The result is proven by applying the same computational techniques as in \cite[Proposition 7.3]{JMNPR24}. For a detailed account, the reader may consult \cite[Lemma 4.0.9]{McG25}.
\end{proof}
Lemma \ref{lemma-main-Archimedean Integral} is applied to establish the holomorphy statement in the following result. 
\begin{proposition}
\label{prop-existence-holomorphic-FJ-coefficient}
Let $\alpha\in 2\Z_{>0}$, $n\in \Z_{>0}$, and suppose $\varphi\in M_{\ell}(1)$. Then the function 
    $$
    \xi^{\varphi}(ny_{\alpha})\colon M'(\R)\to \C, \qquad g_{\infty}\mapsto \overline{\left\{\widetilde{\mathcal{F}}(\varphi;ny_{\alpha})(g_{\infty}), (-ix+y)^{2\ell}\right\}}_{K_{\infty}}
    $$ 
    is the automorphic form associated to a weight $\ell$ holomorphic modular form on $M'$. 
\end{proposition}
\begin{proof}
 The result is proven by an argument parallel to the one given in the proof \cite[Corollary 7.6]{JMNPR24}.  Again, the reader may consult \cite[Proposition 4.0.12]{McG25} for full details.  Since it will be relevant shortly, we note that as a consequence of the argument in (loc. cit.), the Fourier expansion of  $\xi^{\varphi}(ny_{\alpha})$ along the  $N_{Q'}$ takes the form 
 $$
\xi^{\varphi}(ny_{\alpha})(g_{\infty})=\sum_{S\in V'_{1,2}(\Q)_{\geq 0}}\xi^{\varphi}(ny_{\alpha})_S(g_{\infty})
$$ where
 \begin{align*}
    j_{y_{\alpha}}(h,-i\sqrt{2}v_2)^{\ell}\xi^{\varphi}(ny_{\alpha})_S(g_{\infty})
    &= \eta_S \cdot e^{2\pi i(S,Z)}.
\end{align*}
Here $j_{y_{\alpha}}(\cdot, -i\sqrt{2}v_2)$ is the automorphy factor of subsection \ref{subsec:HMFS} and 
$$
\eta_S=\frac{e^{-2\sqrt{2}\pi n \alpha^{1/2}}}{2\sqrt{2}n\alpha^{1/2}}\cdot \overline{\int_{X(\A_f)}a_{[ny_{\alpha}, S]}(\varphi, x_f)\, dx_f\cdot \left\{\sum_{-\ell\leq v\leq \ell} i^v\frac{x^{\ell+v}y^{\ell-v}}{(\ell+v)!(\ell-v)!}, (-ix+y)^{2\ell}\right\}}_{K_{\infty}}.
$$ 
\end{proof}
By examining the explicit formula for $\eta_S$ given in the proof above,  we obtain the following corollary to Proposition \ref{prop-existence-holomorphic-FJ-coefficient}.
\begin{corollary}
\label{cor-FC-of-HMF-identified} Given $\varphi\in M_{\ell}(1)$ and $S\in V_{1,2}'(Q)$ such that $(S,S)\geq 0$, define 
\begin{equation}
\label{defn-of-A_xi(ny_alpha)}
A_{\xi^{\varphi}(ny_\alpha)}[S]=\overline{\int_{\A_f}a_{[ny_{\alpha}, S]}(\varphi, \exp(s_fb_1\wedge b_{-2}))\, ds_f}.  
\end{equation}
Then the numbers $\{A[S]\colon S\in V'_{1,2}(\Q)\colon (S,S)\geq 0\}$ are the Fourier coefficients of a holomorphic modular form on $\h_{y_{\alpha}}$ of level $M'(\Z)$.  Hence, $A_{\xi^{\varphi}(ny_{\alpha})}[S]\neq 0$ implies $S\in V_{1,2}'(\Z)_{\geq 0}^{\vee}$. 
\end{corollary}
\subsection{The Primitivity Theorem}
\label{subsec-the-primitivity-theorem}
In this subsection, we apply the results of Subsections \ref{subsec:Degenerate Fourier-Jacobi Coefficients of Modular Forms on G} and \ref{subsec:Non-Degenerate Fourier-Jacobi Coefficients of Modular Forms on G} to establish the main result of this section,  Theorem~\ref{thm-primitivity-intro-1}.  Before giving the proof of Theorem~\ref{thm-primitivity-intro-1} we require the following lemma.
\begin{lemma}
\label{lemma-non-vanishing-non-trivial-FJ-coefficient}
Assume $\ell\in \Z_{>0}$ and suppose $\varphi\in M_{\ell}(1)$.  Recall the orthogonal parabolic subgroup $R=M_RN_R$ defined in Subsection \ref{subsec:The Klingen Type Parabolic Q}. If $\varphi\equiv \varphi_{N_R}$ then $\varphi\equiv 0$. 
\end{lemma}
\begin{proof}
Assume $\varphi\equiv \varphi_{N_R}$.  The weight $\ell$ is positive,  and hence it suffices to show that $\varphi$ is constant.  Moreover, since $G(\A)=G(\Q)G(\R)G(\widehat{\Z})$,  it is enough to prove that $\varphi\rvert_{G(\R)}$ is constant.  Let $X$ denote the subgroup of $G(\R)$ generated by $G(\Z)$ and $N_R(\R)$.  Then $\varphi(x)=\varphi(1)$ for all $x\in X$, and so we are reduced to proving that $X=G(\R)$.  \\
\indent As $G$ is  semi-simple and split,  $G$ is the special fiber of a finite Chevalley group $\underline{G}$ defined over $\Z$ \cite[Theorem 6]{MR466335}.  So, $G(\Z)$ contains representatives for each of the Weyl reflections in $G$.  Since the root system of $G$ is simply laced,   the Weyl group of $G$ acts transitively on the set of roots of $G$ \cite[Lemma 10.2.2(ii)]{MR1642713}. Hence, $X$ contain the root subgroup $U_{\alpha}(\R)$ for every root $\alpha$ of $G$.  Since $G$ is simply connected, $G(\R)$ is generated by the root subgroup $U_{\alpha}(\R)$ as $\alpha$ runs over the set of roots (see for example \cite[\S 9.4]{MR1642713}). Hence, $X=G(\R)$.
\end{proof}
We are ready to give the proof of our first main result.
\begin{proof}[Proof of Theorem~\ref{thm-primitivity-intro-1}]
Suppose for a contradiction that $\varphi \not\equiv 0$.  Then $\varphi\rvert_{G(\R)}\not\equiv 0$, and applying Lemma \ref{lemma-non-vanishing-non-trivial-FJ-coefficient} in tandem with \eqref{FJ-expansion-of-general-QMF}, there exists a non-zero vector $y\in V_{3,3}(\Z)$ such that $(y, y)\geq 0$ and $\mathcal{F}(\varphi;y)\rvert_{G(\R)}$ is not identically zero.  The proof partitions into the case when $y$ is non-isotropic, and the case when $y$ is isotropic.  \\
\indent First we suppose $y$ is non-isotropic.  Then there exists $\alpha\in 2\Z_{>0}$ and $n\in \Z_{\geq 1}$ such that $\mathcal{F}(\varphi;ny_{\alpha})\rvert_{G(\R)}\not\equiv 0$.  Applying \eqref{eqn-factorization-non-degenerate-FJ-coefficient} and \eqref{defn-of-A_xi(ny_alpha)}, there exists $S\in V_{1,2}'(\Q)$ such that 
$$
A_{\xi^{\varphi}(ny_{\alpha})}[S]\neq 0
$$ By Corollary \ref{cor-FC-of-HMF-identified} and Theorem \ref{thm-3-yamana},  we may assume $S$ is primitive.  Then \eqref{defn-of-A_xi(ny_alpha)} gives that
\begin{align}
\label{eqn-manipulation-of-a-lifetime}
\overline{A_{\xi^{\varphi}(ny_\alpha)}[S]}
&  =\int_{\A_f/\widehat{\Z}}a_{[ny_\alpha, S]}(\exp(sb_1\wedge b_{-2})) = \sum_{s\in \Q/\Z}\Lambda_{\varphi}[ny_\alpha, S+sny_{\alpha}].
\end{align}
If $s\in \Q_{[0,1)}$ and $\Lambda_{\varphi}[ny_\alpha, S+sny_{\alpha}]\neq 0$, then $S+sny_\alpha\in V_{2,2}(\Z)$.  Fix $n',m',r'\in \Z$ such that $\gcd(n',m',r')=1$ and $$S=-n'b_4-m'b_{-4}-\dfrac{r'}{\alpha}y_{\alpha}^{\vee}.$$Then $S+sny_\alpha\in V_{2,2}(\Z)$ implies $\dfrac{-r'}{\alpha}+sn\in \Z$, and thus $n\alpha s\in \Z$.  Hence, 
$$\overline{A_{\xi^{\varphi}(ny_\alpha)}[S]}=\displaystyle{\sum}_{s=0}^{n\alpha-1}\Lambda_{\varphi}\left[ny_{\alpha},-n'b_4-m'b_{-4}+\frac{s-r'}{\alpha}b_3+\dfrac{s+r'}{2}b_{-3}\right].$$  
Let $s\in \Z_{[0,n\alpha-1]}$ be such that $(s+r')/2$ and $(s-r')/\alpha$ are integral.  Then, given $d\in \Z_{\geq 1}$ such that $d$ divides  $n'$, $m'$,  $(s+r')/2$, and $(s-r')/\alpha$,  $d$ also divides $r'$.  Hence,  $d=1$ and \eqref{eqn-manipulation-of-a-lifetime} expresses $\overline{A_{\xi^{\varphi}(ny_{\alpha})}[S]}$ as a sum of primitive Fourier coefficients of $\varphi$,  which is a contradiction.  

\indent Next we consider the case when $y$ is isotropic. Hence, there exists an isotropic vector $y\in V_{3,3}(\Z)$ such that $\mathcal{F}(\varphi;y)\rvert_{G(\R)}\not\equiv 0$.  Then we may suppose $y=nb_2$ for some $n\in \Z\backslash \{0\}$. By Corollary \ref{corollary-Degenerate-QMF-FJs},  the constant term $\varphi_{N_P}$ satisfies $\varphi_{N_P}\rvert_{G(\R)}\not \equiv 0$.  Hence, there exists $C\in \Q^3$ such that the Fourier coefficient $(\varphi_{N_P})_{\exp(e_2\otimes E),C}$ of \eqref{defn-FC-Heisenber-D4} is not identically zero on $G(\R)$.  Hence,  we may may assume the coefficient $a_{\varphi}(C)$ of \eqref{eqn-FE-3-modular-forms-on-d4} is non-zero, and by Lemma \ref{primitivity-holomorphic-modular-forms-on-MP},  we may take $C=(a,b,c)\in \Z^3$ to be such that $\gcd(a,b,c)=1$.  \\
\indent Let $x\in E$ be such that $(x,C)_E\neq 0$,  then by Corollary \ref{cor-FCs-constant-degenerate-Heisenberg-FCS-D4},
\begin{equation}
\label{final-step-proof-main-thm}
a_{\varphi}(C)=B_C\cdot \sum_{s\in \Q/\Z}\Lambda_{\varphi}[ab_3+cb_{-3}+s(C, x)_Eb_{-4}, -bb_{-4}]. 
\end{equation}
However,  since $\mathrm{gcd}(a,b,c)=1$,  if $s\in \Q/\Z$, then by assumption,
$$\Lambda_{\varphi}[ab_3+cb_{-3}+s(C, x)_Eb_{-4}, -bb_{-4}]=0.
$$
So, \eqref{final-step-proof-main-thm} contradicts $a_{\varphi}(C)\neq 0$. Hence $\mathcal{F}(\varphi;y)\rvert_{G(\R)}\equiv 0$, which completes the proof. 
\end{proof}
\begin{definition}
\label{defn-slice-primitive}
We say that a pair $[T_1,T_2]\in V_{2,2}(\Z)^{\oplus 2}$ is slice primitive if 
$$
\Q\linspan\{T_1,T_2\}\cap V_{2,2}(\Z)=\Z\linspan\{T_1,T_2\}.
$$ 
\end{definition}  
\begin{corollary}
\label{slice-primitivity-corollary}
Let $\ell>0$ and suppose $\varphi$ is a weight $\ell$ cuspidal quaternionic modular form of level one.  Write the Fourier expansion of $\varphi_Z$ as
\begin{equation}
\varphi(g_{\infty})=\sum_{T_1,T_2\in V_{2,2}(\Z) \colon [T_1,T_2]\succ 0} \Lambda_{\varphi}[T_1,T_2] \mathcal{W}_{\ell}(g_{\infty}),
\end{equation} 
and assume $\Lambda_{\varphi}[T_1,T_2]=0$ for all slice primitive vectors $[T_1,T_2]\in V_{2,2}(\Z)^{\oplus 2}$. Then $\varphi\equiv 0$.  
\end{corollary}
\begin{proof}
Assume $\Lambda_{\varphi}[T_1,T_2]=0$ for all slice primitive vectors $[T_1,T_2]\in V_{2,2}(\Z)^{\oplus 2}$. By Theorem~\ref{thm-primitivity-intro-1}, it suffices to show that if $[T_1,T_2]\in V_{2,2}(\Z)^{\oplus 2}$ is primitive, then $\Lambda_{\varphi}[T_1,T_2]=0$.  \\
\indent Assume $[T_1,T_2]\succ 0$ is primitive.  Applying \eqref{identification-of-V_22}, we may write 
$$
[T_1,T_2]=\left[\begin{pmatrix} a &b \\ c &d \end{pmatrix} , \begin{pmatrix} e &f \\ g &h \end{pmatrix}\right] 
$$ where $a,\ldots , h\in \Z$ are jointly coprime.  Since $\varphi$ is of level $1$,  \cite[Appendix Ch. II]{MR2051392} implies that we may assume
\begin{equation}
\label{eqn-form-of-B}
[T_1,T_2]=\left[y_{\alpha}, -nb_4-mb_{-4}+rb_{-3}\right].
\end{equation}
Moreover, since $[T_1,T_2]\succ 0$,  Proposition \ref{Properties-of-beta} implies $\alpha>0$. \\ 
\indent Consider the Fourier expansion of the holomorphic modular form $\xi^{\varphi}(y_{\alpha})$ of Corollary \ref{cor-FC-of-HMF-identified}.  By \eqref{eqn-manipulation-of-a-lifetime},  if $n',m',r'\in \Z$ and $S=-n'b_4-m'b_{-4}-\dfrac{r'}{\alpha}y_{\alpha}^{\vee}$ then $\overline{A_{\xi^{\varphi}(y_{\alpha})}[S]}$ equals
\begin{equation}
\label{manipulation-of-a-lifetime-2}
\sum_{s=0}^{\alpha-1}\Lambda_{\varphi}\left[y_{\alpha},-n'b_4-m'b_{-4}+\frac{s-r'}{\alpha}b_3+\dfrac{s+r'}{2}b_{-3}\right] 
=\Lambda\left[y_{\alpha},-n'b_4-m'b_{-4}+r'b_{-3}\right].
\end{equation}
Clearly,  if $S$ is primitive, then $\left[y_{\alpha},-n'b_4-m'b_{-4}+r'b_{-3}\right]$ is slice primitive. Thus, $\overline{A_{\xi^{\varphi}(y_{\alpha})}[S]}=0$ for all primitive $S$, and Corollary \ref{cor-FC-of-HMF-identified}, together with Theorem \ref{thm-3-yamana}, implies $\xi^{\varphi}(y_{\alpha})\equiv 0$. Hence, by \eqref{manipulation-of-a-lifetime-2},  $\Lambda\left[y_{\alpha},-n'b_4-m'b_{-4}+r'b_{-3}\right]=0$ for all $n',m',r'\in \Z$. Thus, \eqref{eqn-form-of-B} implies $
\Lambda[T_1,T_2]=0
$ as required.
\end{proof}
As a corollary to Theorem~\ref{thm-primitivity-intro-1} and Corollary~\ref{slice-primitivity-corollary}, we obtain the following. 
\begin{corollary}
\label{cor-determination}
Let $\varphi_1,\varphi_2\in M_{\ell}(1)$ be level one quaternionic modular forms on $G$. 
\begin{compactenum}[(i)] 
\item If $\Lambda_{\varphi_1}[B]=\Lambda_{\varphi_2}[B]$ for all primitive $B\in V_{2,2}(\Z)^{\oplus 2}$. Then $\varphi_1=\varphi_2$. 
\item If $\varphi_1$ and $\varphi_2$ are cuspidal and $$
\Lambda_{\varphi_1}[B]=\Lambda_{\varphi_2}[B],
$$ for all slice primitive $B\in V_{2,2}(\Z)^{\oplus 2}$, then $\varphi_1=\varphi_2$. 
\end{compactenum}
\end{corollary}
\begin{remark}Yamana \cite{Yamana09} has shown that the proof of Zagier's Theorem is broadly applicable to holomorphic modular forms on classical groups.  Similarly, one can formulate Theorem~\ref{thm-into-cuspidality-criteria} for quaternionic modular forms more generally.  However, it is natural to consider Theorem~\ref{thm-primitivity-intro-1} first in the setting of the group $G$ on account of Theorem~\ref{Slice-Primitive-Lifts}, which has no known analogue outside of type $D_4$ (see Remark~\ref{rmk-intro-1}). Another reason to focus on type $D_4$ is because the arithmetic invariant theory of $\SL_2(\Z)^3$ acting on $\mathrm{M}_2(\Z)^{\oplus 2}$ is particularly rich. Indeed,  Bhargava \cite[\S 2.3, pg. 221]{MR2051392} defines a notion of projectivity for elements $B\in \mathrm{M}_2(\Z)^{\oplus 2}$.  Corollary~\ref{slice-primitivity-corollary} gives progress towards the following conjecture.
\end{remark}
\begin{conjecture}
\label{projectivity-conjecture} Let $\varphi$ be a cuspidal quaternionic modular form on $G$ of weight $\ell>0$ and level one such that $\Lambda_{\varphi}[B]=0$ for all projective $B\in \mathrm{M}_2(\Z)^{\oplus 2}$.  Then $\varphi\equiv 0$.
\end{conjecture}  

\section{An Application to The Quaternionic Maass Spezialschar}
\label{sec-The Quaternionic Maass Spezialschar}
\subsection{Quaternionic Modular Forms on $\SO_8$}
\label{subsec: Quaternionic Modular Forms on SO_8}
Let $V$ denote the $8$-dimensional split quadratic space of Subsection \ref{subsec:The underlying quadratic space $V$}.
In this subsection, we review the definition of quaternionic modular forms on $\SO(V)$ following \cite{pollackCuspidal}. We also explain some notation concerning the Fourier coefficients of modular forms on $\SO(V)$. \\ 
\indent Write $\pi\colon G\to \SO(V)$ for the projection homomorphism defined in Subsection \ref{subsec-the-Lie-Algebra-of-G}. The representation $\V$ of Subsection \ref{subsec-compact-subgroups} occurs in the Lie algebra of the maximal compact subgroup $K_{\infty}\leq \Spin(V)(\R)$. Therefore, $\V$ descends along $\pi$ to give a representation of the identity component of the maximal compact subgroup $K_{\SO_8}\leq \SO(V)(\R)$, which is the image of $K_{\infty}$ under $\pi$. Similarly, the differential operator $D_{\ell}$ of Subsection \ref{subsec: Quaternionic Modular Forms on G} descends to give a $K_{\SO_8}^0$ invariant differential operator on smooth functions $\SO(V)(\R)\to \V_{\ell}$. Let $\so(V)=\Lie(\SO(V))\otimes_{\Q}\C$ and write $Z(\so(V))$ for the center of the universal enveloping algebra of $\so(V)$. 
\begin{definition}
    A weight $\ell>0$ quaternionic modular form on $\SO(V)$ of level one is a smooth moderator growth function $\varphi\colon G(\A)\to \V_{\ell}$ such that 
    \begin{compactenum}
        \item[(i)] If $g\in \SO(V)(\A)$ and $\gamma\in \SO(V)(\Q)$ then $\varphi(\gamma g)=\varphi(g)$, 
        \item[(ii)] If $g\in \SO(V)(\A)$ and $k\in \SO(V)(\widehat{\Z})$ then $\varphi(gk)=\varphi(g)$, 
        \item[(iii)] If $g\in \SO(V)(\A)$ and $k_{\infty}\in K_{\SO_8}^0$ then $\varphi(gk_{\infty})=k_{\infty}^{-1}\varphi(g)$, 
        \item[(iv)] $D_{\ell}\varphi\equiv 0$, and  
        \item[(v)] $\varphi$ is $Z(\so(V))$-finite. 
        \end{compactenum}
\end{definition}
Clearly, if $\varphi$ is a quaternionic modular form on $\SO(V)$ of level one, then the pull-back of $\varphi$ by $\pi$ is a level one quaternionic modular form $\Spin(V)$ in the sense of Subsection \ref{subsec: Quaternionic Modular Forms on G}.
\begin{lemma}
\label{lemma-unique-determination}
    Suppose $\varphi_1$ and $\varphi_2$ are quaternionic modular forms on $\SO(V)$ of level one. For $i=1,2$, let $\pi^{\ast}\varphi_i$ denote the pull-back of $\varphi_i$ to $\Spin(V)$. If $\pi^{\ast}\varphi_1=\pi^{\ast}\varphi_2$ then $\varphi_1=\varphi_2$. 
\end{lemma}
\begin{proof}
    Let $V(\Z)=\Z\linspan\{b_{\pm i} \colon i=1,2,3,4\}$.The quotient $\SO(V)(\Q)\backslash \SO(V)(\A_f)/ \SO(V)(\widehat{\Z})$ classifies isomorphism classes of lattices in the genus of $V(\Z)$. Thus, by the Hasse principle
$$
|\SO(V)(\Q)\backslash \SO(V)(\A_f)/ \SO(V)(\widehat{\Z})|=1,
$$ and if $\Gamma=\SO(V)(\widehat{\Z})\cap \SO(V)(\Q)$, then $\SO(V)(\Q)\backslash \SO(V)(\A)/G(\widehat{\Z})=\Gamma\backslash \SO(V)(\R)$. Since $\Gamma$ contains elements in the non-identity component of $\SO(V)(\R)$, it follows that $\varphi_i$ is determined by its restriction to $\SO(V)(\R)^0$. As $\Spin(V)(\R)$ surjects onto $\SO(V)(\R)^0$, the statement follows. 
\end{proof}
Suppose $\varphi$ is a level one quaternionic modular form on $\SO(V)$ and $B\in V_{2,2}(\Z)^{\oplus 2}$. Define the Fourier coefficient $\Lambda_{\varphi}[B]$ via the equality $\Lambda_{\varphi}[B]=\Lambda_{\pi^{\ast}\varphi}[B]$. Then we obtain the following Corollary to Lemma \ref{lemma-unique-determination} and Corollary \ref{slice-primitivity-corollary}. 
\begin{corollary}
\label{cor-slice-primitive-SO_8}
Suppose $\varphi_1$ and $\varphi_2$ are cuspidal level one quaternionic modular forms on $\SO(V)$. Assume $\Lambda_{\varphi_1}[B]=\Lambda_{\varphi_2}[B]$ for all slice primitive $B\in V_{2,2}(\Z)^{\oplus 2}$. Then $\varphi_1=\varphi_2$.  

\end{corollary}
\subsection{The Quaternionic Saito-Kurokawa Lifting} 
\label{subsec:The Quaternionic Saito-Kurokawa Lifting} 
In this section we define the quaternionic Saito-Kurokawa subspace $\mathrm{SK}_{\ell}$ of modular forms on $\SO_8$.  \ \indent 
Suppose $B=[T_1,T_2]\in V^{\oplus 2}$ and define a $2$-by-$2$ matrix with entries in $\Q$ via the formula
\begin{equation}
\label{Definition-S(T1,T2)}
T(B)=\frac{1}{2}\begin{pmatrix} (T_1,T_1) &(T_1,T_2) \\ (T_2,T_1) &(T_2,T_2)\end{pmatrix}. 
\end{equation}
    \begin{theorem}\emph{\cite[Theorem 4.1.1]{pollackCuspidal}}
    \label{pollack4.1.1}
    Suppose $\xi\colon \Sp_4(\Q)\backslash \Sp_4(\A)\to \C$ is the automorphic function associated to a level one, cuspidal,  holomorphic modular form of even weight $\ell\geq 16$. Write the classical Fourier expansion of $\xi$ in the form $F_{\xi}(Z)=\sum_{T>0}B_{\xi}[T]\exp(2\pi i \tr(TZ))$.   For all $B\in V_{2,2}(\Z)^{\oplus 2}$ satisfying $B\succ 0$,  define
\begin{equation}\label{fouriercoefficients}
		\Lambda_{\theta^{\ast}(F_{\xi})}(B)=\sum_{\tiny \hbox{$r\in \GL_2(\Z)\backslash \mathrm{M}_2(\Z)^{\det\neq 0}$ s.t. $B r^{-1}\in V_{2,2}(\Z)^{\oplus 2}$}}|\det(r)|^{\ell-1}\overline{B_{\xi}[^{t}r^{-1}T(B)r^{-1}]}.
\end{equation}
Here $B r^{-1}$ is computed as a matrix product where $B=[T_1,T_2]$ is a $1\times 2$ matrix with entries in $V_{2,2}(\Z)$. Then the numbers $\Lambda_{\theta^{\ast}(F_{\xi})}(B)$ are the Fourier coefficients of a non-zero, weight $\ell$, level one, cuspidal quaternionic modular form, $\theta^{\ast}(F_{\xi})$,  on $\SO_8$.
    \end{theorem}
    \begin{definition}
    Suppose $\ell\geq 16$ is even and let $S_{\ell}(\Sp_4(\Z))$ denote the space of cuspidal holomorphic Siegel modular forms of weight $\ell$,  genus $2$, and level $1$. Define the quaternionic Saito-Kurokawa subspace of weight $\ell$ as $\mathrm{SK}_{\ell}=\{\theta^{\ast}(F)\colon F\in S_{\ell}(\Sp_4(\Z))\}$.  
    \end{definition}
\subsection{The Quaternionic Maass Spezialschar}
The purpose of this subsection is to present a characterization of $\mathrm{SK}_{\ell}$ which is analogous to the characterization of the classical Saito-Kurokawa subspace given by \cite[Theorem 1(iii)]{ZaFrench}.
\begin{definition}
\label{defn-quaternionic-maass-spezialschar}
The weight $\ell$ quaternionic Maass Spezialschar $\mathrm{MS}_{\ell}$ is the subspace of level one, weight $\ell$ quaternionic modular forms $\varphi$ on $\SO(V)$ such that if $B_1, B_2\in V_{2,2}(\Z)^{\oplus 2}$ are slice primitive and $T(B_1)=T(B_2)$ (see \eqref{Definition-S(T1,T2)}), then $\Lambda_{\varphi}[B_1]=\Lambda_{\varphi}[B_2]$. 
\end{definition}
Definition \ref{defn-quaternionic-maass-spezialschar} should be compared to \cite[Definition 5.8]{JMNPR24}. The reader will note that the condition in Definition  \ref{defn-quaternionic-maass-spezialschar} is equivalent to condition (i) in  (loc. cit.).  The next Lemma shows that condition (ii) of (loc.cit) is redundant,  which explains its absence in Definition \ref{defn-quaternionic-maass-spezialschar}.
\begin{lemma}
\label{equivalence-of-definition} Suppose $\varphi\in \mathrm{MS}_{\ell}$. Given $B\in  V_{2,2}(\Z)^{\oplus 2}$,  choose a slice primitive $\breve{B}\in V_{2,2}(\Z)^{\oplus 2}$ such that $T(B)=T(\breve{B})$, such a choice of $\breve{B}$ exists by \cite[Lemma 5.7]{JMNPR24}. Define 
$$
\Lambda_{\varphi}^{\mathrm{prim}}[B]=\Lambda_{\varphi}[\breve{B}]. 
$$
Then, for all $B\in  V_{2,2}(\Z)^{\oplus 2} $, 
\begin{equation}
\label{eqn-fine-Maass-Relation}
\Lambda_{\varphi}[B]=\sum_{\tiny r\in\GL_2(\Z)\backslash \mathrm{M}_2(\Z)^{\det\neq 0} \colon Br^{-1}\in V_{2,2}(\Z)^{\oplus 2}}|\det(r)|^{\ell-1}\Lambda_{\varphi}^{\mathrm{prim}}[Br^{-1}]. 
\end{equation}
In other words,  Definition \ref{defn-quaternionic-maass-spezialschar} is equivalent to \cite[Definition 5.8]{JMNPR24}. 
\end{lemma} 
 \begin{proof} 
 Suppose $\varphi\in \mathrm{MS}_{\ell}$ and let $F_{\xi}$ be the holomorphic Siegel modular form that is associated to $\varphi$ by \cite[Corollary 7.7]{JMNPR24}. Define $\varphi'=\theta^{\ast}(F_{\xi})$.  By \cite[Lemma 5.10]{JMNPR24},  the Fourier coefficients $\Lambda_{\varphi'}[B]$ satisfy \eqref{eqn-fine-Maass-Relation}.  Hence, it suffice to prove that $\varphi=\varphi'$.  Applying Corollary \ref{cor-slice-primitive-SO_8}, it is enough to show $\Lambda_{\varphi}[B]=\Lambda_{\varphi'}[B]$ for all slice primitive $B\in V_{2,2}(\Z)^{\oplus 2}$.  \\
 \indent Suppose $$
 T=\begin{pmatrix} a &b/2 \\ b/2 & c\end{pmatrix}
 $$ with $a,b,c\in \Z$. On the one hand, \cite[Corollary 7.7]{JMNPR24} implies $
 \overline{B_{\xi}\left[T\right]}=\Lambda_{\varphi}[T_1,T_2]$ where $T_1 =b_3+bb_4+ab_{-3}$ and $T_2 = -c b_4 - b_{-4}$. Since $[T_1,T_2]$ is slice primitive and $\varphi\in \mathrm{MS}_{\ell}$, this implies that $\Lambda_{\varphi}[B]=\overline{B_{\xi}\left[T\right]}$ for all slice primitive $B\in V_{2,2}(\Z)^{\oplus 2}$ satisfying $T(B)=T$. \\
 \indent On the other hand,   when $B$ is slice primitive, the summation \eqref{fouriercoefficients} consists of a single term,  and  so the equality $\varphi'=\theta^{\ast}(F_{\xi})$ implies $\Lambda_{\varphi'}[B]=\overline{B_{\xi}[T]}$ for all slice primitive $B\in V_{2,2}(\Z)^{\oplus 2}$ satisfying $T(B)=T$. Since $T$ was arbitrary, we conclude that $\Lambda_{\varphi}[B]=\overline{B_{\xi}[T(B)]}=\Lambda_{\varphi'}[B]$ for all slice primitive $B\in \mathrm{M}_2(\Z)^{\oplus 2}$,  completing the proof. 
 \end{proof}
 \begin{proof}[Proof of Theorem~\ref{Slice-Primitive-Lifts}]
 In light of the equivalence between Definition \ref{defn-quaternionic-maass-spezialschar} and \cite[Definition 5.8]{JMNPR24},  Theorem~\ref{Slice-Primitive-Lifts} follows as corollary to \cite[Theorem 1.3]{JMNPR24}. 
 \end{proof}
 \begin{remark}
\label{rmk-intro-1}
Theorem \ref{Slice-Primitive-Lifts} seems to be particular to the case of quaternionic modular forms on $\SO_8$. Indeed, if $n\geq 2$, then \cite[Theorem 4.1.1]{pollackCuspidal} describes a more general situation, involving a class of quaternionic modular forms on $\SO(4,n+2)$ that are theta lifts from $\Sp_4$.  It appears that these theta lifts on $\SO(4,n+2)$ only admit a characterization in the style of Theorem \ref{Slice-Primitive-Lifts} in the special case when $n=2$.  Indeed, if $\chi$ is a non-degenerate character of $N_R$, then the stabilizer of $\chi$ in $M_R^{\mathrm{der}}$ is identified with $\Spin(2,3)$.  The proof of \cite[Theorem 1.3]{JMNPR24} makes essential use of the exceptional isomorphism $\Spin(2,3)\simeq \Sp_4$.
\end{remark}
 \section{The Hecke Bound Characterization of Cusp Form on $G$} 
 \label{sec-The Hecke Bound Characterization of Cusp Form on $G$}
The purpose of this section is to prove Theorem~\ref{thm-intro-Hecke-IFF-Cuspidal-Quaternionic} and Theorem~\ref{thm-into-cuspidality-criteria}.  We work in the general setting a group $G_J$, associated to cubic norm structure $J$.  We have that $G_J=G^{\mathrm{ad}}$ when $J=E$ is the algebra of diagonal $3\times 3$ matrices.  Hence, the results we obtain for $G_J$ will also apply to level one forms on the group $G$.  Throughout, any unexplained notation has the same meaning as in \cite{pollackQDS}.  Similarly, we refer the reader to (loc. cit. ) for the precise definitions of quaternionic modular forms on $G_J$.
 \subsection{Non-Vanishing of Rank $3$ Fourier Coefficients}
 \label{subsec-Non-Vanishing of Rank $3$ Fourier Coefficients}
 The purpose of this subsection is to explain the proof of the following proposition.  We refer the reader to \cite[Definition 4.3.2]{pollackLL} for the definition of \textit{ran}k as it pertains to elements of $W_E\simeq V_{2,2}^{\oplus 2}$.
\begin{prop}
\label{prop-non-vanishing-rank-3}
Suppose $\varphi\in M_{\ell}(1)$ is non-zero and non-cuspidal. Then there exists a rank $3$ primitive element $B\in V_{2,2}(\Z)^{\oplus 2}$ such that $\Lambda_{\varphi}[B]\neq 0$. 
\end{prop} 
\begin{proof}[Proof of Proposition \ref{prop-non-vanishing-rank-3} assuming Theorem \ref{thm-into-cuspidality-criteria}] Assume $\varphi$ is non-cuspidal.  Then Theorem~\ref{thm-into-cuspidality-criteria} implies that $\varphi_{N_{J}}\neq 0$.  Hence, the semi-classical holomorphic modular form $\Phi$ of \eqref{eqn-structure-of-Heisenberg-constant-term} is non-zero. As in \eqref{eqn-FE-3-modular-forms-on-SL2}, write $a_{\varphi}(n_1,n_2,n_3)$ for the Fourier coefficients of
$
f_{\Phi}$,  the classical holomorphic modular form on $\h_{M_P^{\mathrm{der}}}$ corresponding to $\Phi$.  By Lemma \ref{primitivity-holomorphic-modular-forms-on-MP}, there exists a triple of coprime positive integers $C=(n_1,n_2,n_3)$ such that $a_{\varphi}(C)\neq 0$.  
Now by \eqref{final-step-proof-main-thm}, there exists $x\in E$ and $s\in \Q/\Z$ such that 
$$
\Lambda_{\varphi}[ab_3+cb_{-3}+s(C,x)b_{-4}, -bb_{-4}]\neq 0. 
$$
Since $\gcd(a,b,c)=1$, the element $B=[ab_3+cb_{-3}+s(C,x)b_{-4}, -bb_{-4}]$ is primitive. Moreover, in the coordinates given by the isomorphism $V_{2,2}(\Z)^{\oplus 2}\simeq W_E$,  $B=(0,0,C, s(C,x))$.  Hence $Q(B)=0$ and $B^{\flat}=(0,0,0,N_E(C))$ where the operation $B\mapsto B^{\flat}$ is described in \cite[\S 4.3 pg. 20]{pollackLL}. Therefore, since $a,b$, and $c$ are positive, $B^{\flat}\neq 0$ and so $B$ has rank $3$ as required. 
\end{proof}
\subsection{The Degenerate Coefficients of a Quaternionic Modular Form}
\label{subsec-The Degenerate Coefficients of a Quaternionic Modular Form}
In this subsection, we assume $J$ is a cubic norm structure such that trace pairing $(\cdot, \cdot)_J$ is positive definite.  Let $K_{J,\infty}$ be the maximal compact subgroup of $G_J(\R)$ from \cite[\S 5]{pollackQDS}.  The Heisenberg parabolic subgroup of $G_J$ \cite[\S 4.3.2]{pollackQDS} is denoted $P_J=N_{J}H_{J}$.  Here $H_J$ is the Levi subgroup of $P_J$ defined in \cite[\S 2.2]{pollackQDS} .  So $\Lie(N_{J})= e\otimes W_J+ \Q E_{13}$ and
$$
W_J\simeq \Lie(N_{J}^{\mathrm{ab}}), \qquad w\mapsto e\otimes w,
$$ where $V_2=\langle e,f\rangle$ is the standard representation of $\SL_2$. Write $\langle \cdot, \cdot \rangle_{W_J}$ for the usual symplectic form on $W_J$.  We refer the reader to \cite{pollackQDS} for any unexplained notation.  \\
\indent Given $w\in W_J$, define $\varepsilon_w\in \mathrm{Hom}(N_{J}(\Q)\backslash N_{J}(\A), \C^{\times})$ by
$$
\varepsilon_w(\exp(e\otimes w'))=\psi(\langle w,w'\rangle_{W_J}). 
$$ If $\varphi\colon G_J(\A) \to \mathbf{V}_{\ell}$ is an automorphic function, the Fourier coefficient $\varphi_w$ is defined as 
$$
\varphi_w\colon G_J(\A) \to \mathbf{V}_{\ell}, \quad \varphi_w(g)=\int_{[N_{J}]}\varphi(ng)\varepsilon_w(n)^{-1}\, dn. 
$$
Regarding the Fourier coefficients $\varphi_w$, we have the following.
\begin{lemma}
\label{lemma-left-invariance-by-X(A)}
   Let $w\in W_J$ be non-zero and $g=g_fg_{\infty}\in G_J(\A)$ be such that $g_{\infty}\in N_{J}(\R)H_J(\R)^{\pm}K_{J,\infty}^0$. Assume $m\in H_J(\A_f)\times H_J(\R)^{\pm}$ is such that there exists $m_{\Q}\in H_J(\Q)$, $m_{\infty}\in H_J(\R)^{\pm}$, and $m_f\in H_J(\A_f)$ satisfying (i) $m=m_{\Q}m_{\infty}m_f$, (ii) $m_{\Q}$ and $m_{\infty}$ stabilize $\varepsilon_{w}$, (iii) $m_{\infty}$ is unipotent, and (iv) $\varphi(m_fg)=\varphi(g)$.     
   Then $
   \varphi_{w}(mg)=\varphi_{w}(g)$.
\end{lemma}
\begin{proof}
Let $n_g\in N_{J}(\R)$, $m_g\in H_J(\R)^{\pm}$, and $k_g\in K_{J,\infty}^0$ be such that $g_{\infty}=n_gm_gk_g$.  By \cite[Theorem 1.2.1]{pollackQDS}, there exists a unique collection of locally constant functions
$$
\{a_w\colon G_J(\A_f)\to \C \colon w\in W_J\}
$$ 
 such that $\varphi_{w}(mg)=a_{w}(m_fg_f)\mathcal{W}_{J,2\pi w}(m_{\infty}g_{\infty})$.  Here $\mathcal{W}_{J,2\pi i w}$ is the generalized Whittaker function of (loc. cit.).  Let $\varepsilon_{w}^{\infty}$ be the archimedean component of $\varepsilon_{w}$. By the equivariance properties of $\mathcal{W}_{J,2\pi w}$,
\begin{equation}
\label{eqn-Lemma-left-invariance-by-X(A)}
\varphi_{w}(mg)=a_{w}(m_fg_f)\varepsilon_{w}^{\infty}(m_{\infty}n_g m_{\infty}^{-1})k_g^{-1}\cdot \mathcal{W}_{J,2\pi w}(m_{\infty}m_g).
\end{equation}
Inspecting the formula in (loc. cit.),  hypothesis (ii) and (iii) imply that $\mathcal{W}_{J,2\pi w}(m_{\infty}m_g)=\mathcal{W}_{J,2\pi w}(m_g)$. Moreover, (ii) implies that $\varepsilon_{w}^{\infty}(m_{\infty}n_gm_{\infty}^{-1})=\varepsilon_{w}^{\infty}(n_g)$, and so \eqref{eqn-Lemma-left-invariance-by-X(A)} simplifies to $\varphi_{w}(mg)=\varphi_{w}(m_fg)$. Applying hypothesis (iv) completes the proof. 
\end{proof}
The next proposition involves the $3$-step parabolic subgroup $Q\leq G_J$  associated to the element $h_Q=E_{11}+E_{22}-2E_{33}$.  Since $\mathrm{ad}(h_Q)$ acts on $\Lie(G_J)$ with eigenvalues $-3,-2,-1,0,1,2,3$,  $Q$ admits a Levi decomposition $Q=M_QN_Q$ where $M_Q$ is the zero eigenspace of $\mathrm{ad}(h_Q)$. The Lie algebra of the unipotent radical $N_Q$ decomposes into $\mathrm{ad}(h_Q)$ eigenspaces as 
 $$
 \Lie(N_Q)=\Lie(N_Q)^{[1]}\oplus \Lie(N_Q)^{[2]}\oplus \Lie(N_Q)^{[3]}.
 $$ 
Here $\Lie(N_Q)^{[1]}=(v_1\otimes J)\oplus (v_2\otimes J)$, $\Lie(N_Q)^{[2]}=\delta_3\otimes J^{\vee}$, and $\Lie(N_{Q})^{[3]}=\Q E_{13}\oplus \Q E_{23}$.  Moreover,   
 $$
 \Lie([N_Q,N_Q])=\Lie(N_Q)^{[1]}\oplus \Lie(N_Q)^{[2]}. 
 $$

Since $h_Q=E_{11}+E_{22}-2E_{33}$, $M_Q^{\mathrm{der}}$ is the principle $\SL_2$ in $G_J$ containing the root subgroup $\exp( \Q E_{12})$.  Let $n_{\alpha}\in M_Q^{\mathrm{der}}(\Z)$ denote a representative for the Weyl reflection in the hyperplane orthogonal to $\alpha$.  We view $n_{\alpha}\in G_J(\A)$ via the diagonal embedding $G(\Q)\hookrightarrow G(\A)$.  \\
\indent Given $C\in J^{\vee}\backslash\{0\}$, let $\eta_{C}\colon [N_Q]\to \C^{\times}$ be the character of $N_Q$ associated to the linear functional 
$
v_1\otimes X+v_2\otimes Y\mapsto (C,Y)_J,
$  and let
$$
S_C=\{\exp(v_2\otimes x) \mid \hbox{$x\in J$ such that $(x,C)_J=0$}\}.
$$
be the stabilizer of $w_C=(0,0,C,0)$ inside $\exp(v_2\otimes J)$.  
\begin{proposition}
     \label{prop-FCs-of-Heisenberg-constant-term}
     Suppose $\varphi$ is a weight $\ell$ modular form on $G_J$ and fix $C \in J^{\vee}$. Write $w_C=(0,0,C,0)\in W_J$ and let $\varphi_{N_{J}}$ denote the constant term of $\varphi$ along $N_{J}$.  Define
     $$
     \varphi_{N_{J},C}(g)=\int_{[\exp(v_2\otimes J)]}\varphi_{N_{J}}(ug)\eta_{C}^{-1}(u)du. 
     $$
     If $g\in G(\A)$ is such that $g_{\infty}\in N_{J}(\R)H_J(\R)^{\pm}K_{J,\infty}^0$ then
     \begin{equation}
     \label{eqn-proposition-FCs-of-Heisenberg-constant-term}
     \varphi_{N_{J},C}(g)=\int_{S_{C}(\A)\backslash \exp(v_2\otimes J)(\A)}\varphi_{w_C}(un_\alpha g)du.
     \end{equation}
     \end{proposition}
     \begin{proof}
 The result is proven by calculating
         $$
     \varphi_{N_Q,\eta_C}(g):=\int_{[N_Q]}\varphi(cg)\eta_C(c)^{-1}\,dc 
         $$ in two different ways. \\
         \indent To obtain the left hand side of \eqref{eqn-proposition-FCs-of-Heisenberg-constant-term}, we calculate $\varphi_{N_Q,\eta_C}$ by Fourier expanding $\varphi$ along the center $Z$. Let $\varphi_{[N_Q,N_Q]}$ be the constant term of $\varphi$ along $[N_Q,N_Q]$. Then $\varphi_{[N_Q,N_Q]}$ Fourier expands in characters of $[N_{J}^{\mathrm{ab}}]$ as 
         $
         \varphi_{[N_Q,N_Q]}(g)=\sum_{d\in \Q, c\in J^{\vee}}\varphi_{(0,0,c,d)}(g).
         $
         Hence,
         \begin{align}
         \label{eqn-first-manipulation-prop-degenerate-FJ-QMF}
             \varphi_{N_Q,\eta_C}(g)
             &=\int_{[\exp(v_2\otimes J)\backslash N_Q^{\mathrm{ab}}]}\int_{[\exp(v_2\otimes J)]}\varphi_{[N_Q,N_Q]}(cxg)\eta_C^{-1}(x)\,dx\,dc \notag \\
             &=\int_{[\exp(v_2\otimes J)]}\sum_{d\in \Q, c\in J^{\vee}}\varphi_{(0,0,c,d)}(xg)\eta_C^{-1}(x)\int_{[\exp(v_2\otimes J)\backslash N_Q^{\mathrm{ab}}]}\varepsilon_{(0,0,c,d)}(c)\,dc\,dx 
         \end{align}
         The quotient $\exp(v_2\otimes J)\backslash N_Q^{\mathrm{ab}}$ is identified with the subgroup of $N_Q^{\mathrm{ab}}$ whose Lie algebra is spanned by $v_1\otimes J$. So the inner integral in \eqref{eqn-first-manipulation-prop-degenerate-FJ-QMF} vanishes unless the summation index $c=0$. Since the subgroup $\exp(v_2\otimes J)$ stabilizes $\varepsilon_{(0,0,0,d)}$ for all $d\in \Q$, \eqref{eqn-first-manipulation-prop-degenerate-FJ-QMF} simplifies to 
         \begin{equation}\label{eqn-second-manipulation-prop-degenerate-FJ-QMF}
         \varphi_{N_Q,\eta_C}(g)=\int_{[\exp(v_2\otimes J)]}\varphi_{N_{J}}(xg)\eta_C^{-1}(x)dx+\int_{[\exp(v_2\otimes J)]}\sum_{d\neq 0}\varphi_{(0,0,0,d)}(xg)\eta_C^{-1}(x)dx
         \end{equation}
         Applying Lemma \ref{lemma-left-invariance-by-X(A)}, the function $x\mapsto \varphi_{(0,0,0,d)}(xg)$ is constant for all $d\neq 0$. Hence the second integral in \eqref{eqn-second-manipulation-prop-degenerate-FJ-QMF} vanishes and we obtain $\varphi_{N_Q,\eta_C}(g)=\varphi_{N_{J}, C}(g)$. \\ 
         \indent To finish the proof, we obtain the right hand side of \eqref{eqn-proposition-FCs-of-Heisenberg-constant-term} by calculating $\varphi_{N_Q,\eta_C}$ in a different way. This time we factor $\varphi_{[N_Q,N_Q]}$ across the central subgroup $Z'=n_{\alpha}Zn_{\alpha}^{-1}$ as
         \begin{equation}
         \label{eqn-Fourier-expansion-phi-along-[C,C]2}
         \varphi_{[N_Q,N_Q]}(g)=\sum_{c\in J^{\vee}, d\in \Q}\varphi_{(0,0,c,d)}(n_{\alpha}g).
         \end{equation}
         Then applying \eqref{eqn-Fourier-expansion-phi-along-[C,C]2}, we obtain
         \begin{align}
         \label{eqn-third-manipulation-prop-degenerate-FJ-QMF}
             \varphi_{N_Q,\eta_C}(g)
             &= \int_{[\exp(v_2\otimes J)]}\int_{[\exp(v_1\otimes J)]} \varphi_{[N_Q,N_Q]}(cxg)\eta_C^{-1}(c)\,dx\,dc \notag \\ \
             &= \int_{[\exp(v_1\otimes J)]}\sum_{d\in \Q, c\in J^{\vee}}\varphi_{(0,0,c,d)}(n_{\alpha}xg)\int_{[\exp(v_2\otimes J)]}\varepsilon_{(0,0,c,d)}(n_{\alpha}cn_{\alpha}^{-1})\eta_C^{-1}(c)\,dc\,dx.
         \end{align}
         The inner integral appearing in \eqref{eqn-third-manipulation-prop-degenerate-FJ-QMF} vanishes unless $c=C$. Hence 
         \begin{align}
         \label{eqn-fourth-manipulation-prop-degenerate-FJ-QMF}
         \varphi_{N_Q,\eta_C}(g)
         &=\int_{[\exp(v_2\otimes J)]}\sum_{d \in \Q} \varphi_{(0,0,C,d)}(xn_{\alpha}g)\,dx \notag  \\
         &=\int_{[\exp(v_2\otimes J)]}\sum_{u\in S_{C}(\Q)\backslash \exp(v_2\otimes J)(\Q)}\varphi_{w_C}(uxn_{\alpha}g)\,dx\notag \\
         &=\int_{S_{C}(\Q)\backslash \exp(v_2\otimes J)(\A)}\varphi_{
w_C}(xn_{\alpha}g)\,dx.
         \end{align}
        By Lemma \ref{lemma-left-invariance-by-X(A)}, the function $x\mapsto \varphi_{w_C}(xn_{\alpha}g)$ is left $S_{C}(\A)$ invariant. So, \eqref{eqn-fourth-manipulation-prop-degenerate-FJ-QMF} simplifies to 
       $
        \varphi_{N_Q,\eta_C}(g)=\int_{S_{C}(\A)\backslash \exp(v_2\otimes J)(\A)}\varphi_{w_C}(xn_{\alpha}g)dx
      $ and the proof is complete. 
     \end{proof}
    Recall \cite[Proposition 11.1.1]{pollackQDS}, which states that if  $m\in H_J(\A_f)\times H_J(\R)^{\pm}$, then 
     $$
     \varphi_{N_{J}}(m)=\nu(m)^{\ell}|\nu(m)|(\Phi(m)[x^{2\ell}]+\beta[x^{\ell}y^{\ell}]+\Phi'(m)[y^{2\ell}]).
     $$
         Here $\beta\in \C$ is constant and $\Phi$ is a holomorphic modular form on $\mathcal{H}_J^+$. \\
         \indent For $C\in J^{\vee}\backslash \{0\}$, write $C\geq 0$ if $(C,  1_J)_J$ and $(C^{\#}, 1_J)_J)\geq 0$. Then, $\Phi$ Fourier expands as 
         $$
         \Phi(m_fm_{\infty})=\Phi_{\exp(v_2\otimes J)}(m_fm_{\infty})+\sum_{C\in J^{\vee}\backslash\{0\}\colon C\geq 0}A_{C}(m_f)j(m_{\infty},i)^{-\ell}e^{2\pi i (m_{\infty}\cdot (i 1_J), C)_J}.
         $$
         Here  notation is as follows: 
         \begin{compactenum}
         \item $\Phi_{\exp(v_2\otimes J)}$ is the constant term of $\Phi$ along $\exp(v_2\otimes J)$, 
         \item $j(m_{\infty},i)$ is the automorphy factor of \cite[\S 2.3]{pollackQDS}, and
         \item 
         $
         \{A_C\colon H_J(\A_f)\to \C\}
         $
         is a family of locally constants functions.
         \end{compactenum}
         \indent 
          Let $u_0\in \V_{\ell}$ be such that if $\{\cdot , \cdot\}_{K_{J,\infty}^{0}}$ is the $K_{J,\infty}^0$ invariant pairing on $\V_{\ell}$, then 
         $$
         \left\{[x^{\ell-v}y^{\ell+v}],u_0\right\}_{K_{J,\infty}^{0}}
         =\begin{cases} 1, &\hbox{if $v=\ell$,} \\ 0,&\hbox{else.}\end{cases}
         $$So by definition, if $m=m_{f}m_{\infty}\in H_J(\A_f)H_J(\R)^{\pm}$,  
         \begin{equation}
         \label{eqn-extension-formula}
         \left\{\varphi_{N_{J},C}(m_{f}m_{\infty}),u_0\right\}_{K_{J,\infty}^{0}}=\nu(m)^{\ell}|\nu(m)|A_{C}(m_f)j(m_{\infty},i)^{-\ell}e^{2\pi i (m_{\infty}\cdot (i 1_J), C)_J}.
         \end{equation}
       Since  $\varphi_{w_C}(g_fg_{\infty})=a_{w_C}(g_f)\mathcal{W}_{J, 2\pi w_C}(g_{\infty})$, Proposition \ref{prop-FCs-of-Heisenberg-constant-term} implies that $\varphi_{N_{J},C}(g)$ is factorizable for all $g_f\in G_J(\A_f)$ and $g_{\infty}\in N_{J}(\R)H_J(\R)^{\pm}K_{J,\infty}^0$.  In particular,  the functions $\nu^{\ell}|\nu|A_C$ extend to functions on $G_J(\A_f)$ by the formula
         \begin{equation}
         \label{eqn-equality-of-finite-adelics}
         \nu^{\ell}|\nu|A_C(g_f)\doteq \int_{S_{C}(\A_f)\backslash \exp(v_2\otimes J)(\A_f)}a_{w_C}(un_\alpha g_f)du.
         \end{equation}
         Here $\doteq$ means that the left-hand side is a constant multiple of the right hand side.  
         \begin{lemma}
         Suppose $C\in J^{\vee}\backslash\{0\}$ is such that $C\geq 0$. Then the implied constant in \eqref{eqn-equality-of-finite-adelics} is non- zero and indepedent of $C$. 
         \end{lemma} 
                  \begin{proof} In \cite[\S 11]{pollackQDS}, the author constructs Klingen Eisenstein series on $G_J$ for which the Fourier coefficients $A_C$ are non-trivial.  Hence, the non-vanishing statement is a consequence of showing that the implied constant in  \eqref{eqn-equality-of-finite-adelics} is independent from $C$.  \\
                  \indent Assume  $C\geq 0$. Since the trace pairing $(\cdot, \cdot)_J$ on $J$ is positive definite, $S_C$ is contained in $\exp(v_2\otimes J)$ as subgroup of codimension $1$.  Hence, $\G_a\xrightarrow{\sim} S_C\backslash \exp(v_2\otimes J)$ via the map 
                  \begin{equation}
                  \label{eq-ref}
                  t\mapsto x(t):=\exp(v_2\otimes tC/(C,C)).
                  \end{equation}
                 If $R$ is a topological ring equipped with a measure,  then $S_C(R)\backslash   \exp(v_2\otimes J)(R)$ inherits a measure via \eqref{eq-ref}.  With this normalization,  the implied constant in \eqref{eqn-equality-of-finite-adelics} is equal to 
\begin{equation}
\label{manipulation 1}
e^{2\pi (1_J,C)_J}\int_{\R}\{\mathcal{W}_{J,2\pi w_C} (x(t)n_{\alpha}), u_0\}_{K_{J,\infty}}\,dt.
\end{equation}
Since $n_{\alpha}\in K_{J,\infty}^0$,  \eqref{manipulation 1} simplifies to 
$$
e^{2\pi (1_J,C)_J}\left\{\int_{\R}\mathcal{W}_{J,2\pi w_C} (x(s))\,ds , n_{\alpha}\cdot u_0\right\}_{K_{J,\infty}}
$$
Therefore,  it suffices to show that 
\begin{equation}
\label{eqn-indep-of-C}
e^{(C,1_J)_J}\cdot \int_{\R}\mathcal{W}_{J,2\pi w_C} (x(s))\,ds
\end{equation}
is independent of $C$.  If $r_0(i)=(1,-i1_J, -1_J, i)_J\in W_J(\C)$, then 
$$
\langle w_C, x(s)\cdot r_0(i)\rangle_{W_J},=s+i(C,1_J)_J.
$$
Hence,by  \cite[Theorem 1.2.1]{pollackQDS}, 
$$
\int_{\R}\mathcal{W}_{J,2\pi w_C}(x(s))\,ds=\frac{1}{2\pi}\sum_{-\ell\leq v \leq \ell} [x^{\ell-v}][y^{\ell+v}]\int_{\R} \left(\frac{|s+2\pi i( C,1_J)_J|}{ s+2\pi i( C,1_J)_J}\right)^vK_v(|s+2\pi i(C,1_J)_J|)\,ds. 
$$
Applying \cite[Lemma A.4]{pollack2024} and the fact that $(C,1_J)_J>0$, 
$$
\int_{\R} \left(\frac{|s+2\pi i(C,1_J)_J|}{s+2\pi i(C,1_J)_J}\right)^vK_v(|s+2\pi i( C,1_J)_J|)\,ds=\frac{1}{2} \cdot i^v\cdot e^{-2\pi(C,1_J)_J}
$$
Hence, \eqref{eqn-indep-of-C} is independent of $C$ as required. 
\end{proof}
         \subsection{The General Cuspidality Criterion}
\label{subsec-The General Cuspidality Criterion}
We continue to assume $J$ is a general cubic norm structure with a positive definite trace form. The proof of the following lemma is adapted from \cite[Lemma 13.6.]{pollack2024}.
     \begin{lemma}
     \label{lemma-auto-convergence-13.?}
         Suppose $C\in J^{\vee}$ is non-zero and define 
                  $$
         \overline{\varphi}_{C}\colon G_J(\A_f)\to \C,\qquad 
         \overline{\varphi}_{C}(g)=\int_{S_{C}(\A_f)\backslash \exp(v_2\otimes J)(\A_f)}a_{w_{C}}(ug)\,du.
         $$If $g\in G_J(\A_f)$, then
         $$
         a_{w_{C}}(g)=\int_{\A_f}\overline{\varphi}_{C}(\exp(\alpha E_{12})g)\,d\alpha. 
         $$
     \end{lemma}
     \begin{proof}
        Using the measure defined by \eqref{eq-ref}, 
         $
\overline{\varphi}_{C}(g)=\int_{\A_f}a_{w_{C}}(\exp(v_2\otimes sx)g)ds
         $
         where $x=C/(C,C)$.  Then, if $\varepsilon^f_{w_{C}}$ is the finite adelic part of the character $\varepsilon_{w_{C}}$, $\alpha\in \A_f$, and $g\in G_J(\A_f)$, 
         \begin{align*}
             \overline{\varphi}_{C}(\exp(\alpha E_{12})g)
             &= 
             \int_{\A_f}a_{w_{C}}(\exp(sv_2\otimes x)\exp(\alpha E_{12})g)\,ds \\ 
             &= \int_{\A_f}\varepsilon^f_{w_{C}\cdot \exp(sv_2\otimes x)}(\exp(\alpha E_{12}))a_{w_{C}}(\exp(sv_2\otimes x)g)\,ds \\ 
             &=\int_{\A_f}\psi^f((C, \alpha s x)_J)a_{w_{C}}(\exp(sv_2\otimes x)g)ds. 
         \end{align*}
         Now fix $g_f\in G_J(\A_f)$. Since $g'\mapsto a_{w_C}(g'g_f)$ is smooth on $G_J(\A_f)$, there exists $M_g\in \Z$ such that for $u\in \exp(M_g\widehat{\Z}E_{12})$, and $g'\in G_J(\A_f)$,  $a_{w_C}(g'ug_f)=a_{w_C}(g'g_f)$. Then 
         $$
         \overline{\varphi}_{C}(\exp(\alpha E_{12})g)
         = 
         \int_{\A_f/M_g\widehat{\Z}}\psi^f((C, \alpha sx)_J)a_{w_{C}}(\exp(sv_2\otimes x)g)\int_{M_g \widehat{\Z}}\psi^f((C, \alpha s'x)_J)ds'ds. 
         $$
        So $\overline{\varphi}_{C}(\exp(\alpha E_{12})g)\neq 0$ implies $\alpha \in V_{C, g}:= \{\alpha'\in \A_f \colon \psi^f((C, \alpha' M_g\widehat{\Z}x))=1\}$. Therefore, 
        \begin{align*}
            \int_{\A_f}\overline{\varphi}_{C}(\exp(\alpha E_{12})g)\,d\alpha
            &= 
            \int_{V_{C, g}}\int_{\A_f}\psi((C, \alpha s x)_{J})a_{w_{C}}(\exp(sv_2\otimes x)g)\,ds\,d\alpha \\
            &= \int_{\A_f}a_{w_{C}}(\exp(sv_2\otimes x)g)\int_{V_{C, g}}\psi((C, \alpha sx)_J)\, d\alpha \,ds \\ 
            &= \sum_{s\in \Q/ M_g\Z}a_{w_{C}}(\exp(sv_2\otimes x)g)\int_{V_{C, g}}\psi((C, \alpha sx)_J)\, d\alpha.
        \end{align*}
        To complete the proof, it suffices to show that if $s\in \Q$ satisfies $\int_{V_{C,g}}\psi((C, \alpha s x)_J)\,d\alpha \neq 0$, then $s\in M_g\Z$.  Let $\beta \in \Q$ be such that $\psi^f(a)=\psi_{\mathrm{std}}^f(\beta a)$ for all $a\in \A_f$. Here $\psi_{\mathrm{std}}^f$ is the standard additive character of $\A_f$ with $\ker(\psi_{\mathrm{std}}^f)=\widehat{\Z}$. So, if $\psi((C, \alpha s x)_J)=1$ for all $\alpha\in V_{C, g}$, then $\beta (C, \alpha s x)_J\in \widehat{\Z}$ for all $\alpha \in V_{C,g}$. Since $1/(\beta(C, M_gx)_J)\in V_{C, g}$,  $s\in M_g \Z$ as required.
     \end{proof}
     Before completing the proof of Theorem~\ref{thm-into-cuspidality-criteria} we require one additional lemma. 
\begin{lemma}
\label{orbit-lemma}
Suppose $J=\mathbb{G}_a^3$ or $J=H_3(C)$ with $C$ a composition algebra over $\Q$.  If $w\in W_J(\Q)$ is non-zero and $\mathrm{rank}(w)<4$ then there exists $m\in H_J(\Q)$ and $C\in J^{\vee}$ such that $\mathrm{rank}(w)=\mathrm{rank}(C)$, $\nu(m)=1$, and $m\cdot w=w_C$. 
\end{lemma}
\begin{proof}
Applying \cite[Lemma 10.0.2]{pollackQDS}, there exists an element $m\in H_J(\R)$ such that $\nu(m)=1$ and $m\cdot w=(0,0,c,d)$ for some $C\in J^{\vee}$ and $d\in \Q$.  In fact, the proof of (loc. cit.) is valid when the field $\R$ is replaced by $\Q$. Hence,  we obtain $m\in H_J(\Q)$ such that $\nu(m)=1$ and $m\cdot w=(0,0,C,d)$.  \\
\indent  Since $J$ contains a non-zero element $y$ such that $y^{\#}\neq 0$, we may act on $m\cdot w$ through the element $n^{\vee}(y)$ of \cite[\S 2.2]{pollackQDS} to ensure that $C\neq 0$. Assuming $C\neq 0$ and $d\neq 0$,  we may then arrange for the case when $d=0$ by selecting $x\in J$ appropriately such that $(C,x)\neq 0$, and acting on $(0,0,C,d)$ by the element $n(x)$ of (loc. cit.). In this way we construct $m\in H_J(\Q)$ such that $\nu(m)=1$ and 
$$
m\cdot w=w_C. 
$$
It remains to prove that $\mathrm{rank}(w_C)=\mathrm{rank}(C)$. This is immediate in the case when $\mathrm{rank}(w_C)=3$ since in this case $w_C^{\flat}=(0,0,0,N_J(C))$, and so $N_J(C)\neq 0$.  For the case when $\mathrm{rank}(w_C)=2$,  $N_J(C)= 0$ since $w_C^{\flat}=0$, and one can apply \cite[Lemma 4.3.4]{pollackLL} to conclude that $C^{\#}\neq 0$.  Finally,  when $\mathrm{rank}(w_C)=1$, $\mathrm{rank}(C)=1$ as a consequence of (loc. cit.). 
\end{proof}
\begin{proof}[Proof of Theorem~\ref{thm-into-cuspidality-criteria}] Suppose $\varphi_{N_{J}}\equiv 0$.  Then applying \eqref{eqn-extension-formula},  if $C\in J^{\vee}$ then $\nu^{\ell}|\nu|A_C(g_f)=0$ for all $g_f\in G(\A_f)$. Therefore,  by \eqref{eqn-equality-of-finite-adelics},  $\overline{\varphi}_C(g_f)=0$ for all $C\in J^{\vee}$ and $g_f\in G(\A_f)$.  Hence, by the result of Lemma \ref{lemma-auto-convergence-13.?}, $a_{w_C}(g_f)=0$ for all $C\in J^{\vee}$ and $g_f\in G(\A_f)$.  \\
\indent We conclude that $a_{w}(g_f)=0$ whenever $w\in W_J$ is in the same $H_J(\Q)$ orbit of an element of the form $w_C$ for $C\in J^{\vee}$.  Hence, by the result of Lemma \ref{orbit-lemma}, $a_w(g_f)=0$ for all $g_f\in G(\A_f)$ and non-zero $w\in W_J$ satisfying $\mathrm{rank}(w)<4$.  In case of $J=\mathbb{G}_a^3$, one may now complete the proof by applying Corollary \ref{FE-CUSPIDAL-QMF}.  \\ 
\indent For the case when $J=H_3(C)$ it remains so that if $\varphi_{N_{J},w}=0$ for all $w\in W_J$ satisfying $\mathrm{rank}(w)<4$, then $\varphi$ cuspidal.  This is achieved by a similar argument to the one given in Corollary \ref{FE-CUSPIDAL-QMF}.  Namely,  let $U$ denote the intersection of the unipotent radicals of the maximal parabolic subgroups containing a fixed minimal parabolic subgroup $P_0$.  In this case $G_J$ has a rational root system of type $F_4$, and the description of the $\Lie(U)$ furnished by \cite[Remark 9.4.7]{pollackSiegelWeil} implies that the constant term of $\varphi$ along $U$ takes the form 
$$
\varphi_{U}(g)=\sum_{\tiny w=(\ast, \left(\begin{smallmatrix} 0 &0 &0 \\ 0 &\ast & \ast  \\ 0 &\ast &\ast \end{smallmatrix}\right), \left(\begin{smallmatrix} \ast &0 &0 \\ 0 &0 & 0  \\ 0 &0 &0 \end{smallmatrix}\right),0)\in W_J}\varphi_{N_{J},w}(g). 
$$
Since any element $w\in W_J$ of the form  $w=(\ast, \left(\begin{smallmatrix} 0 &0 &0 \\ 0 &\ast & \ast  \\ 0 &\ast &\ast \end{smallmatrix}\right), \left(\begin{smallmatrix} \ast &0 &0 \\ 0 &0 & 0  \\ 0 &0 &0 \end{smallmatrix}\right),0)$ satisfies $\mathrm{rank}(w)<4$,  $\varphi_U\equiv 0$, which implies that $\varphi$ is cuspidal. 
\end{proof}
\subsection{The Hecke Bound Implies Cuspidality}
We recall that if $\alpha\in 2\Z_{\geq 1}$, then
$$
V_{2,2}=\Q y_{\alpha}+V_{1,2}'
$$
where $V_{1,2}'$ denote the orthogonal complement of $y_{\alpha}$ in $V_{2,2}$. 
Before we establish the main result of this section, we require one preparatory lemma. 
\begin{lemma} 
\label{lemma-orbit-rank3}
\begin{compactenum}
\item[(i)] Suppose $n, \alpha\in \Z$ and $T\in V_{2,2}$. Write $S$ for the orthogonal  projection of $T$ onto the subspace $V_{1,2}'$. If $B=[ny_{\alpha}, T]$ then 
$$
Q(B)=n^2\alpha (S,S). 
$$
\item[(ii)]Assume $B\in V_{2,2}(\Z)^{\oplus 2}$ is a primitive element satisfying $\mathrm{rank}(B)=3$ and $B\succeq 0$. Then there exists $\alpha\in 2\Z_{\geq 1}$, and $n,m,r\in \Z$ such that $B$ is in the same $M_P^{\mathrm{der}}(\Z)$-orbit as 
$$
[y_{\alpha}, -nb_4-mb_{-4}+rb_{-3}]
$$
\end{compactenum}
\begin{proof}
Statement (i) is a direct computation using  the formula 
$$
Q([T_1,T_2])=(T_1,T_1)(T_2,T_2)-(T_1,T_2)^2.
$$
For statement (ii), we apply the primitivity of $B$ and \cite[Appendix Ch. II]{MR2051392} to reduce to the case 
$$
B=[y_{\alpha}, -nb_4-mb_{-4}+rb_{-3}]
$$
where $\alpha, n,m, r\in \Z$.
It remains to show that we may assume $\alpha >0$.  \\
\indent Since $B\succeq 0$, Proposition \ref{Properties-of-beta}(i) implies $\alpha\geq 0$.  Hence, without loss of generality, we may assume $\alpha=0$.  Then via the isomorphism \eqref{eqn-character-lattice isom}, $B$ maps to an element of the form 
$$
w=(r,b,0,1)
$$
where $b=(n,0,m)\in E$. Since $\mathrm{rank}(B)=3$,  
$$
w^{\flat}=(-r^2,rb, 2b^{\#},r)\neq 0.
$$ Therefore, either $r\neq 0$ or $b^{\#}\neq 0$.  Since $Q(B)=r^2+4N_E(b)=0$,  $r\neq 0$ and $b^{\#}=0$ gives a contradiction.  Hence, we may assume $b^{\#}\neq 0$,  in which case $T_2=-nb_4-mb_{-4}+rb_{-3}$ satisfies $(T_2,T_2)\neq 0$.  Therefore, since $B\succeq 0$, Proposition \ref{Properties-of-beta} implies $(T_2,T_2)>0$, which implies $mn>0$.  Therefore, since $\alpha=0$, there exists $k\in \Z$ such that $\gcd(k,m,n)=1$ and 
$$
(T_2+ky_{\alpha}, T_2+ky_{\alpha})>0.
$$
Hence,  $B$ is in the $M_P^{\mathrm{der}}(\Z)$ orbit of an element $[y_{\alpha}, T_2']$ where $T_2'$ is a primitive element in $V_{2,2}(\Z)$ of positive norm.  Hence, there exists $h\in \SO(V_{2,2})(\Z)$ and $\beta\in 2\Z_{\geq 1}$ such that $h\cdot T_2'=y_{\beta}$. Therefore, $B$ is in the $M_P^{\mathrm{der}}(\Z)$ orbit of the element $[y_{\beta}, -h\cdot y_{\alpha}]$,  as required. \\
\end{proof}
\end{lemma}
We are now ready to establish the main result of this section. 
\begin{theorem}
\label{main-theorem-cuspidal-IF-Hecke}
Suppose $\ell\geq 5$ and let $\varphi\in M_{\ell}(1)$ be a weight $\ell$, level one, quaternionic modular form on $G=\Spin(V)$.  Assume that if $B\in V_{2,2}(\Z)^{\oplus 2}$ is primitive and satisfies $B\succ 0$, then 
$$
\Lambda_{\varphi}[B]\ll_{\varphi} Q(B)^{\frac{\ell+1}{2}}. 
$$
Then $\varphi$ is cuspidal. 
\end{theorem}
\begin{proof}
By the result of Proposition \ref{prop-non-vanishing-rank-3}, it suffices to show that if $B\in V_{2,2}(\Z)^{\oplus 2}$ is primitive and $\mathrm{rank}(B)=3$, then $\Lambda_{\varphi}[B]=0$.  Hence, we assume $B\succeq 0$ is primitive of rank $3$.  Then by Lemma \ref{lemma-orbit-rank3}(ii), there exists $\alpha\in \Z_{\geq 1}$, and $n,m,r\in \Z$ such that 
$$
\Lambda_{\varphi}[B]=\Lambda_{\varphi}[y_{\alpha}, -nb_4-mb_{-4}+rb_{-3}]. 
$$  
By \eqref{manipulation-of-a-lifetime-2}, if $S=-nb_4-mb_{-4}-ry_{\alpha}^{\vee}/\alpha$, 
$$
\Lambda_{\varphi}[B]=\overline{A_{\xi^{\varphi}(y_{\alpha})}[S]}.
$$ 
Since $S$ is the orthogonal projection of $T=-nb_4-mb_{-4}+rb_{-3}$ onto the $V_{1,2}'$,  Lemma \ref{lemma-orbit-rank3}(i) implies $(S,S)=0$.  Hence,  to show $\Lambda_{\varphi}[B]=0$, it suffices to prove that $\xi^{\varphi}(y_{\alpha})$ is cuspidal. \\
\indent With a view to to applying Theorem \ref{thm-hecke-bound-IFF-cuspidal-holomorphic}, assume $S'=-n'b_4-m'b_{-4}-r'y_{\alpha}^{\vee}/\alpha\in V_{1,2}'(\Z)_{\geq 0}^{\vee}$ satisfies $(S',S')>0$.  Then by \eqref{manipulation-of-a-lifetime-2}, 
$$
\overline{A_{\xi^{\varphi}(y_{\alpha})}[S']}=\Lambda_{\varphi}[B'],
$$
where $B'=[y_{\alpha}, -n'b_4-m'b_{-4}+r'b_{-3}]$.  Without loss of generality we may assume $B'\succeq 0$. Lemma~\ref{lemma-orbit-rank3}(i) implies $Q(B')=\alpha (S',S')>0$.  Therefore, $B'\succ 0$, and so by assumption 
$$
\overline{A_{\xi^{\varphi}(y_{\alpha})}[S']}=\Lambda_{\varphi}[B']\ll_{\varphi}Q(B')^{\frac{\ell+1}{2}}\ll_{\alpha}(S',S')^{\frac{\ell+1}{2}}. 
$$
Hence, $\xi^{\varphi}(y_{\alpha})$ satisfies the hypothesis of Theorem \ref{thm-hecke-bound-IFF-cuspidal-holomorphic}, and is thus cuspidal.  
\end{proof}
\begin{proof}[Proof of Theorem~\ref{thm-intro-Hecke-IFF-Cuspidal-Quaternionic}] This is a direct consequence of Theorem~\ref{main-theorem-cuspidal-IF-Hecke} and Proposition~\ref{hecke-bound-quaternionic-modular-forms}.
\end{proof}
\printbibliography[title=Bibliography]
\end{document}